\documentclass[a4paper,11pt]{amsart}

\usepackage{amsthm}
\usepackage[arrow,curve,matrix]{xy}
\usepackage{amsmath}  
\usepackage{amssymb}
\usepackage{graphicx}
\usepackage{bbm}
\usepackage{mathrsfs}
\usepackage{amscd}
\usepackage[utf8x]{inputenc}
\usepackage{lscape}
\usepackage{a4wide}

\renewcommand{\mathbbm}{\mathbb}

\allowdisplaybreaks

\DeclareMathOperator{\Stab}{Stab}

\DeclareMathOperator{\pr}{pr}
\DeclareMathOperator{\pt}{pt}

\DeclareMathOperator{\id}{id}
\DeclareMathOperator{\sgn}{sgn}
\DeclareMathOperator{\Hom}{Hom}

\DeclareMathOperator{\sHom}{\mathcal{H}\textit{om}}
\DeclareMathOperator{\Ext}{Ext}

\DeclareMathOperator{\Coh}{Coh}

\DeclareMathOperator{\sExt}{\mathcal{E}\textit{xt}}
\DeclareMathOperator{\D}{D}
\DeclareMathOperator{\K}{K}
\DeclareMathOperator{\Ho}{H}
\DeclareMathOperator{\Coho}{\mathcal H}

\DeclareMathOperator{\Aut}{Aut}

\DeclareMathOperator{\codim}{codim}

\DeclareMathOperator{\rank}{rank}
\DeclareMathOperator{\supp}{supp}
\DeclareMathOperator{\Pic}{Pic}

\DeclareMathOperator{\triv}{\mathsf{triv}}

\DeclareMathOperator{\im}{im}

\DeclareMathOperator{\cone}{cone}
\DeclareMathOperator{\coker}{coker}

\DeclareMathOperator{\SL}{SL}
\DeclareMathOperator{\Index}{\mathsf{Index}}

\DeclareMathOperator{\orb}{\mathsf{orb}}

\DeclareMathOperator{\Inf}{\mathsf{Inf}}
\DeclareMathOperator{\Res}{\mathsf{Res}}

\DeclareMathOperator{\FM}{\mathsf{FM}}

\DeclareMathOperator{\MM}{\mathsf{M}}

\newcommand{\sym}{\mathfrak S}
\newcommand{\Alt}{\mathfrak A}
\newcommand{\C}{\mathbbm C}
\newcommand{\N}{\mathbbm N}
\newcommand{\Q}{\mathbbm Q}

\newcommand{\Z}{\mathbbm Z}
\newcommand{\cA}{\mathcal A}
\newcommand{\cB}{\mathcal B}

\newcommand{\cF}{\mathcal F}

\newcommand{\cE}{\mathcal E}

\newcommand{\cC}{\mathcal C}

\newcommand{\cG}{\mathcal G}
\newcommand{\cH}{\mathcal H}

\newcommand{\cS}{\mathcal S}
\newcommand{\cP}{\mathcal P}
\newcommand{\cQ}{\mathcal Q}
\newcommand{\cT}{\mathcal T}
\newcommand{\cL}{\mathcal L}
\newcommand{\cK}{\mathcal K}
\newcommand{\cR}{\mathcal R}

\newcommand{\m}{\mathfrak m}
\newcommand{\alt}{\mathfrak a}

\newcommand{\reg}{\mathcal O}
\newcommand{\I}{\mathcal I}

\newcommand{\eps}{\varepsilon}

\renewcommand{\P}{\mathbbm P}
\renewcommand{\theta}{\vartheta}
\renewcommand{\rho}{\varrho}
\renewcommand{\phi}{\varphi}
\renewcommand{\_}{\underline{\,\,\,\,}}

\newtheorem{theorem}{Theorem}[section]
\newtheorem{prop}[theorem]{Proposition}
\newtheorem{lemma}[theorem]{Lemma}
\newtheorem{cor}[theorem]{Corollary}

\newtheorem*{propA}{Proposition A}
\newtheorem*{corB}{Corollary B}
\newtheorem*{theoremC}{Theorem C}

\newtheorem*{propAprime}{Proposition A'}
\newtheorem*{corBprime}{Corollary B'}
\newtheorem*{theoremCprime}{Theorem C'}

\theoremstyle{definition}

\newtheorem{remark}[theorem]{Remark}
\newtheorem{conv}[theorem]{Convention}
\newtheorem{conj}[theorem]{Conjecture}

\usepackage{mathtools}

\begin{document}
\title{$\P$-functor versions of the Nakajima operators}
\author{Andreas Krug}
\address{Mathematisches Institut, Universit\"at Bonn\\ Endenicher Allee 60\\ 53115 Bonn, Germany}
\email{akrug@math.uni-bonn.de}
\begin{abstract}
For a smooth quasi-projective surface $X$ we construct a series of $\P^{n-1}$-functors 
$H_{\ell,n}\colon \D^b(X\times X^{[\ell]})\to \D^b(X^{[n+\ell]})$ for $n> \ell$ and $n>1$
using the derived McKay correspondence. They can be considered as analogues of the Nakajima operators. The functors also restrict to $\P^{n-1}$-functors on the generalised Kummer varieties. 
We also study the induced autoequivalences and obtain, for example, a universal braid relation in the groups of derived autoequivalences of Hilbert squares of K3 surfaces and Kummer fourfolds. 
\end{abstract}
\maketitle
\section{Introduction}
A central result in the theory of Hilbert schemes of points on surfaces is the identification of their cohomology 
with the Fock space representation of the Heisenberg algebra  
by means of the \textit{Nakajima operators}
$q_{\ell,n}\colon \Ho^*(X\times X^{[\ell]},\Q)\to \Ho^*(X^{[n+\ell]},\Q)$; see \cite{Nak} and \cite{Groj}.
They are induced by the correspondences 
\begin{align}\label{Zln}X^{[\ell]}\times X\times X^{[n+\ell]}\supset Z^{\ell,n}:=\{([\xi],x,[\xi'])\mid \xi\subset \xi'\,,\,\text{$\xi$ and $\xi'$ only differ in $x$}\}\,.\end{align}
More recently, autoequivalences of the (bounded) derived categories $\D^b(X^{[n]})$ of the Hilbert schemes were intensively studied; see \cite{Plo}, \cite{Add}, \cite{PS}, \cite{Mea}, \cite{Kru3}, \cite{CLS}, \cite{KSos}. 
In particular, Addington \cite{Add} defined the notion of a \textit{$\P^n$-functor} as a Fourier--Mukai transform $F\colon \D^b(M)\to \D^b(N)$ between derived categories of varieties with right-adjoint $F^R\colon \D^b(N)\to \D^b(M)$ and the main property
that
\[F^RF\cong \id \oplus D\oplus D^2\oplus \dots\oplus D^n\] for some autoequivalence $D\colon \D^b(M)\to \D^b(M)$ which is then called the \textit{$\P$-cotwist} of $F$. 
Every $\P^n$-functor $F$ gives an \textit{induced $\P$-twist} $P_F\colon \D^b(N)\to \D^b(N)$ which is an autoequivalence.

The main example in \cite{Add} is the $\P^{n-1}$-functor $F_n=\FM_{\I_\Xi}\colon\D^b(X)\to\D^b(X^{[n]})$ given by the Fourier--Mukai transform along the ideal sheaf of the universal family $\Xi\subset X\times X^{[n]}$ for $X$ a K3 surface and $n\ge 2$. Similarly,
 for $X=A$ an abelian surface and $\hat \Xi\subset A\times K_nA$ the universal family of the generalised Kummer variety $K_nA\subset A^{[n+1]}$ the functor $\hat F_n=\FM_{\I_{\hat \Xi}}\colon\D^b(A)\to \D^b(K_nA)$ is a $\P^{n-1}$-functor for $n\ge 2$ and a 
$\P^1$-functor for $n=1$; see \cite{Mea} and \cite{KMea}. We will refer to these examples of $\P$-functors as the \textit{universal ideal functors}.

An important tool for the investigation of the derived categories of the Hilbert schemes of points on surfaces is the \textit{Bridgeland--King--Reid--Haiman} equivalence (also known as \textit{derived McKay correspondence}; see \cite{BKR} and \cite{Hai})
\[\Phi=\Phi_n\colon \D^b(X^{[n]})\xrightarrow\cong \D^b_{\sym_n}(X^n)\,.\]
This provides an identification of $\D^b(X^{[n]})$ with the derived category of equivariant coherent sheaves on the product $X^n$, i.e.\ of coherent sheaves equipped with a linearisation by the symmetric group $\sym_n$.  

In \cite{Kru3} it was shown that for every smooth surface $X$ and $n\ge 2$ there is a $\P^{n-1}$-functor $H_{0,n}:=H\colon \D^b(X)\to \D^b_{\sym_n}(X^n)$ given by the composition $\D^b(X)\xrightarrow{\triv}\D^b_{\sym_n}(X)\xrightarrow{\delta_*} \D^b_{\sym_n}(X^n)$ 
where the first functor equips every object with the trivial $\sym_n$-linearisation and the second is the push-forward along the embedding of the small diagonal. Its $\P$-cotwist is $S_X^{-1}:=(\_)\otimes \omega_X^{-1}[-2]$.
 The FM kernel of $H$ is given by $\reg_D$ with $X\cong D=\{x=x_1=\dots=x_n\}\subset X\times X^n$.
Under the BKRH equivalence, this functor corresponds to a $\P^n$-functor $\Phi^{-1}\circ H\colon \D^b(X)\to \D^b(X^{[n]})$ whose kernel is supported on $Z^{0,n}=\bigl\{(x,[\xi'])\mid \supp(\xi')=\{x\}\bigr \}$. Thus, one can regard $H=H_{0,n}$ as a lift of the
 Nakajima operator $q_{0,n}\colon \Ho^*(X,\Q)\to \Ho^*(X^{[n]},\Q)$.   
\subsection{Main results}
The question is whether the other Nakajima operators also have analogues in the form of $\P^{n-1}$-functors 
\[H_{\ell,n}\colon \D^b_{\sym_\ell}(X\times X^\ell)\cong \D^b(X\times X^{[\ell]})\to \D^b(X^{[n+\ell]})\cong \D^b_{\sym_{n+\ell}}(X^{n+\ell})\,.\]
A first natural guess for a generalisation of $H_{0,n}$ would be the functor
\begin{align}\label{Hcompo}H_{\ell,n}^0\colon \D^b_{\sym_\ell}(X\times X^\ell)\xrightarrow{\triv}\D^b_{\sym_n\times\sym_\ell}(X\times X^\ell)\xrightarrow{\delta_{[n]*}}\D^b_{\sym_n\times \sym_\ell}(X^n\times X^\ell)\xrightarrow{\Inf}\D^b_{\sym_{n+\ell}}(X^{n+\ell})
\,;\end{align}
see Section \ref{equiFM} for details on the inflation functor $\Inf$ and its adjoint $\Res$. Here, for $J\subset [n+\ell]=\{1,\dots,n+\ell\}$ of cardinality $|J|=n$ the morphism $\delta_J$ denotes the closed embedding of the \textit{partial diagonal}
\[X\times X^\ell\cong\Delta_J=\bigl\{(y_1,\dots,y_{n+\ell})\mid y_a=y_b \text{ for all $a,b\in J$}  \bigr\}\subset X^{n+\ell}\,.\]
The functor is given on objects by   
\begin{align}\label{Infsum}H_{\ell,n}^0(E)=\bigoplus_{J\subset [n+\ell]\,,\,|J|=n}\delta_{J*}E \quad \text{for $E\in \D^b_{\sym_\ell}(X\times X^\ell)$}\,.\end{align}
The first part of the composition (\ref{Hcompo}) namely $\delta_{[n]*}\circ\triv$ can be rewritten as $H_{0,n}\boxtimes \id_{X^\ell}$. Since $H_{0,n}$ is a $\P^{n-1}$-functor so is $\delta_{[n]*}\circ \triv$ and 
\[(\delta_{[n]*}\circ \triv)^R(\delta_{[n]*}\circ \triv)\cong \bar S_X^{-[0,n-1]}:=\id \oplus \bar S_X^{-1}\oplus\dots\oplus \bar S_X^{-(n-1)}\]
where we write $\bar S_X:=(\_)\otimes (\omega_X\boxtimes \reg_{X^\ell})[2]$ even in the case that $X$ is not projective.
But for the computation of $H_{\ell,n}^{0R}\circ H_{\ell,n}$ also the other summands of (\ref{Infsum}) besides $\delta_{[n]*}E$ have to be taken into account which yields
\begin{align}\label{error}
H_{\ell,n}^{0R} H_{\ell,n}^0(E)\cong \bar S_X^{-[0,n-1]}(E)\oplus\bigl(\text{terms supported on partial diagonals of $X\times X^\ell$}\bigr)\,. 
\end{align}
The approach is to adapt $H_{\ell,n}^0$ slightly in order to erase the error term in (\ref{error}). We succeed in doing so by replacing $H_{\ell,n}^0$ by a complex of functors
\[H_{\ell,n}:=\bigl(0\to H_{\ell,n}^0\to H_{\ell,n}^1\to \dots\to H_{\ell,n}^\ell\to 0)\,.\]
More concretely, this means the following. 
First, we actually have to consider the functor $H_{\ell,n}^0:=\Inf\circ \delta_{[n]*}\circ \MM_{\alt_n}\circ \triv$ instead. That means that the definition differs from (\ref{Hcompo}) by
 $\MM_{\alt_n}:=(\_)\otimes_\C \alt_n\colon \D^b_{\sym_n\times \sym_\ell}(X\times X^\ell)\to \D^b_{\sym_n\times \sym_\ell}(X\times X^\ell)$, 
where the \textit{alternating representation} $\alt_n$ is the non-trivial character of $\sym_n$. 
The kernel of the FM transform $H_{\ell,n}^0$ is then given by 
\[\cP_{\ell,n}^0=\bigoplus_{J\subset[n+\ell]\,,\,|J|=n}\Gamma_{\delta_J}\otimes \alt_J\in\D^b_{\sym_\ell\times\sym_{n+\ell}}(X\times X^\ell\times X^{n+\ell})\,.\]
 In Section \ref{cP} we construct a complex 
\[\cP_{\ell,n}=(0\to \cP_{\ell,n}^0\to \dots\to \cP_{\ell,n}^\ell\to 0)\]
whose terms $\cP_{\ell,n}^i$ are given by direct sums of structure sheaves of certain subvarieties of the graphs $\Gamma_{\delta_J}$. Then we set $H_{\ell,n}=\FM_{\cP_{\ell,n}}$. The definition of  
$\cP_{\ell,n}\in \D^b_{\sym_\ell\times\sym_{n+\ell}}(X\times X^\ell\times X^{n+\ell})$ makes sense if $X$ is a variety of arbitrary dimension and our first result is
\begin{propA}
Let $X=C$ be a smooth curve and $n> \max\{\ell,1\}$.
\begin{enumerate}
 \item We have $H_{\ell,n}^R\circ H_{\ell,n}=\id$ which means that $H_{\ell,n}$ is fully faithful.
\item Let $\ell',n'$ be positive integers such that $n'+\ell'=n+\ell$ and $\ell'>\ell$. Then $H_{\ell',n'}^R\circ H_{\ell,n}=0$.
\end{enumerate}
\end{propA}
\begin{corB}
 Let $m\ge 2$. For $m$ even we set $r=\frac m2-1$ and for $m$ odd we set $r=\frac{m-1}2$. For every smooth curve $C$ there is a semi-orthogonal decomposition 
\[\D^b_{\sym_m}(C^m)=\langle \cA_{0,m},\cA_{1,m-1},\dots,\cA_{r,m-r},\cB\rangle\]
where $\cA_{\ell,m-\ell}=H_{\ell,m-\ell}(\D^b_{\sym_\ell}(C\times C^\ell))$. 
\end{corB}
Note that for $C=E$ an elliptic curve the canonical bundle of the product $E^m$ is trivial. But as a $\sym_n$-bundle it is given by $\reg_{E^n}\otimes \alt_n$ which means that the canonical bundle of the quotient stack $[E^n/\sym_n]$ is non-trivial. 
Otherwise, $\D^b([E^m/\sym_m])\cong \D^b_{\sym_m}(E^m)$ could not allow a non-trivial semi-orthogonal decomposition.

We will see in Section \ref{curveauto} that the fully faithful functors $H_{\ell,m-\ell}$ also induce autoequivalences of $\D^b_{\sym_m}(C^m)$. Nevertheless, the case of most interest is that of a smooth quasi-projective surface due to the BKRH equivalence. 
So Proposition A can be regarded as a warm-up for 
\begin{theoremC}
Let $X$ be a smooth surface and $n>\max\{\ell,1\}$. Then $H_{\ell,n}$ is a $\P^{n-1}$-functor with $\P$-cotwist $\bar S_X$. In particular, $H_{\ell,n}^R\circ H_{\ell,n}\cong \bar S_X^{-[0,n-1]}$. 
\end{theoremC}
In the case that $X=A$ is an abelian variety, there is also the following variant of the above construction. For $m\ge 2$ we consider the subvariety
\[A^{m-1}\cong N_{m-1}A:=\bigl\{(a_1,\dots,a_m)\mid a_1+\dots+a_m=0  \bigr\}\subset A^m \,.\]
It is invariant under the $\sym_m$-action on $A^m$ and we call the quotient stack $[N_{m-1}A/\sym_m]$ the \textit{generalised Kummer stack}. The reason for the name is that in the case that $A$ is an abelian surface there is a variant of the BKRH equivalence as
 an equivalence 
\[\hat \Phi\colon \D^b(K_{m-1}A)\xrightarrow\cong \D^b_{\sym_m}(N_{m-1}A)\cong \D^b([N_{m-1}A/\sym_m])\]
with the derived category of the generalised Kummer variety; see \cite{Nam} or \cite{Mea}. We also consider the $\sym_\ell$-invariant subvariety
\[M_{\ell,n}:=\bigl\{(a,a_1,\dots,a_\ell)\mid n\cdot a+a_1+\dots+a_\ell=0\bigl\}\subset A\times A^\ell\,.\]
Then for all $n\ge 2$ the functor $H_{\ell,n}\colon \D^b_{\sym_\ell}(A\times A^\ell)\to \D^b_{\sym_{n+\ell}}(A^{n+\ell})$ restricts to 
a functor $\hat H_{\ell,n}\colon \D^b_{\sym_\ell}(M_{\ell,n}A)\to \D^b_{\sym_{n+\ell}}(N_{n+\ell-1}A)$; see Section \ref{Kummercase} for details. 

For $\ell=0$ we have $M_{0,n}A=A_n\subset A$ where $A_n$ denotes the set of $n$-torsion points.
The functor $\hat H_{0,n}$ is given by sending the skyscraper sheaf $\C(a)$ of $a\in A_n$ to the object $\C(a,\dots,a)\otimes \alt_n\in \D^b_{\sym_n}(N_{n-1}A)$. As shown in \cite{Kru3}, for $A$ an abelian surface the objects $\C(a,\dots,a)\otimes \alt_n$ 
are $\P^{n-1}$-objects in the sense of \cite{HT}. Similarly, for $A=E$ an elliptic curve they are exceptional. 

In contrast, for $\ell\ge 1$ we have a (not $\sym_\ell$-equivariant) isomorphism $M_{\ell,n}A\cong A^\ell$ so that the domain category of the functor $\hat H_{\ell,n}$ is indecomposable.
\begin{propAprime}
Let $A=E$ be an elliptic curve and $n> \max\{\ell,1\}$.
\begin{enumerate}
 \item We have $\hat H_{\ell,n}^R\circ \hat H_{\ell,n}=\id$ which means that $\hat H_{\ell,n}$ is fully faithful.
\item Let $\ell',n'$ be positive integers such that $n'+\ell'=n+\ell$ and $\ell'>\ell$. Then $\hat H_{\ell',n'}^R\circ \hat H_{\ell,n}=0$.
\end{enumerate}
\end{propAprime}
\begin{corBprime}
 For every elliptic curve $E$ there is a semi-orthogonal decomposition 
\[\D^b_{\sym_m}(N_{m-1}E)=\langle \C(a,\dots,a)\otimes \alt_n\mid a\in A_m\,,\hat \cA_{1,m-1},\dots,\hat \cA_{r,m-r},\hat \cB\rangle\]
where $\hat \cA_{\ell,m-\ell}=\hat H_{\ell,m-\ell}(\D^b_{\sym_\ell}(M_{\ell,m-\ell}A))$. 
\end{corBprime}
\begin{theoremCprime}\label{thmC'}
Let $A$ be an abelian surface and $n>\max\{\ell,1\}$. Then $\hat H_{\ell,n}$ is a $\P^{n-1}$-functor with $\P$-cotwist $[-2]$. In particular, $H_{\ell,n}^R\circ H_{\ell,n}\cong \id\oplus[-2]\oplus\dots\oplus[-2(n-1)]$. 
\end{theoremCprime}
%
\subsection{Structure of the proof}
The approach is to compute the compositions $H_{\ell,n}^{iR}\circ H_{\ell,n}^j$ in order to deduce formulae for 
$H_{\ell,n}^{R}\circ H_{\ell,n}^j$ and finally for $H_{\ell,n}^{R}\circ H_{\ell,n}$. 

The general proof may appear complicated due to the sheer number of direct summands occurring.
We try to provide intuition by performing some calculations of the functor compositions for $\ell=1,2$ in Sections \ref{l1} -- \ref{ln2}. However, we will not determine the induced maps between the $H_{\ell,n}^{iR}H^j_{\ell,n}$ in these subsections so that 
the arguments for the final computations of $H_{\ell,n}^RH_{\ell,n}$ will be a bit vague as explained in Section \ref{approach}. 

In Section \ref{mainproof} we give rigorous proofs for general $\ell$ and $n$ by doing the computations on the level of the FM kernels.        

For $m\ge 2$ let $\rho_m$ be the standard representation of the symmetric group $\sym_m$.
One can say that the reason for the different behaviour of $H_{\ell,n}$ in the curve and in the surface case lies in the difference of the $\sym_m$-representations $\wedge^i\rho_m$ and $\wedge^i(\rho_m^{\oplus 2})$; compare Lemma \ref{Tinva}. 
The former is always an irreducible representation while the latter has a one-dimensional subspace of invariants for $0\le i\le 2(m-1)$ even. For $d\ge 3$ the subspace of invariants of $\wedge^i(\rho^{\oplus d})$ is in general higher-dimensional which 
explains that the shape of $H_{\ell,n}^R\circ H_{\ell,n}$ is not that nice if $\dim X\ge 3$; see Remark \ref{higherdiminva}.
\subsection{Similarities to the Nakajima operators}\label{simNaka}
Let $X$ be a smooth quasi-projective surface. 
In order to justify the title of our paper, we will explain some similarities between the $\P^{n-1}$-functors $H_{\ell,n}$ and the Nakajima operators $q_{\ell,n}$.
 
As indicated above, the $\P^{n-1}$-functors $H_{\ell,n}$ correspond under the BKRH equivalence to $\P^{n-1}$ functors
$\Phi_{n+\ell}^{-1}\circ H_{\ell,n}\circ(\id\boxtimes \Phi_\ell)\colon \D^b(X\times X^{[\ell]})\to \D^b(X^{[n+\ell]})$
which we will often again denote by $H_{\ell,n}$. 
The FM kernel of the BKRH equivalence $\Phi_m$ is the structure sheaf $\reg_{I^mX}$ of the \textit{isospectral Hilbert scheme} 
\[I^mX=\bigl\{([\xi],x_1,\dots,x_m)\mid \mu([\xi])=x_1+\dots+x_m  \bigr\}\subset X^{[m]}\times X^m\]
where $\mu\colon X^{[n]}\to S^mX=X^m/\sym_m$ denotes the Hilbert--Chow morphism and points in the symmetric product are written as formal sums. 
One can deduce that the FM kernel of $\Phi_{n+\ell}^{-1}\circ H_{\ell,n}\circ(\id\boxtimes \Phi_\ell)$ is supported on $Z^{\ell,n}$, the correspondence defining the Nakajima operator $q_{\ell,n}$. Clearly, it would be desirable to have a more concrete description of the FM kernel but for the time being we are unable to provide one.

In the case that $X$ is projective with trivial canonical bundle $\omega_X\cong \reg_X$, the functors $H_{\ell,n}$ also fulfil some of the Heisenberg relations on the level of the Grothendieck groups. For $\alpha\in \Ho^*(X,\Q)$ there is by the 
K\"unneth formula the map
\[i_\alpha\colon \Ho^*(X^{[\ell]},\Q)\to \Ho^*(X\times X^{[\ell]},\Q)\cong \Ho^*(X,\Q)\otimes \Ho^*(X^{[\ell]},\Q)\quad,\quad i_{\alpha}(\beta)=\alpha\otimes \beta\,.\]
The operators $q_{\ell,n}(\alpha):=q_{\ell,n}\circ i_\alpha\colon\Ho^*(X^{[\ell]},\Q)\to \Ho^*(X^{[n+\ell]},\Q)$ are again called Nakajima operators. Furthermore, $q_{\ell,-n}(\alpha)\colon\Ho^*(X^{[n+\ell]},\Q)\to \Ho^*(X^{[\ell]},\Q)$ is defined as the 
adjoint of $q_{\ell,n}(\alpha)$ with respect to the intersection pairing.
One usually considers all of these operators for varying values of $\ell$ together as operators on $\mathbbm H:=\oplus_{\ell\ge 0}\Ho^*(X^{[\ell]},\Q)$ by setting
\begin{align*}q_n{(\alpha)}:=\oplus_\ell\, q_{\ell,n}(\alpha)\colon \oplus_{\ell}\Ho^*(X^{[\ell]},\Q)&\to\oplus_\ell\Ho^*(X^{[n+\ell]},\Q)\,,\\
q_{-n}{(\alpha)}:=\oplus_\ell\, q_{\ell,-n}(\alpha)\colon \oplus_{\ell}\Ho^*(X^{[n+\ell]},\Q)&\to\oplus_\ell\Ho^*(X^{[\ell]},\Q)\,.
\end{align*}  
Then as shown in \cite{Nak} the commutator relations between these operators are given by
\begin{align}\label{Nakarel}
 [q_{n}(\alpha),q_{n'}(\beta)]=n\cdot \delta_{n,-n'}\langle\alpha,\beta\rangle\cdot\id_{\mathbbm H}\,.
\end{align}
Taking $n=-n'$ and considering the degree $\ell$ piece of formula (\ref{Nakarel}) for $\ell< n$ we get
\begin{align}\label{NGrel}
 q_{\ell,-n}(\alpha)\circ q_{\ell,n}(\beta)=n\cdot\langle \alpha,\beta\rangle\cdot\id\colon \Ho^*(X^{[\ell]},\Q)\to \Ho^*(X^{[\ell]},\Q)\,.
\end{align}
Let $X\xleftarrow q X\times X^{[\ell]} \xrightarrow p X^{[\ell]}$ be the projections. Now, for $E\in \D^b(X)$ we consider the functor
\[I_E:=q^* E\otimes p^*(\_)\colon \D^b(X^{[\ell]})\to \D^b(X\times X^{[\ell]})\quad, \quad I_E(F)=E\boxtimes F\,.\]
Its right adjoint is $I_E^R=p_*(q^*E^\vee\otimes(\_))$. We set $H_{\ell,n}(E):=H_{\ell,n}\circ I_E\colon \D^b(X^{[\ell]})\to \D^b(X^{[n+\ell]})$. For $E,F\in \D^b(X)$ we get by Theorem C in the case that $\omega_X$ is trivial  
\begin{align*}
H_{\ell,n}(E)^R\circ H_{\ell,n}(F)&\cong I_E^R\circ H_{\ell,n}^R\circ H_{\ell,n}\circ I_F\\
&\cong I_E^R\circ I_F([0]\oplus[-2]\oplus\dots\oplus[-2(n-1)])\\ 
&\cong(\_)\otimes_\C\Ext^*(E,F)([0]\oplus[-2]\oplus\dots\oplus[-2(n-1)])\,.
\end{align*}
On the level of the Grothendieck group this gives
\begin{align}
 H_{\ell,n}(E)^R\circ H_{\ell,n}(F)=n\cdot\chi(E,F)\cdot \id\colon \K(X^{[\ell]})\to \K(X^{[\ell]})
\end{align}
which fits nicely with (\ref{NGrel}). However, as we will see in Section \ref{nonorth}, not all the compositions $H_{\ell,n}(E)^R\circ H_{\ell',n'}(F)$ with $n\neq n'$ fit that well with (\ref{Nakarel}). Thus, it seems like the collection of 
the $H_{\ell,n}$ does not give rise to a categorified action of the Heisenberg algebra.

In \cite{CL} such a categorified action was constructed in the case that $X$ is a minimal resolution of a Kleinian singularity.
In Section \ref{CL} we will compare the construction of \cite{CL} to ours.
\subsection{The induced autoequivalences}
In the Sections \ref{Hilbauto} -- \ref{braidauto} we study the autoequivalences of the derived categories of Hilbert schemes of points on surfaces and generalised Kummer varieties induced by the $\P$-functors of the Theorems C and C', to which we will refer as 
the \textit{Nakajima $\P$-functors}. We will see that the twists are rather independent from each other and from the subgroup of standard autoequivalences; see Proposition \ref{abel} and (\ref{Kummerabel}). 
For $X$ a K3 surface, the universal ideal $\P$-functor can be truncated in a certain way to give another $\P$-functor $\D^b(X)\to \D^b(X^{[n]})$; see \cite[Section 5]{KSos}. The existence of these truncated universal ideal functors is in some sense explained by 
the Nakajima functors; see Section \ref{trunca}. In Section \ref{braidauto} we show the existence of a universal braid relation in the groups of derived autoequivalences of Hilbert squares of K3 surfaces and Kummer fourfolds.
In the final Section \ref{conjectures} we make a conjecture about certain cases in which we expect the twists along the Nakajima $\P$-functors to generate the full group of derived autoequivalences and give an idea which kind of autoequivalences might 
still wait to be constructed.  

\noindent
\textbf{Convention.} We will work over the complex numbers throughout though many parts remain true over more general ground fields.
\smallskip

\noindent
\textbf{Acknowledgements.} The author thanks Nicolas Addington, Will Donovan, Daniel Huybrechts, Ciaran Meachan, David Ploog, and Pawel Sosna for helpful comments and discussions. This work was supported by the SFB/TR 45 of the DFG (German Research Foundation).

\tableofcontents
\section{Definition of the functors}
\subsection{Equivariant Fourier--Mukai transforms}\label{equiFM}
For details on equivariant derived categories and Fourier--Mukai transforms we refer to \cite[Section 4]{BKR} and \cite{Plo}.
Let $G$ be a finite group acting on a variety $M$. Then we denote by $\D^b_G(M):=\D^b(\Coh_G(M))$ the bounded derived category of the category $\Coh_G(M)$ of coherent $G$-equivariant sheaves.
Let $U\subset G$ be a subgroup. Then there is the \textit{forgetful} or \textit{restriction} functor $\Res_G^U\colon \D^b_G(M)\to \D^b_U(M)$. It has the \textit{inflation} functor $\Inf_U^G\colon \D^b_U(M)\to \D^b_G(M)$ as a left and right adjoint.
This functor is given by $\Inf_U^G(E)=\oplus_{[g]\in U\setminus G}g^* E$ with the $G$-linearisation  given as a combination of the $U$-linearisation of $E$ and the permutation of the direct summands.
In the following we will often simply write $\Res$ and $\Inf$ for these functors when the groups $G$ and $U$ should be clear from the context. 
 In the case that $G$ acts trivially on $M$ there is also the functor $\triv\colon \D^b(M)\to \D^b_G(M)$ which equips every object with the trivial $G$-linearisation. Its left and right adjoint is given by the functor of invariants $(\_)^G\colon \D^b_G(M)\to \D^b(M)$.

Let $G'$ be a second finite group acting on $M'$. Then every object $\cP\in \D^b_{G\times G'}(M\times M')$ induces the \textit{equivariant Fourier--Mukai transform} 
\[
 \FM_{\cP}:=\bigl[\pr_{M'*}\bigl(\pr_M^*(\_)\otimes \cP)\bigr]^{G\times 1}\colon \D^b_G(M)\to \D^b_{G'}(M')\,.
\]
For example, in the case that $M=M'$ and $G$ acts trivially, the functor $\triv\colon \D^b(M)\to \D^b_G(M)$ is the FM transform along $\reg_{\Delta}\in \D^b_{1\times G}(M\times M)$ and $(\_)^G\colon \D^b_G(M)\to \D^b(M)$ is the FM transform along $\reg_\Delta\in \D^b_{G\times 1}(M\times M)$. 
Note that, for an arbitrary action of $G$ on $M$, the sheaf $\reg_\Delta$ on $M\times M$ has a canonical $G_\Delta$-linearisation where $G\cong G_\Delta=\{(g,g)\}\subset G\times G$. The identity functor $\id\colon\D^b_G(M)\to \D^b_G(M)$ is the 
FM transform along $\Inf_{G_\Delta}^{G\times G}\reg_\Delta\cong \oplus_{g\in G}\reg_{\Gamma_g}$ where $\Gamma_g$ denotes the graph of the action $g\colon X\to X$. More generally, for $E\in \D^b_G(M)$ the tensor product functor $(\_)\otimes E\colon \D^b(M)\to \D^b(M)$ is given by 
 $\Inf_{G_\Delta}^{G\times G}\delta_*E\cong \oplus_{g\in G}(1\times g)_*E$ where $\delta=(1\times 1)\colon M\to M\times M$ is the diagonal embedding. The restriction $\Res_G^U$ and the inflation $\Inf_U^G$ are also FM transforms with
 kernels $\Inf_{U_\Delta}^{G\times U}\reg_\Delta$ and $\Inf_{U_\Delta}^{U\times G}\reg_\Delta$, respectively. 

If $G''$ is a third finite group acting on $M''$ and $\cQ\in \D^b_{G'\times G''}(M'\times M'')$, we have $\FM_{\cQ}\circ \FM_{\cP}=\FM_{\cQ\star\cP}$, where 
\begin{align}\label{conprod}
 \cQ\star \cP=\bigl[\pr_{M\times M'*}\bigl(\pr_{M'\times M''}^*\cQ\otimes \pr_{M\times M'}^*\cP \bigr)\bigr]^{1\times G'\times 1}\,\in \D^b_{G\times G''}(M\times M'')
\end{align}
is the \textit{equivariant convolution product}. 
\subsection{$\P$-functors} 
Let $G$ and $H$ be finite groups acting on varieties $M$ and $N$.
Following \cite{Add}, a \textit{$\P^n$-functor} is an (equivariant) FM transform $F\colon \D^b_G(M)\to \D^b_H(N)$ with right and left adjoints $F^R,F^L\colon \D^b_H(N)\to \D^b_G(M)$ such that
\begin{enumerate}
 \item There is an autoequivalence $D_F=D$ of $\D^b_G(M)$, called the \textit{$\P$-cotwist} of $F$,  such that 
\[F^R\circ F\simeq \id\oplus D\oplus D^2\oplus\dots \oplus D^n.\]
\item
Let $\eps\colon F\circ F^R\to \id$ be the counit of the adjunction.
 The map 
\[D\circ F^R\circ F\hookrightarrow F^R\circ F\circ F^R\circ F\xrightarrow{F^R\eps F} F^R\circ F\,,\]
when written in the components
\begin{align*}D\oplus D^2\oplus\dots\oplus D^n\oplus D^{n+1}\to \id\oplus D\oplus D^2\oplus\dots\oplus D^n,\end{align*}
is of the form
\begin{align*}\begin{pmatrix}
  * & * &\cdots &*&*\\
1&*&\cdots&*&*\\
0&1&\cdots&*&*\\
\vdots&\vdots&\ddots&\vdots&\vdots\\
0&0&\cdots&1&* 
  \end{pmatrix}.
  \end{align*}
\item $F^R\simeq D^n\circ F^L$. If $\D^b_G(M)$ as well as $\D^b_H(N)$ have Serre functors, this is equivalent to $S_N \circ F\circ D^n\cong F\circ S_M$.
\end{enumerate}
A $\P^1$-functor is the same as a split spherical functor. A general \textit{spherical functor} is a FM transform $F$ such that $C:=\cone(\id\xrightarrow{\eta} F^R\circ F)$ is an autoequivalence and $F^R\cong C\circ F^L$. Here $\eta$ is the unit of the adjunction. 
The cone is well defined as a FM transform, since the natural transform $\eta$ is induced by a morphism between the FM kernels; see \cite{AL}. This is the reason that we restrict ourself in the definition of spherical and $\P$-functors to functors between derived categories of coherent sheaves. More generally, one can work with dg-enhanced triangulated categories; see \cite{ALdg}. 
\subsection{Notations and conventions}\label{conventions}
\begin{enumerate}
\item For $E$ a complex we denote its $q$-th cohomology by $\Coho^q(E)$. Furthermore, we set $\Coho^*(E):=\oplus_{i\in \Z}\Coho^i(E)[-i]$.
\item For $L$ a $G$-equivariant line bundle on a variety $M$ we denote the tensor product functor by $\MM_L:=(\_)\otimes L\colon \D^b_G(M)\to \D^b_G(M)$. If we write $\reg_M$ as a $G$-sheaf we mean the structural sheaf equipped with its canonical linearisation.
\item The \textit{alternating representation} $\alt_n$ of the symmetric group $\sym_n$ is the one-dimensional representation on which $\sigma\in \sym_n$ acts by multiplication by $\sgn(\sigma)$. If $\sym_n$ acts on a variety $M$, we set
$\MM_{\alt_n}:=\MM_{\reg_M\otimes_\C \alt_n}\colon \D^b_{\sym_n}(M)\to\D^b_{\sym_n}(M)$. 
 \item For $u\le v$ positive integers, we use the notations $[u,v]:=\{u, u+1,\dots , v\}\subset \N$ and $[v]:=[1,v]=\{1,\dots,v\}\subset \N$.
\item\label{empty} We set $[0]:=\emptyset$.
\item For $A,B\subset\N$ two finite subsets of the same cardinality $|A|=|B|$ we let $e\colon A\to B$ denote the unique 
strictly increasing bijection.
\end{enumerate}
\subsection{The Fourier--Mukai kernel}\label{cP}
 Let $X$ be a smooth variety of arbitrary dimension $d=\dim X$. In the following we will construct the functors $H_{\ell,n}\colon \D^b_{\sym_\ell}(X\times X^\ell)\to \D^b_{\sym_{n+\ell}}(X^{n+\ell})$ for $n,\ell\in \N$ with $n\ge 2$. 
We will consider $\ell$ and $n$ as fixed and mostly omit them in the notations. For $i=1,\dots,\ell$ we set
\[\Index(i):=\bigl\{(I,J,\mu)\,\mid\, I\subset [\ell],\,|I|=i,\,J\subset [n+\ell],\,|J|=n+i,\, \mu\colon \bar I\to \bar J \text{ bijection}   \bigr\}\]
where $\bar I:=[\ell]\setminus I$ and $\bar J:=[n+\ell]\setminus J$ denote the complements of $I$ and $J$, respectively. For $(I,J,\mu)\in \Index(i)$ we consider the subvariety $\Gamma_{I,J,\mu}\subset X\times X^\ell\times X^{n+\ell}$ given by 
\begin{align*}\Gamma_{I,J,\mu}:=\bigl\{(x,x_1,\dots,x_\ell,y_1,\dots,y_{n+\ell})\mid x=x_a=y_b\,\forall\, a\in I\,,\, b\in J\text{ and } x_c=y_{\mu(c)}\,\forall\,c\in \bar I\bigr\}\,.  
\end{align*}
This subvariety is invariant under the action of the subgroup 
\[\sym_I\times \sym_{\bar I,\mu}\times \sym_J:=\{(\sigma,\tau)\mid \sigma(I)=I\,,\, \sigma(J)=J\,,\, (\mu\circ\sigma)_{|\bar I}=(\tau\circ \mu)_{|\bar I}\}\subset \sym_\ell\times \sym_{n+\ell}\]
and thus $\reg_{I,J,\mu}$ carries a canonical linearisation by this subgroup. 
Note that there is the isomorphism of groups 
$\sym_I\times \sym_{\bar I,\mu}\times \sym_J\cong \sym_{i}\times \sym_{\ell-i}\times \sym_{n+i}$ given by $(\sigma,\tau)\mapsto (\sigma_{|I}, \sigma_{|\bar I},\tau_{|J})$.
Let $\alt_J$ denote the one-dimensional representation of $\sym_I\times \sym_{\bar I,\mu}\times \sym_J$ on which the factor $\sym_J=\{\sigma=\id\}$ acts by multiplication by the 
sign of the permutations and the other factor $\sym_I\times \sym_{\bar I,\mu}=\{\tau_{|J}=\id_J\}$ acts trivially. We set $\cP(I,J,\mu):=\reg_{I,J,\mu}\otimes \alt_J$ and
\[\cP^i:=\Inf_{\sym_{[i]}\times \sym_{[i+1,\ell],e}\times \sym_{[n+i]}}^{\sym_{\ell}\times \sym_{n+\ell}} \cP([i],[n+i],e)=\bigoplus_{\Index(i)}\cP(I,J,\mu)\in \Coh_{\sym_\ell\times \sym_{n+\ell}}(X\times X^\ell\times X^{n+\ell})\,.\]   
For $c\in \bar I$ we have $\Gamma_{I\cup\{c\},J\cup \{\mu(c)\},\mu_{| \bar I\setminus \{c\}}}\subset \Gamma_{I,J,\mu}$. This allows us to define a differential $d^i\colon \cP^i\to \cP^i$ by letting the component 
$\cP(I,J,\mu)\to \cP(I\cup\{c\},J\cup \{\mu(c)\},\mu_{| \bar I\setminus \{c\}})$ be $\eps_{J,\mu(c)}:=(-1)^{\# \{b\in J\mid b<\mu(c)\}}$ times the map given by restriction of sections and setting all components which are not of this form to be zero.
The resulting $\sym_\ell\times \sym_{n+\ell}$-equivariant complex 
\[\cP_{\ell,n}:=\cP:=(0\to \cP^0\to \dots\to \cP^\ell\to 0)\,\in\, \D^b_{\sym_\ell\times \sym_{n+\ell}}(X^{n+\ell}) \]
is the Fourier--Mukai kernel of our functor $H_{\ell,n}$, that means 
\[H_{\ell,n}:=\FM_{\cP_{\ell,n}}\colon \D^b_{\sym_\ell}(X\times X^\ell)\to \D^b_{\sym_{n+\ell}}(X^{n+\ell})\,.\]
\subsection{Adjoint kernels}\label{cond3}
Even though we do not assume that $X$ is projective, since $X\times X^\ell$ and $X^{n+\ell}$ are smooth and $\supp\cP=\supp\cP^0$ is projective over $X\times X^\ell$ as well as over $X^{n+\ell}$, the functor $H_{\ell,n}$ has right and left adjoints 
$H_{\ell,n}^R,H_{\ell,n}^L\colon \D^b_{\sym_{n+\ell}}(X^{n+\ell})\to \D^b_{\sym_\ell}(X\times X^\ell)$ mapping to the bounded derived category. Their FM kernels are given by
\begin{align}\label{RLadjoints}\cP^R=\cP^\vee\otimes(\omega_{X\times X^\ell}\boxtimes \reg_{X^{n+\ell}})[(\ell+1)d]\quad,\quad \cP^L=\cP^\vee\otimes(\reg_{X\times X^\ell}\boxtimes \omega_{X^{n+\ell}})[(n+\ell)d]\,;
\end{align}
see e.g.\ \cite[Section 2.1]{Kuz}.

Using this, we can go ahead and show that for $X$ a smooth surface the functor $H_{\ell,n}$ satisfies condition (iii) of a $\P^{n-1}$-functor with cotwist $\bar S_X^{-1}:=(\_)\otimes (\omega_X^{-1}\boxtimes \reg_{X^\ell})[-2]$, i.e.\ that 
$H_{\ell,n}^R\cong \bar S_X^{-(n-1)}\circ H_{\ell,n}^L$. By (\ref{RLadjoints}) this amounts to the invariance of $\cP^\vee$ under tensor product by $\omega_X^n\boxtimes \omega_{X^\ell}\boxtimes \omega_{X^{n+\ell}}^{-1}$. This follows from the fact that
\[(\omega_X^n\boxtimes \omega_{X^\ell}\boxtimes \omega_{X^{n+\ell}}^{-1})_{|\Gamma_{I,J,\mu}}\cong \reg_{I,J,\mu}\quad\text{for all $0\le i\le \ell$ and $(I,J,\mu)\in \Index(i)$.}\]

\subsection{Description of the functor}
For $I\subset [\ell]$ and $J\subset [n+\ell]$ with $|I|=i$ and $|J|=n+i$ there are the \textit{partial diagonals}
\begin{align*}
X\times X^\ell\supset D_I:=\{(x,x_1,\dots,x_\ell)\mid x=x_a\,\forall\, a\in I\}\cong X\times X^{\ell-i}\,,\\ 
X^{n+\ell}\supset \Delta_J:=\{(y_1,\dots,y_ {n+\ell})\mid y_a=y_b\,\forall\, a,b\in J\}\cong X\times X^{\ell-i}\,,
\end{align*}
and we denote the closed embeddings by $\iota_I\colon X\times X^{\ell-i}\to X\times X^\ell$ and $\delta_J\colon X\times X^{\ell-i}\to X^{n+\ell}$. We set $\iota_{[0]}=\iota_\emptyset:=\id\colon X\times X^\ell\to X\times X^\ell$.
The functor $H^i:=H_{\ell,n}^i:=\FM_{\cP^i}$ is the composition 
\begin{align}\label{Hi}
& \D^b_{\sym_\ell}(X\times X^\ell)\xrightarrow{\Res}\D^b_{\sym_i\times \sym_{\ell-i}}(X\times X^\ell)\xrightarrow{\iota_{[i]}^*}
\D^b_{\sym_i\times \sym_{\ell-i}}(X\times X^{\ell-i})\\
\notag &\xrightarrow{(\_)^{\sym_{i}}}
\D^b_{\sym_{\ell-i}}(X\times X^{\ell-i}) \xrightarrow{\triv} \D^b_{\sym_{n+i}\times \sym_{\ell-i}}(X\times X^{\ell-i})\xrightarrow{\MM_{\alt_{n+i}}} \D^b_{\sym_{n+i}\times \sym_{\ell-i}}(X\times X^{\ell-i})\\
\notag & \xrightarrow{\delta_{[n+i]*}} \D^b_{\sym_{n+i}\times \sym_{\ell-i}}(X^{n+i}\times X^{\ell-i})
 \xrightarrow{\Inf} \D^b_{\sym_{n+\ell}}(X^{n+\ell})\,.
\end{align}
For $i=0$ this reduces to
\begin{align}\label{H0}
&\D^b_{\sym_\ell}(X\times X^\ell) \xrightarrow{\triv} \D^b_{\sym_{n}\times \sym_{\ell}}(X\times X^{\ell})\xrightarrow{\MM_{\alt_{n}}} \D^b_{\sym_{n}\times \sym_{\ell}}(X\times X^{\ell})\\\notag & \xrightarrow{\delta_{[n]*}} \D^b_{\sym_{n}\times \sym_{\ell}}
(X^{n}\times X^{\ell})
 \xrightarrow{\Inf} \D^b_{\sym_{n+\ell}}(X^{n+\ell})\,.
\end{align}
Note that (\ref{H0}) differs by the tensor product $\MM_{\alt_n}$ with the alternating representation from the suggestion (\ref{Hcompo}) from the introduction.
Using slightly shortened notation, the functor $H^i$ is on the level of objects given by
\begin{align}\label{Hiobj} H^i\colon E\mapsto \bigoplus_{J\subset [n+\ell]\,,\, \#J=n+i} \delta_{J*}\bigl(\alt_J\otimes \iota_{[i]}^*(E)^{\sym_{[i]}}\bigr)\,.
\end{align}
The right-adjoint $H^{iR}\colon \D^b_{\sym_{n+\ell}}(X^{n+\ell})\to \D^b_{\sym_\ell}(X^\ell)$ is given by the composition  
\begin{align}\label{HiR}
& \D^b_{\sym_\ell}(X\times X^\ell)\xleftarrow{\Inf}\D^b_{\sym_i\times \sym_{\ell-i}}(X\times X^\ell)\xleftarrow{\iota_{[i]*}}
\D^b_{\sym_i\times \sym_{\ell-i}}(X\times X^\ell)\\\notag
& \xleftarrow{\triv}\D^b_{\sym_{\ell-i}}(X\times X^{\ell-i}) \xleftarrow{(\_)^{\sym_{n+i}}} \D^b_{\sym_{n+i}\times \sym_{\ell-i}}(X\times X^{\ell-i})\xleftarrow{\MM_{\alt_{n+i}}} 
\D^b_{\sym_{n+i}\times \sym_{\ell-i}}(X\times X^{\ell-i})\\\notag &\xleftarrow{\delta_{[n+i]}^!} \D^b_{\sym_{n+i}\times \sym_{\ell-i}}(X^{n+i}\times X^{\ell-i}) \xleftarrow{\Res} \D^b_{\sym_{n+\ell}}(X^{n+\ell})\,.
\end{align}
which means on the level of objects $F\in \D^b_{\sym_{n+\ell}}(X^{n+\ell})$ that
\begin{align}\label{HiRobj}H^{iR}\colon F\mapsto \bigoplus_{I\subset [\ell]\,,\, \#I=i} \iota_{I*}(\alt_{[n+i]}\otimes \delta_{[n+i]}^!F)^{\sym_{[n+i]}}\,.\end{align}
\section{Techniques and examples}
\subsection{Derived intersections}
Given a vector bundle $E$ of rank $c$ on a variety $Z$ we write $\wedge^*E:=\oplus_{i=0}^c\wedge^iE[-i]$ and $\wedge^{-*}E:=\oplus_{i=0}^c\wedge^iE[i]$ as objects in $\D^b(Z)$.
\begin{theorem}[\cite{AC}]\label{selfint}
 Let $\iota\colon Z\hookrightarrow M$ be a regular embedding of codimension $c$ such that the normal bundle sequence $0\to T_Z\to T_{M|Z}\to N_{\iota}\to 0$ splits. Then there is an isomorphism 
\begin{align*}\iota^*\iota_*(\_)\simeq (\_)\otimes \wedge^{-*} N_{\iota}^\vee\end{align*}
of endofunctors of $\D^b(Z)$.
\end{theorem}
Recall that the right-adjoint of $\iota_*$ is given by $\iota^!=\MM_{\omega_\iota}\circ \iota^*[-\codim \iota]$ where $\omega_\iota=\wedge^{\codim \iota}N_{\iota}$; see \cite[Corollary III 7.3]{Har1}. 
We have $\wedge^{-*}N_{\iota}^\vee\otimes \omega_{\iota}[-\codim \iota]\cong \wedge^*N_{\iota}$.  
\begin{cor}\label{accor}
Under the same assumptions, there is an isomorphism
\begin{align*}\iota^!\iota_*(\_)\simeq (\_)\otimes\wedge^* N_{\iota}\end{align*}
\end{cor}
In particular the \textit{derived self-intersection} $\iota^*\iota_*\reg_Z$ of $Z$ in $M$ is given by 
$\iota^*\iota_*\reg_Z=\wedge^{-*}N_{\iota}^\vee$. 
Similar results for \textit{derived intersections}, i.e.\ for $\iota_2^*\iota_{1*}\reg_{Z_1}$ when $\iota_1\colon Z_1\to M$, $\iota_2\colon Z_2\to M$ are two different closed embeddings, are proven in \cite{Gri}.
 However, we will always be in the following situation where Theorem \ref{selfint} is sufficient. Assume that there is the diagram
\begin{align}\label{selfintdiag}\xymatrix{
            & Z_2 \ar^{\iota_2}[dr]\ar^{r}[d]&  \\
   T=Z_1\cap Z_2 \ar^{u}[ur]\ar_{v}[dr]   & W\ar^t[r] &  M\,.    \\
        &   Z_1\ar^{s}[u]\ar_{\iota_1}[ur]    &
} \end{align}  
where all the maps are regular closed embeddings, $t$ has a splitting normal bundle sequence, and the intersection of $Z_1$ and $Z_2$ inside of $W$ is transversal.
\begin{lemma}
Under the above assumptions there is the isomorphism of functors
\[\iota_2^* \iota_{1*}(\_)\cong u_*(v^*(\_)\otimes \wedge^{-*}N_{t|T}^\vee)\,.\]
In particular, $\iota_2^*( \iota_{1*}\reg_{Z_1})\cong u_*(\wedge^{-*}N_{t|T}^\vee)$.
\end{lemma}
\begin{proof}
Indeed, we have
\begin{align*}
\iota_2^*\iota_{1*}\cong r^*t^*t_*s_*\overset{\ref{selfint}}{\cong} r^*(s_*(\_)\otimes \wedge^{-*}N_t^\vee)\cong r^*s_*(\_)\otimes \wedge^{-*}N_{t|Z_2}^\vee&\cong u_*v^*(\_)\otimes \wedge^{-*}N_{t|Z_2}^\vee\\&\cong u_*(v^*(\_)\otimes \wedge^{-*}N_{t|T}^\vee) 
\end{align*}
where the prior to last isomorphism is the base change theorem \cite[Corollary 2.27]{Kuz}.
\end{proof}
\begin{cor}\label{dintcor}
 Under the same assumptions \begin{align*}\iota_2^!\iota_{1*}(\_)\cong u_*(v^*(\_)\otimes \wedge^{-*}N_{t|T}^\vee)\otimes \omega_{\iota_2}[-\codim\iota_2]\cong u_*(v^*(\_)\otimes \wedge^{*}N_{t|T}\otimes \omega_v)[-\codim v]\,.\end{align*}
In particular, $
\iota_2^!(\iota_{1*}\reg_{Z_1})\cong u_*(\wedge^{-*}N_{t|T}^\vee)\otimes \omega_{\iota_2}[-\codim\iota_2]\cong u_*(\wedge^{*}N_{t|T}\otimes \omega_{v})[-\codim v]$.
\end{cor}
Note that by Grothendieck duality $\iota_{2*}\Coho^p(\iota_2^!\iota_{1*}\reg_{Z_1})\cong \sExt^{p}_{\reg_M}(\reg_{Z_2},\reg_{Z_1})$.
\begin{remark}\label{normalinduced}
In the above situation consider in addition $Z_2\subset W'\subset W$ such that $w'\colon W'\to W$ is a regular embedding and $W'$ and $Z_1$ intersect transversally. We set $Z_1'=W'\cap Z_1$.
 We also consider $Z_1\subset W''\subset W$ such that $w''\colon W''\to W$ is a regular embedding and $W''$ and $Z_2$ intersect transversally in $Z_2''=W''\cap Z_2$. So we have the two diagrams of closed embeddings 
\begin{align*}\xymatrix{
         & Z_2\ar^\id[r]\ar[d]  & Z_2 \ar^{\iota_2}[dr]\ar^{r}[d]&  \\
   T \ar^{u}[ur]\ar_{v'}[dr] & W' \ar^{w'}[r] & W\ar^t[r] &  M\quad,    \\
    & Z_1'\ar^{z'}[r]\ar[u]   &   Z_1\ar^{s}[u]\ar_{\iota_1}[ur]    &
}
\quad
\xymatrix{
         & Z_2''\ar^{z''}[r]\ar[d]  & Z_2 \ar^{\iota_2}[dr]\ar^{r}[d]&  \\
   T \ar^{u''}[ur]\ar_{v}[dr] & W'' \ar^{w''}[r] & W\ar^t[r] &  M\,.    \\
    & Z_1\ar^{\id}[r]\ar[u]   &   Z_1\ar^{s}[u]\ar_{\iota_1}[ur]    &
}
\end{align*}  
We set $\iota_1'=\iota_1\circ z'$, $t'=t\circ w'$, $\iota_2''=\iota_2\circ z''$, and $t''=t\circ w''$. The restriction map $\iota_{1*}\reg_{Z_1}\to \iota_{1*}'\reg_{Z_1'}$ induces for every $q=0,\dots,\codim (t)$ the map
\[u_*(\wedge^qN_{t|T}^\vee)\otimes \omega_{\iota_2}\cong \sExt_{\reg_M}^{\codim(\iota_2)-q}(\reg_{Z_2},\reg_{Z_1})\to \sExt_{\reg_M}^{\codim(\iota_2)-q}(\reg_{Z_2},\reg_{Z'_1})\cong u_*(\wedge^qN_{t'|T}^\vee)\otimes \omega_{\iota_2}\,.\]
As one can check locally using the Koszul resolutions this map is given by the $q$-th wedge power of the canonical map $N^\vee_{t|W'}\to N^\vee_{t'}$. Similarly, for $q=0,\dots,\codim (t)$ the induced map
\[u_*(\wedge^qN_{t''|T}\otimes \omega_{v})\cong \sExt_{\reg_M}^{\codim (v)+q}(\reg_{Z''_2},\reg_{Z_1})\to \sExt_{\reg_M}^{\codim (v)+q}(\reg_{Z_2},\reg_{Z_1})\cong u_*(\wedge^qN_{t|T}\otimes \omega_{v})\]
is given by the $q$-th wedge power of the canonical map $N_{t''}\to N_{t|W''}$. 
\end{remark}
\begin{remark}
 Let $G$ be a finite group acting on $M$ such that all the subvarieties occurring above are invariant under this action. Then all the normal bundles carry a canonically induced $G$-linearisation. 
All the results of this subsection continue to hold as isomorphisms in the (derived) categories of $G$-equivariant sheaves when considering the normal bundles as $G$-bundles equipped with the canonical linearisations; compare \cite[Section 28]{Has}.   
\end{remark}
\subsection{Invariants of the standard representation}
Let $I$ be a finite set of cardinality at least 2.
The \textit{standard representation} $\rho_I$ of the symmetric group $\sym_I$ can be considered either as the subrepresentation $\rho_I\subset \C^I$ of the regular representation consisting of all vectors whose components add up to zero or as the
 quotient $\rho_I=\C^I/\C$ by the one-dimensional subspace of invariants. For $I\subset I'$ the first point of view gives a canonical $\sym_I$-equivariant inclusion 
$\rho_I\to \rho_{I'}$ while the second one gives a canonical $\sym_I$-equivariant surjection $\rho_{I'}\to \rho_I$. 
For $X$ a smooth variety and $\delta_n\colon X\to X^n$ the embedding of the small diagonal there is the $\sym_n$-equivariant isomorphism $N_{\delta_n}\cong T_X\otimes \rho_n$; see \cite[Section 3]{Kru3}. More generally, for $I\subset [n]$ the
 normal bundle of the partial diagonal $\Delta_I\cong X\times X^{\bar I}$ is as a $\sym_I$-bundle given by
\begin{align}\label{partialnormal}
 N_{\delta_I}\cong (T_X\otimes \rho_I)\boxtimes \reg_{X^{\bar I}}\quad,\quad N_{\delta_I}^\vee\cong (\Omega_X\otimes \rho_I)\boxtimes \reg_{X^{\bar I}}\,.
\end{align}
Furthermore, the normal bundle sequence of $\delta_I$ splits since $\Delta_I$ is the fixed point locus of the $\sym_I$-action on $X^n$; see \cite[Section 1.20]{AC}.

\begin{remark}\label{partialinduced}
For $I\subset I'\subset [n]$, the embedding $\Delta_{I'}\to \Delta_I$ induces maps $N_{\delta_{I'}}\to N_{\delta_I|\Delta_{I'}}$ and $N_{\delta_I|\delta_{I'}}^\vee\to N_{\delta_{I'}}^\vee$. 
Under the isomorphisms (\ref{partialnormal}) they are given by the canonical maps $\rho_{I'}\to \rho_I$ and $\rho_I\to \rho_{I'}$, respectively.
\end{remark}
For $m\ge 2$ and $X$ a smooth variety of dimension $d$, we set 
\begin{align}\label{Lambdadef}\Lambda_m^*(X):=\bigoplus_{i=0}^{(m-1)d}\left(\wedge^i(T_X\otimes \rho_m)\right)^{\sym_m}[-i]\,.\end{align}
\begin{lemma}\label{Lambdacurvesurface}\label{Tinva}
\[
\Lambda^*_m(X)=\begin{cases}
                \reg_X[0]\quad&\text{for $X$ a curve,}\\
               \reg_X[0]\oplus \omega_X^{-1}[-2]\oplus\dots\oplus \omega_X^{-(m-1)}[-2(m-1)]\quad&\text{for $X$ a surface.}                
\end{cases}
\]
\end{lemma}
\begin{proof}
For the surface case see \cite[Lemma B.5]{Sca1} and \cite[Corollary 3.5]{Kru3}. Since $\wedge^0(T_X\otimes \rho_m)=\reg_X$ is equipped with the trivial $\sym_m$-action, 
we only have to show that $\wedge^i(T_X\otimes \rho_m)$ has no invariants for $i\ge 1$ in the case that $X$ is a curve. For this it is sufficient to consider the fibres which are given by $\wedge^i\rho_m$. By \cite[Proposition 2.12]{FHrep} 
the representations $\wedge^i\rho_m$ are irreducible. They are non-trivial for $i\ge 1$ hence their invariants vanish. 
\end{proof}
\begin{remark}\label{higherdiminva}
 For $d= \dim X\ge 3$ also vector bundles of higher rank occur as direct summands of $\Lambda_m^*(X)$. For example, for $m=2$ we have
\[
\Lambda_2^*(X)\cong \bigoplus_{0\le k\le d/2} \wedge^{2k} T_X[-2k]\,.  
\]
\end{remark}
\begin{remark}
For $I\subset [n]$ of cardinality $m:=|I|\ge 2$ consider the functor $G=\delta_{I*}\circ \triv$ as well as $G^R\circ G$, i.e.\ the composition  
\[\D^b_{\sym_{\bar I}}(X\times X^{\bar I})\xrightarrow{\triv}\D^b_{\sym_I\times \sym_{\bar I}}(X\times X^{\bar I})\xrightarrow{\delta_{I*}} 
\D^b_{\sym_I\times \sym_{\bar I}}(X^n)\xrightarrow{\delta_I^!}\D^b_{\sym_I\times \sym_{\bar I}}(X\times X^{\bar I})\xrightarrow{(\_)^{\sym_I}}\D^b_{\sym_{\bar I}}(X\times X^{\bar I}).
\]
Let $\pr_X\colon X\times X^{\bar I}\to X$ be the projection to the first factor. Corollary \ref{accor} together with Lemma \ref{Tinva} give
\begin{align}\label{GRG}
 G^R\circ G\cong (\_)\otimes \pr_X^*\Lambda^*_{m}(X)\cong \begin{cases}
                  \id\,\,&\text{for $X$ a curve,}
                \\\bar S_X^{-[0,m-1]}:=\id\oplus \bar S_X^{-1}\oplus\dots\oplus \bar S_X^{-(m-1)}\,\,&\text{for $X$ a surface.}
                 \end{cases}
\end{align}
Here, $\bar S_X:=(\_)\otimes(\omega_X\boxtimes \reg_{X^{\bar I}})[-2]\colon \D^b_{\sym_{\ell}}(X)\to \D^b_{\sym_{\ell}}(X)$ for any (not necessarily projective) smooth surface.
\end{remark}
%
%
%
%
\subsection{The case $\ell=0$}
In the special case that $I=[n]$ we have $G=\MM_{\alt_n}\circ H_{0,n}$ and (\ref{GRG}) gives
\begin{align*}
 H_{0,n}^R\circ H_{0,n}\cong \begin{cases}
                  \id\quad&\text{for $X$ a curve,}
                \\S_X^{-[0,n-1]}:=\id\oplus S_X^{-1}\oplus\dots\oplus S_X^{-(n-1)}\quad&\text{for $X$ a surface.}
                 \end{cases}
\end{align*}
This proves the case $\ell=0$ of Proposition A (i) and most of Theorem C (for the proof that condition (ii) of a $\P^{n-1}$-functor holds for $H_{0,n}$ in the surface case, see \cite[Section 3]{Kru3}).  
\subsection{The approach for general $\ell$}\label{approach}
There is the commutative diagram 
\begin{align}\label{trianglediag}\xymatrix{
    \cP^{\ell R}\star \cP       \ar[r]\ar[d]  &  \cP^{\ell R}\star \cP^0\ar[r]\ar[d]  &\hdots \ar[r]\ar[d]  & \cP^{\ell R}\star \cP^\ell \ar[d]\\
  \vdots \ar[r]\ar[d]   &\vdots  \ar[r]\ar[d] &\ddots \ar[r]\ar[d]   & \vdots  \ar[d]  \\
\cP^{0R}\star \cP   \ar[r]\ar[d]   & \cP^{0R}\star \cP^0 \ar[r]\ar[d] &\hdots \ar[r]\ar[d]   &  \cP^{0R}\star \cP^\ell \ar[d]  \\   
\cP^R\star \cP \ar[r]   & \cP^R\star \cP^0  \ar[r]  & \hdots \ar[r]   & \cP^R\star \cP^\ell 
} \end{align}
where the $\cP^{iR}\star \cP$ and $\cP^R\star\cP^j$ are the left and right convolutions of the rows and columns, respectively.
That means in particular that $\cP^R\star\cP^j$ can be written as a multiple cone
\[
 \cP^R\star \cP^j\cong \cone\bigl((\dots\cone(\cone(\cP^{\ell R}\star \cP^j\to \cP^{(\ell-1)R}\star \cP^j)\to \cP^{(\ell-2)R}\star \cP^j))\dots)\to \cP^{0R}\star \cP^j\bigr)\,;
\]
see e.g.\ \cite[Section 2]{Kawstack} for the notion of convolutions in triangulated categories.
The strategy of the proof of Proposition A and Theorem C that will be given in Section \ref{mainproof} is to start with the computation of the $\cP^{iR}\star\cP^j$, then use the results to compute the $\cP^R\star \cP^j$, and finally deduce the 
desired formulae for $\cP^R\star \cP$. 

In the following subsections we will do the computation of the $\cP^{iR}\star \cP^{j}$ in the case $\ell=1$ (and some of it for $\ell=2$) on the level of the functors. That means that we will compute the compositions $H_{\ell,n}^{iR}\circ H_{\ell,n}^j$.
 We will see that the undesired terms (compare (\ref{error})) are of a form which give them a good chance to chancel out when passing to $H_{\ell,n}^R\circ H_{\ell,n}$. But we will not compute the 
induced maps $H_{\ell,n}^{iR}H_{\ell,n}^j\to H_{\ell,n}^{iR}H_{\ell,n}^{j+1}$ and $H_{\ell,n}^{iR}H_{\ell,n}^j\to H_{\ell,n}^{(i-1)R}H_{\ell,n}^j$. 
Hence, we will not see that the terms really chancel until Section \ref{mainproof} where the computations are performed for general $\ell$ on the level of the FM kernels.
\subsection{Invariants of inflations}\label{Dansection}
For the computation of the invariants we will use the following principle; compare \cite[Lemma 2.2]{Dan} and \cite[Remark 2.4.2]{Sca1}.
Let $M$ be a variety and $G$ a finite group that we consider acting trivially on $M$. Let $\cE=(E,\lambda)\in \D^b_G(M)$ such that
$E=\oplus_I E_i$ in $\D^b(M)$ for some finite index set $I$. Let us assume that there is an action of $G$ on $I$ such that $\lambda_{g}(E_i)=E_{g(i)}$ for all $i\in I$.
We say that the $G$-action on $I$ \textit{is induced by} the $G$-linearisation of $E$. 
We denote $E_i$ together with the $G_i$-linearisation $(\lambda_{g|E_i})_{g\in G_i}$ by $\cE_i\in \D^b_{G_i}(M)$ where $G_i=\Stab_G(i)$.
The induced action of $G$ on $I$ is transitive if and only if $\cE\cong\Inf_{G_i}^G \cE_i$ for any $i\in I$; see \cite[Section 8.2]{BL}.
In that case, for every $i\in I$ the projection $E\to E_i$ induces the isomorphism $\cE^G\cong \cE_i^{G_i}$. The inverse is given by $s\mapsto \oplus_{[g]\in G/G_i}\lambda_g(s)$. 

Let the action of $G$ on $I$ be not transitive with $i_1,\dots,i_k$ being a system of representatives of the $G$-orbits. Then $\cE\cong \Inf_{G_{i_1}}^G \cE_1\oplus\dots\oplus \Inf_{G_{i_k}}^G \cE_k$ and 
\begin{align}\label{Danlem}\cE^G=\cE_1^{G_{i_1}}\oplus \dots\oplus \cE_k^{G_{i_k}}\,.\end{align}  
\subsection{The case $\ell=1$}\label{l1}
Let $\ell=1$ and $n>1$. 
We set $H:=H_{1,n}$ and aim to compute $H^R\circ H\colon \D^b(X\times X)\to \D^b(X\times X)$
using the descriptions (\ref{Hiobj}) and (\ref{HiRobj}) of $H^j$ and $H^{iR}$.
For $E\in \D^b(X\times X)$ we have 
\[H^{0R}H^0(E)\cong \bigl[\alt_{[n]}\otimes \delta_{[n]}^!\bigl(\bigoplus_{a\in[n+1]}\delta_{[n+1]\setminus\{a\}*}(E\otimes \alt_{[n+1]\setminus \{a\}})\bigr)     \bigr]^{\sym_{[n]}}\,.
\]
For $\sigma\in \sym_{[n]}$ the $\sym_{[n]}$-linearisation of  
$\bigoplus_{a\in[n+1]}\delta_{[n]}^!\delta_{[n+1]\setminus\{a\}*}(E\otimes \alt_{[n+1]\setminus \{a\}})$ maps the summand  $\delta_{[n]}^!\delta_{[n+1]\setminus \{a\}*}(E\otimes \alt_{[n+1]\setminus \{a\}})$ to
 $\delta_{[n]}^!\delta_{[n+1]\setminus \{\sigma(a)\}}(E\otimes \alt_{[n+1]\setminus \{\sigma(a)\}})$. Thus, the induced action on the index set $[n+1]$ is given by $a\mapsto \sigma(a)$. Hence there are two $\sym_{[n]}$-orbits, namely $[n]$ and $\{n+1\}$. 
We have $\Stab_{\sym_{[n]}}(n+1)=\sym_{[n]}$ and $\Stab_{\sym_{[n]}}(n)=\sym_{[n-1]}$.
As explained in Section \ref{Dansection} it follows that 
\[H^{0R}H^0(E)\cong \delta_{[n]}^!\delta_{[n]*}(E)^{\sym_{[n]}}\oplus 
\delta_{[n]}^!\delta_{[n-1]\cup \{n+1\}*}(E)^{\sym_{[n-1]}}\,.\]
The first direct summand equals $E\otimes \pr_1^*\Lambda_{n}^*(X)$ by (\ref{GRG}). 
\begin{conv}
For $\{b\}\subset [m]$ a set with one element we set $\Delta_{\{b\}}=X^m$. Furthermore, we set $\Lambda_1^*(X)=\reg_X[0]$. 
These notations occur in this subsection in the case that $n=2$ and later in the more general case that $n=\ell+1$.  
\end{conv}
For the computation of the second summand, consider the commutative diagram of closed embeddings 
\begin{align*}\xymatrix{
            & \Delta_{[n]} \ar^{\delta_{[n]}}[dr]\ar^{r}[d]&  \\
   \Delta_{[n+1]} \ar^{u}[ur]\ar_{v}[dr]   & \Delta_{[n-1]}\ar^{\delta_{[n-1]}}[r] &  X^{n+1}\,.    \\
        &   \Delta_{[n-1]\cup\{n+1\}}\ar^{s}[u]\ar_{\delta_{[n-1]\cup\{n+1\}}}[ur]    &
} \end{align*}  
It fulfils the properties of diagram (\ref{selfintdiag}) which means that $\Delta_{[n+1]}=\Delta_{[n]}\cap \Delta_{[n-1]\cup\{n+1\}}$ and that this intersection is transversal inside $\Delta_{[n-1]}$. This allows us to apply Corollary 
\ref{dintcor} to get
\begin{align}\label{preinva}\delta_{[n]}^!\delta_{[n-1]\cup\{n+1\}*}(E)\cong u_*(v^*(\_)\otimes \wedge^{*}N_{\delta_{[n-1]}|\Delta_{[n+1]}}\otimes \omega_v)[-\codim v]\,.\end{align}
Under the isomorphisms $\Delta_{[n+1]}\cong X$ and $\Delta_{[n]}\cong X\times X\cong \Delta_{[n-1]\cup\{n+1\}}$ the embeddings $u$ and $v$ equal the diagonal embedding $\iota\colon X\to X\times X$. Thus, $\codim v=\dim X=d$ and $\omega_v\cong \wedge^dN_v\cong \omega_X^{-1}$. 
It follows together with (\ref{partialnormal}) that after taking $\sym_{[n-1]}$-invariants in (\ref{preinva}) we get $\delta_{[n]}^!\delta_{[n-1]\cup\{n+1\}*}(E)^{\sym_{[n-1]}}\cong \iota_*(\iota^*(E)\otimes \Lambda_{n-1}^*(X)\otimes \omega_X^{-1})[-d]$.
In summary, 
\begin{align}\label{00}
 H^{0R}H^0(E)\cong E \otimes\pr_1^* \Lambda_{n}^*(X)\oplus \iota_*(\iota^*(E)\otimes \Lambda_{n-1}^*(X)\otimes \omega_X^{-1})[-d].
\end{align}
The computation of the other three functor compositions is easier. Note that $\delta_{[n+1]}=\delta_{[n]}\circ u$ and $u^!\cong u^*(\_)\otimes \omega_X^{-1}[-d]$. Hence,
\begin{align}\label{01}
 H^{0R}H^1(E)\cong \bigl[\alt_{[n]}\otimes \delta_{[n]}^!\delta_{[n+1]*}(\iota^*E\otimes \alt_{[n+1]}) \bigr]^{\sym_{[n]}}\cong\, &\delta_{[n]}^!\delta_{[n]*}u_*\iota^*(E)^{\sym_{[n]}}\\\notag\overset{(\ref{GRG})}\cong &\iota_*\bigl(\iota^*(E)\otimes \Lambda_{n}^*(X)\bigr)\,,
\end{align}
 \begin{align}\label{10}
 H^{1R}H^0(E)\cong\, &\iota_*\bigl[\alt_{[n]}\otimes \delta_{[n+1]}^!\bigl(\bigoplus_{a=1}^{n+1}\delta_{[n+1]\setminus \{a\}*}(E\otimes \alt_{[n+1]\setminus \{a\}})\bigr) \bigr]^{\sym_{[n+1]}}
 \\
 \notag
 \overset{(\ref{Danlem})}\cong& \iota_*\delta_{[n+1]}^!\delta_{[n]*}(E)^{\sym_{[n]}}
 \\\notag
 \cong\,& \iota_*u^!\delta_{[n]}^!\delta_{[n]*}\iota^*(E)^{\sym_{[n]}}
 \\\notag
 \overset{(\ref{GRG})}\cong &\iota_*\bigl(\iota^*(E)\otimes \Lambda_{n}^*(X)\otimes \omega_X^{-1}\bigr)[-d]\,,
\end{align}
\begin{align}\label{11}
 H^{1R}H^1(E)\cong\, \iota_*\bigl[\delta_{[n+1]}^!\delta_{[n+1]*}\iota^*(E)\bigr]^{\sym_{[n+1]}}\overset{(\ref{GRG})}\cong 
 \iota_*\bigl(\iota^*(E)\otimes \Lambda_{n+1}^*(X)\bigr)\,.
\end{align}
Let $X=C$ be a smooth curve. By Lemma \ref{Lambdacurvesurface} we have $\Lambda_m^*(C)=\reg_C[0]$ for all $m\ge 0$. Plugging this into (\ref{00}), (\ref{01}), (\ref{10}), and (\ref{11}) we get
\begin{align}\label{cd}
\begin{CD}
H^{1R}H^0 
@>{}>>  H^{1R}H^1
 \\
@V{}VV
@V{}VV \\
H^{0R}H^0
@>{}>>
H^{0R}H^1
\end{CD} 
\quad\cong\quad 
\begin{CD}
 \iota_*(\iota^*(\_)\otimes \omega_C^{-1})[-1]
@>{}>>  
\iota_*\iota^*
 \\
@V{}VV
@V{}VV \\
\id \oplus \iota_*(\iota^*(\_)\otimes \omega_C^{-1})[-1]
@>{}>>
\iota_*\iota^*
\end{CD} 
\end{align}
We will see in Section \ref{curveinduced} that the right-hand vertical map of this diagram as well as the component $\iota_*(\iota^*(\_)\otimes \omega_C^{-1})[-1]\to \iota_*(\iota^*(\_)\otimes \omega_C^{-1})[-1]$ of the left-hand vertical map are isomorphisms. 
Thus, by taking cones in the diagram (\ref{trianglediag}) enlarging (\ref{cd}) we get $H^RH^0\cong \id$ and $H^RH^1=0$. Considering the triangle $H^RH\to H^RH^0\to H^RH^1$ shows $H^RH=\id$, i.e.\ Proposition A (i) in the case $\ell=1$.

For $X$ a smooth surface we have $\Lambda^*_m(X)\cong S_X^{-[0,m-1]}$. This gives
 \[
\begin{CD}
H^{1R}H^0 
@>{}>>  H^{1R}H^1
 \\
@V{}VV
@V{}VV \\
H^{0R}H^0
@>{}>>
H^{0R}H^1
\end{CD} 
\quad\cong\quad 
\begin{CD}
 \iota_*S_X^{-[1,n]}\iota^*
@>{}>>  
\iota_*S_X^{-[0,n]}\iota^*
 \\
@V{}VV
@V{}VV \\
\bar S_X^{-[0,n-1]} \oplus \iota_* S_X^{-[1,n-1]}\iota^*
@>{}>>
\iota_* S_X^{-[0,n-1]}\iota^*
\end{CD} 
\]
where $\bar S_X^{-1}=(\_)\otimes \pr_1^*\omega_X^{-1}[-2]$. Again, we will see later that all components of the form $\iota_*S_X^{-k}\iota^*\to \iota_*S_X^{-k}\iota^*$ in the diagram are isomorphisms which gives by taking cones 
\[H^RH^0\cong \bar S_X^{-[0,n-1]}\oplus \iota_*S_X^{-n} \iota^*[1]\quad,\quad H^RH^1\cong  \iota_* S_X^{-n}\iota^*[1]\]
and finally $H^RH^0\cong \bar S_X^{-[0,n-1]}$ 
as claimed in Theorem C. 
\subsection{Orthogonality in the curve case}
We want to compute that $H_{1,n}^RH_{0,n+1}=0$ for $X=C$ a curve which is one instance of Proposition A (ii). 
We have 
\begin{align}
\label{e1} H_{1,n}^{1R}H_{0,n+1}(E)&\cong \iota_*\bigl(\delta_{[n+1]}^!\delta_{[n+1]*}(E)\bigr)^{\sym_{n+1}}\cong \iota_*(E\otimes \Lambda_{n+1}^*(X))\,,\\ 
\label{e2} H_{1,n}^{0R}H_{0,n+1}(E)&\cong \delta_{[n]}^!\delta_{[n+1]*}(E)^{\sym_{n}}\cong \delta_{[n]}^!\delta_{[n]*}u_*(E)^{\sym_n}\cong \iota_*(E \otimes \Lambda_n^*(X))\,.
\end{align}
For $X=C$ a curve this gives $H_{1,n}^{1R}H_{0,n+1}\cong \id$ and $H_{1,n}^{0R}H_{0,n+1}\cong \id$. It follows that $H_{1,n}^{R}H_{0,n+1}=0$ because of the exact triangle 
\begin{align}\label{12triangle}H_{1,n}^{1R}H_{0,n+1}\to H_{1,n}^{0R}H_{0,n+1}\to H_{1,n}^{R}H_{0,n+1}\,.\end{align}                                                                                                                                                                   
\subsection{Non-orthogonality in the surface case}\label{nonorth}
For $X$ a smooth surface (\ref{e1}) and (\ref{e2}) give 
$H_{1,n}^{1R}H_{0,n+1}\cong\iota_* S_X^{-[0,n]}$ and $H_{1,n}^{0R}H_{0,n+1}\cong\iota_* S_X^{-[0,n-1]}$. The components $\iota_*S_X^{-k}\to \iota_*S_X^{-k}$ of the induced map $H_{1,n}^{1R}H_{0,n+1}\to H_{1,n}^{0R}H_{0,n+1}$ are isomorphisms for $k=0,\dots, n-1$. Thus, 
\begin{align}\label{surfaceco}H_{1,n}^RH_{0,n+1}\cong \iota_*S_X^{-n}[1]\end{align}
by triangle (\ref{12triangle}). 
Let $E,F\in \D^b(X)$.
As in Section \ref{simNaka} we set $H_{0,n+1}(F)=H_{0,n+1}\circ I_F$ and $H_{1,n}(E)=H_{1,n}\circ I_E$. Note that the domain of $H_{0,n+1}(F)$ is $\D^b(\pt)$ and the functor is given by sending the generator $\C[0]$ to $\delta_{[n+1]*}(F)$.
By (\ref{surfaceco}) we have
\[
H_{1,n}(E)^RH_{0,n+1}(F)\cong \pr_{2*}\sHom\bigl(E\boxtimes \reg_X, \iota_*(F\otimes \omega_X^{-n})\bigr)[-2n+1]\cong \sHom(E, F\otimes \omega_X^{-n})\bigr[-2n+1]\,.  
\]
That seems not to fit well with the commutator relation (\ref{Nakarel}) which states  
$q_{1,-n}\circ q_{0,n+1}=0$.

\subsection{The case $\ell=2$, $n=2$}\label{ln2}
Set $H=H_{2,2}\colon \D^b_{\sym_2}(X\times X^2)\to \D^b_{\sym_4}(X^4)$.
We consider this case in order to illustrate why the assumption that $n>\ell$ is necessary for Proposition A and Theorem C.
We have 
\[H^{0R}H^0(E)\cong \bigl[\alt_{[2]}\otimes \delta_{[2]}^!\bigl(\bigoplus_{J\subset [4]\,,\, |J|=2}\delta_{J*}(E\otimes \alt_J)\bigr)     \bigr]^{\sym_{[2]}}\,.\]
For $J=[2]$ we get the direct summand $\delta_{[2]}^!\delta_{[2]*}(E)^{\sym_{[2]}}\cong E\otimes \pr_X^*\Lambda_{2}^*(X)$ which has the shape of a fully faithful functor for $X$ a curve and the shape of a $\P^1$-functor with cotwist $\bar S_X^{-1}$ for $X$ a surface. For $J=[3,4]$ we consider the diagram 
\[
 \begin{CD}
 \Delta_{[2]}\cap\Delta_{[3,4]} 
@>{u}>>  
\Delta_{[2]}
 \\
@V{v}VV
@V{\delta_{[2]}}VV \\
\Delta_{[3,4]}
@>{\delta_{[3,4]}}>>
X^4
\end{CD} 
\]
which is a transversal intersection. Under the isomorphism $\Delta_{[2]}\cong X\times X^2$ the subvariety $X\times X\cong \Delta_{[2]}\cap \Delta_{[3,4]}\subset \Delta_{[2]}$ equals $X\times \Delta_X$. Thus, for an appropriate choice of $E$ the direct summand 
$[\alt_{[2]}\otimes \delta_{[2]}^!\delta_{[3,4]*}(E\otimes \alt_{[3,4]})]^{\sym_{[2]}}$ of $H^{0R}H^0(E)$ is supported on the whole $X\times \Delta_X$. On the other hand one can see easily that all direct summands of $H^{iR}H^j$ for $(i,j)\neq(0,0)$ are supported on one of the subvarieties
 $D_1$, $D_2$ or $D_{[2]}$ of $X\times X^2$, neither of them containing $X\times \Delta_X$. It follows that the  direct summand $[\alt_{[2]}\otimes \delta_{[2]}^!\delta_{[3,4]*}(E\otimes \alt_{[3,4]})]^{\sym_{[2]}}$ of $H^{0R}H^0$ survives taking the multiple cones in the diagram 
(\ref{trianglediag}) which prevents $H^RH$ from being isomorphic to $\id$ or $\bar S_X^{-[0,1]}$.     
%
%
\section{Proof of the main results}\label{mainproof}
Let $X$ be a smooth variety of dimension $d:=\dim X$.
Let furthermore $\ell,n\in \N$ with $n>\max\{\ell,1\}$ and consider $\cP=\cP_{\ell,n}$; see Section \ref{cP}.
We will compute the convolution products $\cP^{iR}\star \cP^j$. In the case that $X$ is a curve or a surface that will lead to formulae for $\cP^R\star\cP$. 
\subsection{Computation of the direct summands}\label{directsummand}
For $(I_1,J_1,\mu_1)\in \Index(j)$ and $(I_2,J_2,\mu_2)\in \Index(i)$ we set $K_1:=I_1\cup \mu_1^{-1}(J_2), K_2:=I_2\cup\mu_2^{-1}(J_1)\subset [\ell]$ and $\mu:=\mu^{-1}_{2|\overline{J_1\cup J_2}}\circ \mu_{1|\bar K_1}$ as a bijection between  $\bar K_1=[\ell]\setminus K_1$ and $\bar K_2=[\ell]\setminus K_2$.
Furthermore, let $\Gamma_{K_1,K_2,\mu}\subset X\times X^\ell\times X\times X^\ell$ be the subvariety given by
\[\Gamma_{K_1,K_2,\mu}:=\bigl\{(x,x_1,\dots,x_\ell,z,z_1,\dots,z_\ell)\mid x=x_a=z_b=z\forall a\in K_1,b\in K_2,x_c=z_{\mu(c)}\forall c\in \bar K_1 \bigr\}.\]
Consider the diagram
\begin{align}\label{graphintdiag}\xymatrix{
            & X\times X^\ell\times \Gamma_{I_2,J_2,\mu_2} \ar^{\iota_2}[dr]\ar^{r}[d]& & \\
   T \ar_{\pi_{13}}^\cong[dd]\ar^{u}[ur]\ar_{v}[dr]   & X\times X^\ell\times \Delta_{J_1\cap J_2}\times X\times X^\ell\ar^t[r] &  X\times X^\ell\times X^{n+\ell}\times X\times X^\ell\ar_{\pr_{13}}[dd] &   \\
        &   \Gamma_{I_1,J_1,\mu_1}\times X\times X^\ell\ar^{s}[u]\ar_{\iota_1}[ur] \ar_{\pi'_{13}}^\cong[d]   & &\\
\Gamma_{K_1,K_2,\mu}\ar^{\tilde v}[r]\ar_{\pi_1}^\cong[d]& D_{I_1}\times X\times X^\ell \ar[r] \ar_{\pi'_1}[d] & X\times X^\ell\times X\times X^\ell\ar^{\,\,\,\quad p}[r] \ar_{\pr_1}[d] & X\\
D_{K_1}\ar^\alpha[r]& D_{I_1}\ar[r] & X\times X^\ell & 
} \end{align} 
where $T:=(\Gamma_{I_1,J_1,\mu_1}\times X\times X^\ell)\cap (X\times X^\ell\times \Gamma_{I_2,J_2,\mu_2})$, $\pi_{13}$ and $\pi_{13}'$ are the restrictions of the projection $\pr_{13}$, 
$\pi_1$ and $\pi_1'$ are the restrictions of the projection $\pr_1$,
$p$ is the projection to the third factor, and all the other arrows denote the appropriate closed 
embeddings.
Note that $J_1\cap J_2\neq \emptyset$ because of the assumption that $n>\ell$. We have
\[T=\begin{Bmatrix}(x,x_1,\dots,x_\ell,y_1,\dots,y_{n+\ell},z,z_1,\dots,z_\ell)&\mid x=x_a=y_b=z_c=z\,,\,\, x_d=y_{\mu_1(d)}=z_{\mu(d)}\\&\forall\, a\in K_1,b\in J_1\cup J_2,c\in K_2,\, d\in \bar K_1 \end{Bmatrix}.
\]
We see that a point in $T$ is determined by its $(x,x_1,\dots,x_\ell)$ component. 
Thus, $\pi_{13}$ and $\pi_1$ are isomorphisms. Similarly, $\pi_{13}'$ is an isomorphism.
Furthermore, we have $\codim v=\codim \tilde v=(k-j+\ell+1)d$ where $k:=|K_1|=|K_2|$. Let $\pi_2\colon \Gamma_{K_1,K_2,\mu}\to X\times X^\ell$ be the restriction of the projection $\pr_2\colon X\times X^\ell\times X\times X^\ell\to  X\times X^\ell$ to the second factor. By the adjunction formula
\begin{align*}\pi_{13*}\omega_v\cong \omega_{\tilde v}\cong 
\pi_1^*\omega_{D_{K_1}}\otimes \tilde v^*\pi_1'^*\omega_{D_{I_1}}^{-1}\otimes \pi_2^*\omega_{X\times X^\ell}^{-1}\cong \omega_\alpha \otimes \pi_2^*\omega_{X\times X^\ell}^{-1}
\cong p^*\omega_X^{-(k-j)} \otimes \pi_2^*\omega_{X\times X^\ell}^{-1}
\end{align*}
where the last isomorphism is due to the fact that on $\Gamma_{K_1,K_2,\mu}$ the projection to the first factor $X\times X^\ell\times X\times X^\ell\to X$ coincides with the projection $p$ to the third factor.
It follows that 
$ \pi_{13*}\omega_v \otimes \pi_2^*\omega_{X\times X^\ell}\cong p^*\omega_X^{-(k-j)}$.
Note that $|J_1\cup J_2|=n+k$ and $|J_1\cap J_2|=n+i+j-k$.
One can check that diagram (\ref{graphintdiag}) with the two bottom lines removed satisfies the properties of diagram (\ref{selfintdiag}), i.e.\ the square consisting of $u$, $v$, $s$, and $r$ is a transversal intersection.
Using Corollary \ref{dintcor} together with (\ref{partialnormal}) we get
\begin{align}
\notag (\reg_{I_2,J_2,\mu_2})^R\star\reg_{I_1,J_1,\mu_1} &\cong \pr_{13*}\sHom(\pr_{23}^*\reg_{I_2,J_2,\mu_2},\pr_{12}^*\reg_{I_1,J_1,\mu_1})\otimes \pr_2^*\omega_{X\times X^\ell}  [(\ell+1)\cdot d]
\\ \label{summandT}
&\cong \reg_{K_1,K_2,\mu}\otimes p^*\bigl(\wedge^*(T_X\otimes \rho_{J_1\cap J_2})\otimes \omega_X^{-(k-j)}  \bigr) [-(k-j)\cdot d]
\\\label{summandO}
&\cong \reg_{K_1,K_2,\mu}\otimes p^*\bigl(\wedge^{-*}(\Omega_X\otimes \rho_{J_1\cap J_2})\otimes \omega_X^{-(n+i-1)}  \bigr)[-(n+i-1)\cdot d]\,.
\end{align}
Note that here $(\reg_{I_2,J_2,\mu_2})^R\star\reg_{I_1,J_1,\mu_1}=\pr_{13*}(\pr_{23}^*\reg_{I_2,J_2,\mu_2})^R\otimes\pr_{12}^* \reg_{I_1,J_1,\mu_1}$ is the non-equivariant convolution product, i.e. the functor of invariants is not applied; compare (\ref{conprod}).

\subsection{The induced maps}\label{inducedmapsummands}
Let $c\in \bar I_1$ with $\mu_1(c)\in J_2$ and set $I_1'=I_1\cup\{c\}$, $J_1'=J_1\cup\{\mu_1(c)\}$, and $\mu_{1}':=\mu_{1|\bar I_1\setminus\{c\}}$. The restriction $\reg_{I_1,J_1,\mu_1}\to \reg_{I_1',J_1',\mu_1'}$ induces for $q=0,\dots,(n+i+j-k)d$ a map 
\[\Coho^{(n+i-1)d-q}\bigl((\reg_{I_2,J_2,\mu_2})^R\star\reg_{I_1,J_1,\mu_1}\bigr)\to \Coho^{(n+i-1)d-q}\bigl((\reg_{I_2,J_2,\mu_2})^R\star\reg_{I'_1,J'_1,\mu'_1}\bigr)\]
which corresponds under the isomorphism (\ref{summandO}) to a map
\[\reg_{K_1,K_2,\mu}\otimes p^*\bigl(\wedge^{q}(\Omega_X\otimes \rho_{J_1\cap J_2})\otimes \omega_X^{-(n+i-1)}\bigr)\to \reg_{K_1,K_2,\mu}\otimes p^*\bigl(\wedge^{q}(\Omega_X\otimes \rho_{J_1'\cap J_2})\otimes \omega_X^{-(n+i-1)}\bigr)\,.\]     
By Remarks \ref{normalinduced} and \ref{partialinduced}, it is given by the canonical inclusion $\rho_{J_1\cap J_2}\to \rho_{J_1'\cap J_2}=\rho_{(J_1\cap J_2)\cup\{\mu_1(c)\}}$.
To see this, set $\tilde \Delta_J=X\times X^\ell\times \Delta_J\times X\times X^\ell$ for $J\subset [n+\ell]$ and consider the diagram
\[
\xymatrix{
         & X\times X^\ell\times \Gamma_{I_2,J_2,\mu_2} \ar^\id[r]\ar[d]  & X\times X^\ell\times \Gamma_{I_2,J_2,\mu_2} \ar^{\iota_2}[dr]\ar^{r}[d]&  \\
   T \ar^{u}[ur]\ar_{v'}[dr] & \tilde \Delta_{J_1'\cap J_2} \ar^{w'}[r] & \tilde \Delta_{J_1\cap J_2}\ar^{t\quad}[r] & X\times X^\ell\times X^{n+\ell}\times X\times X^\ell\,.    \\
    & \Gamma_{I'_1,J'_1,\mu'_1}\times X\times X^\ell\ar^{z'}[r]\ar[u]   &   \Gamma_{I_1,J_1,\mu_1}\times X\times X^\ell\ar^{s}[u]\ar_{\iota_1}[ur]    &
} 
\]
Similarly, let $c\in \bar I_2$ with $\mu_2(c)\in J_1$ and set $I_2'=I_2\cup\{c\}$, $J_2'=J_2\cup\{\mu_2(c)\}$, and $\mu_{2}':=\mu_{2|\bar I_2\setminus\{c\}}$. Then the restriction $\reg_{I_2,J_2,\mu_2}\to \reg_{I_2',J_2',\mu_2'}$ induces for $q=0,\dots,(n+i+j-k)d$
 a map \[\Coho^{(k-j)d+q}\bigl((\reg_{I'_2,J'_2,\mu'_2})^R\star\reg_{I_1,J_1,\mu_1}\bigr)\to \Coho^{(k-j)d+q}\bigl((\reg_{I_2,J_2,\mu_2})^R\star\reg_{I_1,J_1,\mu_1}\bigr)\]
which corresponds under the isomorphism (\ref{summandT}) to a map
\[\reg_{K_1,K_2,\mu}\otimes p^*\bigl(\wedge^{q}(T_X\otimes \rho_{J_1\cap J'_2})\otimes \omega_X^{-(k-j)}\bigr)\to \reg_{K_1,K_2,\mu}\otimes p^*\bigl(\wedge^{q}(T_X\otimes \rho_{J_1\cap J_2})\otimes \omega_X^{-(k-j)}\bigr)\]     
which is given by the canonical surjection $\rho_{J_1\cap J'_2}=\rho_{(J_1\cap J_2)\cup\{\mu_2(c)\}}\to \rho_{J_1\cap J_2}$.
In particular, the induced map $\Coho^{(k-j)d}\bigl((\reg_{I'_2,J'_2,\mu'_2})^R\star\reg_{I_1,J_1,\mu_1}\bigr)\to \Coho^{(k-j)d}\bigl((\reg_{I_2,J_2,\mu_2})^R\star\reg_{I_1,J_1,\mu_1}\bigr)$ is given by the identity on $\reg_{K_1,K_2,\mu}\otimes p^*\omega_X^{-(k-j)}$. 
\subsection{Computation of the $\cP^{iR}\star\cP^j$}\label{ifsub}
We make use of the principle explained in Section \ref{Dansection} in order to compute $\cP^{iR}\star\cP^j\cong \pr_{13*}(\pr_{23}^*\cP^{iR}\otimes \pr_{12}^*\cP)^{1\times \sym_{n+\ell}\times 1}$.
The $\sym_\ell\times \sym_{n+\ell}\times \sym_\ell$-linearisation of 
\begin{align*}\pr_{13*}(\pr_{23}^*\cP^{iR}\otimes \pr_{12}^*\cP)\cong\bigoplus_{\Index(j)\times\Index(i)} \pr_{13*}\bigl(\pr_{23}^*\cP(I_2,J_2,\mu_2)^R\otimes \pr_{12}^*\cP(I_1,J_1,\mu_1)\bigr)\end{align*}
induces on the index set $\Index(j)\times\Index(i)$ the action 
\[\sigma_1\times \tau\times\sigma_2\colon \bigl(I_1,J_1,\mu_1;I_2,J_2,\mu_2\bigr)\mapsto\bigl(\sigma_1(I_1),\tau(J_1),\tau\circ \mu_1\circ \sigma_1^{-1};
 \sigma(I_2),\tau(J_2),\tau\circ \mu_2\circ \sigma_2^{-1}\bigr)\,.
\]
Let $O(i,j)$ be a set of representatives of the $1\times \sym_{n+\ell}\times 1$-orbits in $\Index(i)\times \Index(j)$. One can check that $O(i,j)$ is in bijection with
\begin{align*}\Index(i,j):=
\begin{Bmatrix}(I_1,K_1,I_2,K_2,\mu)&\mid I_1\subset K_1\subset[\ell]\supset K_2\supset I_2,\, |I_1|=j,|I_2|=i,\\&|K_1|=|K_2|,\, \mu\colon \bar K_1\to \bar K_2 \text{ bijection} \end{Bmatrix}.
\end{align*}
via the assignment
\[(I_1,J_1,\mu_1;I_2,J_2,\mu_2)\mapsto\bigl(I_1, K_1= I_1\cup\mu_1^{-1}(J_2), I_2, K_2= I_2\cup\mu_2^{-1}(J_1),\mu= \mu_{2|\overline{J_1\cup J_2}}^{-1}\circ \mu_{1|\bar K_1}\bigr)\,. 
\]
Furthermore, the $\sym_{n+\ell}$-stabiliser of $(I_1,J_1,\mu_1;I_2,J_2,\mu_2)$ is $\sym_{J_1\cap J_2}$.   
 It follows by (\ref{Danlem}) and (\ref{summandT}) that the equivariant convolution product $\cP^{iR}\star \cP^j$ is given by
\begin{align}
\notag \cP^{iR}\star\cP^j&\cong \pr_{13*}(\pr_{23}^*\cP^{iR}\otimes \pr_{12}^*\cP)^{1\times \sym_{n+\ell}\times 1}\\
\label{PiPjsumorbit}&\cong\bigoplus_{O(i,j)} \pr_{13*}\bigl(\pr_{23}^*\cP(I_2,J_2,\mu_2)^R\otimes \pr_{12}^*\cP(I_1,J_1,\mu_1)\bigr)^{1\times \sym_{J_1\cap J_2}\times 1}\\
\label{PiPjsum}&\cong\bigoplus_{\Index(i,j)}\reg_{K_1,K_2,\mu}\otimes\alt_{K_1\setminus I_1}\otimes \alt_{K_2\setminus I_2}\otimes  p^*\left(\Lambda_{n+i+j-k}^*(X)\otimes \omega_X^{-(k-j)}\right)[-(k-j)\cdot d]\,.
\end{align}
We denote the direct summands of (\ref{PiPjsum}) by $\cQ(I_1,K_1,I_2,K_2,\mu)$.
Note that the $\sym_\ell\times \sym_\ell$-linearisation on $\cP^{iR}\star \cP^j$ induces on $O(i,j)\cong \Index(i,j)$ the action
\[\sigma_1\times \sigma_2\colon(I_1, K_1,I_2,K_2,\mu)\mapsto (\sigma_1(I_1), \sigma_1(K_1),\sigma_2(I_2),\sigma_2(K_2),\sigma_2\circ \mu\circ\sigma_1^{-1} )\,. 
\]
The $\sym_\ell\times \sym_\ell$-stabiliser of $(I_1,K_1,I_2,K_2,\mu)$ is equal to $\sym_{I_1}\times \sym_{K_1\setminus I_1}\times \sym_{\bar K_1,\mu}\times \sym_{I_2}\times \sym_{K_2\setminus I_2}$. 
With this notation we indicate the subgroup of $\sym_\ell\times \sym_\ell$ given by
\[
 \bigl\{(\sigma_1,\sigma_2)\mid \sigma_1(I_1)=I_1,\,\sigma_1(K_1)=K_1,\,\sigma_2(I_2)=I_2,\, \sigma_2(K_2)=K_2,\, (\sigma_2\circ \mu)_{|\bar K_1}=(\mu\circ \sigma_1)_{|\bar K_1}   \bigr\}\,.
\]
Furthermore, the orbits of the $\sym_\ell\times \sym_\ell$-action on $\Index(i,j)$ are given by
\[\Index(i,j)_k:=\{|K_1|=|K_2|=k\}\subset \Index(i,j)\quad\text{for $k=\max\{i,j\},\dots,\ell$}\,.
\]
A representative of the orbit $\Index(i,j)_k$ is $([j],[k],[i],[k],e)$ where $e=\id_{[k+1,\ell]}$. 
We get
\begin{align}\label{PiPjInf}
\cP^{iR}\star\cP^{j}\cong\bigoplus_{k=\max\{i,j\}}^\ell \Inf_{\sym_j\times\sym_{k-j}\times \sym_{\ell-k,e}\times \sym_i\times \sym_{k-i}}^{\sym_\ell\times\sym_\ell}\cQ([j],[k],[i],[k],e)\,. 
\end{align}
We denote the direct summands by $\cP(i,j)_k:=\Inf_{\sym_j\times\sym_{k-j}\times \sym_{\ell-k,e}\times \sym_i\times \sym_{k-i}}^{\sym_\ell\times\sym_\ell}\cQ([j],[k],[i],[k],e)$.
\subsection{Spectral sequences}\label{specs}
For $j=0,\dots,\ell$ there is the spectral sequence $\cE(j)$ associated to the functor $\sHom(\_,\pr_{12}^*\cP^j)$ and the complex $\pr_{23}^*\cP$ given by
\[\cE(j)_1^{p,q}=\sExt^{q}(\pr_{23}^*\cP^{-p},\pr_{12}^*\cP^j)\,\Longrightarrow\, \cE(j)^{p+q}=\sExt^{p+q}(\pr_{23}^*\cP,\pr_{12}^*\cP^j)\,;\]
see e.g.\ \cite[Remark 2.67]{Huy}.
By Section \ref{directsummand}, every term of this spectral sequence is finitely supported over $X\times X^\ell\times X\times X^\ell$ hence $\pr_{13*}$-acyclic. Since the functors $(\_)^{\sym_{n+\ell}}$, and 
$(\_)\otimes \pr_2^*\omega_{X\times X^\ell}$ are exact, we can apply the functor $\pr_{13*}(\_)^{1\times \sym_{n+\ell}\times 1}\otimes \pr_2^*\omega_{X\times X^\ell}$ to every level of the spectral sequence $\cE(j)$ to get a spectral sequence 
in $\Coh_{\sym_\ell\times \sym_\ell}(X\times X^\ell\times X\times X^\ell)$. Shifting this spectral sequence by $(\ell+1)d$ in the $q$-direction we get the spectral sequence $E(j)$ with the property
\[E(j)_1^{p,q}=\Coho^q((\cP^{-p})^R\star\cP^j)\,\Longrightarrow\,E(j)^{p+q}=\Coho^{p+q}(\cP^R\star \cP^j)\,.\]
Similarly, we get a spectral sequence 
\begin{align}\label{specs2}E_1^{p,q}=\Coho^q(\cP^R\star \cP^p)\,\Longrightarrow\, E^{p+q}=\Coho^{p+q}(\cP^R\star\cP)\,.\end{align}
\subsection{Long exact sequences}
For $k\ge 1$ there is the long exact sequence of $\sym_k$-representations
\begin{align*}
 0\to \C\to \dots\to \Inf_{\sym_i\times \sym_{k-i}}^{\sym_k}\alt_i\xrightarrow{\check d^i}
 \Inf_{\sym_{i+1}\times \sym_{k-i-1}}^{\sym_k}\alt_{i+1}\to\dots\to \alt_k\to 0
\end{align*}
which we denote by $\check\cC_k^\bullet$. We consider $\check \cC_k^\bullet$ as a complex in degrees $[0,k]$.
The terms are 
\[
\check\cC_k^i= \Inf_{\sym_i\times \sym_{k-i}}^{\sym_k}\alt_i\cong \bigoplus_{I\subset[k]\,,\, |I|=i}\alt_I\,.
\]
Under this identification the differential $\check d^i$ is determined by its components $\alt_I\to \alt_J$ which are given by $\eps_{I,b}=(-1)^{\#\{a\in I\mid a<b\}}$ if $J=I\cup\{b\}$ and which are zero if $I\not\subset J$. 
That the sequence is exact can be checked either by hand or by considering it as a special case of a \v Cech complex.   
We also set $\hat \cC_k^\bullet:=\check \cC_k^\bullet\otimes \alt_k$. Then $\hat \cC_k^\bullet$ is the exact complex
\begin{align*}
 0\to \alt_k\to \dots\to \Inf_{\sym_i\times \sym_{k-i}}^{\sym_k}\alt_{k-i}\xrightarrow{\hat d^i}
 \Inf_{\sym_{i+1}\times \sym_{k-i-1}}^{\sym_k}\alt_{k-i-1}\to\dots\to \C\to 0\,.
\end{align*}
Let $M$ be a variety on which we consider $\sym_k$ to act trivially. For $E\in \Coh(M)$ we set $\check\cC_k^\bullet(E):=E\otimes_\C \check\cC_k^\bullet$. This is an exact complex in $\Coh_{\sym_k}(M)$ given by
\begin{align*}
 0\to E\to \dots\to \Inf_{\sym_i\times \sym_{k-i}}^{\sym_k}(E\otimes\alt_i)\xrightarrow{\check d^i(E)}
 \Inf_{\sym_{i+1}\times \sym_{k-i-1}}^{\sym_k}(E\otimes\alt_{i+1})\to\dots\to E\otimes \alt_k\to 0\,.
\end{align*}
There also is the exact complex $\hat \cC_k^\bullet(E):=E\otimes_\C\hat\cC_k^\bullet\cong \check\cC_k^\bullet(E)\otimes \alt_k$. 
\begin{lemma}
Let $E\in \Coh(M)$ be simple, i.e.\ $\Hom(E,E)=\C$. Then 
\[\Hom_{\sym_k}(\check \cC_k^i(E),\check \cC_k^{i+1}(E))\cong \C\cong \Hom_{\sym_k}(\hat \cC_k^i(E),\hat \cC_k^{i+1}(E))\,.\] 
\end{lemma}
\begin{proof}
 By the adjunction $(\Inf,\Res)$ we have
\begin{align*}\Hom_{\sym_k}(\check \cC_k^i(E),\check \cC_k^{i+1}(E))&\cong
 \bigl[\bigoplus_{|I|=i+1}\Hom(E\otimes \alt_{[i]}, E\otimes\alt_I)\bigr]^{\sym_i\times \sym_{k-i}}\\&\cong
 \bigl[\bigoplus_{|I|=i+1}\Hom(E , E)\otimes\alt_{[i]\setminus I}\otimes \alt_{I\setminus [i]}\bigr]^{\sym_i\times \sym_{k-i}}\,.\end{align*}
For $[i]\not\subset I$ we have $|I\setminus[i]|\ge 2$ and hence $(\Hom(E , E)\otimes\alt_{[i]\setminus I}\otimes \alt_{I\setminus [i]})^{\sym_{I\setminus[i]}}=0$. It follows by
Section \ref{Dansection} that 
\[\Hom_{\sym_k}(\check \cC_k^i(E),\check \cC_k^{i+1}(E))\cong \Hom(E\otimes \alt_{[i]}, E\otimes\alt_{[i+1]})^{\sym_{[i]}\times \sym_{[i+2,k]}}=\C\,.\]
The second assertion follows from $\hat \cC_k^\bullet(E)\cong \check\cC_k^\bullet(E)\otimes \alt_k$.
\end{proof}
\begin{cor}\label{dsufficient}
Let $E\in \Coh(M)$ be simple. Then, up to isomorphism, every non-zero $\sym_k$-equivariant morphism $\check \cC_k^i(E)\to \check\cC_k^{i+1}(E)$ equals $\check d^i(E)$ and every non-zero $\sym_k$-equivariant morphism $\hat \cC_k^i(E)\to\hat \cC_k^{i+1}(E)$ equals $\hat d^i(E)$.  
\end{cor}
\begin{conv}\label{C0}
We also define $\check \cC_0^\bullet=\C[0]=\hat \cC_0^\bullet$ to be the one-term complex with $\C$ in degree zero. Obviously, the complexes $\check\cC_0^\bullet$ and $\hat\cC_0^\bullet$ are not exact in contrast to the case $k\ge 1$ described above. 
\end{conv}
\subsection{The curve case: induced maps}\label{curveinduced}
We will need the following easy fact in the following.
\begin{lemma}\label{Homvanish}
Let $\iota_1\colon Z_1\to M$ and $\iota_2\colon Z_2\to M$ be closed embeddings of irreducible subvarieties. If $Z_1\not\supset Z_2$ we have $\Hom_{M}(\iota_{1*}L_1,\iota_{2*}L_2)=0$ for 
all $L_1\in \Pic(Z_1)$ and $L_2\in \Pic(Z_2)$. 
\end{lemma}
Let $X=C$ be a smooth curve. By Lemma \ref{Tinva} we have $\Lambda_m^*=\reg_C[0]$. Using the results of Section \ref{ifsub} we get for $k=\max\{i,j\},\dots,\ell$ isomorphisms  
\begin{align}
 \notag\cP(i,j)_k&\cong\Inf_{\sym_j\times\sym_{k-j}\times \sym_{\ell-k,e}\times \sym_i\times \sym_{k-i}}^{\sym_\ell\times\sym_\ell}\reg_{[k],[k],e}\otimes\alt_{[j+1,k]}\otimes \alt_{[i+1,k]}\otimes  p^* \omega_C^{-(k-j)}[-(k-j)]\\ \label{Pijkcurve}
&\cong\Inf_{\sym_j\times \sym_{k-j}\times \sym_{\ell-k,e}\times \sym_k}^{\sym_\ell\times\sym_\ell}\check\cC_k^{k-i}(\reg_{[k],[k],e}\otimes  p^* \omega_C^{-(k-j)})\otimes \alt_{[j+1,k]}[-(k-j)]\,.
\end{align}
The second isomorphism is due to the general fact that for subgroups $V\subset U\subset G$ of a finite group $G$ there is an isomorphism
of functors $\Inf_{V}^G\cong \Inf_U^G\circ \Inf_V^U$.  
\begin{lemma}
Let $\max\{i,j\}\le k\le \ell$.
The component \[\Coho^{k-j}(\cP^{iR}\star\cP^j)\cong\Coho^{k-j}(\cP(i,j)_k)\to\Coho^{k-j} (\cP(i-1,j)_k)\cong \Coho^{k-j}(\cP^{(i-1)R}\star\cP^j)\] of the morphism $\cP^{iR}\star\cP^j\to \cP^{(i-1)R}\star\cP^j$ that is induced by the differential $\cP^{i-1}\to \cP^i$ is given under the 
isomorphism (\ref{Pijkcurve}) by $\Inf_{\sym_j\times \sym_{k-j}\times \sym_{\ell-k,e}\times  \sym_{k}}^{\sym_\ell\times\sym_\ell}\check d^{k-i}(\reg_{[k],[k],e}\otimes  p^* \omega_C^{-(k-j)})$. 
\end{lemma}
\begin{proof}
Note that 
\[\check\cC_k^{k-i}(\reg_{[k],[k],e}\otimes  p^* \omega_C^{-(k-j)})\otimes \alt_{[j+1,k]}[-(k-j)]\cong \bigoplus_{I_2\subset [k],\,|I_2|=i}\cQ([j],[k],I_2,[k],e)\,.\]
By the previous lemma, all the components $\cQ([j],[k],I_2,[k],e)\to \cQ(I_1',K_1',I_2',K_2',\mu)$ of the induced map $\cP(i,j)_k\to \cP(i-1,j)_k$  are zero unless $K_1'=K_2'=[k]$ and $\mu=e$.
They are also zero for $I_1'\neq[j]$ since the map is induced by $\cP^{i-1}\to \cP^i$ and $\cQ(I_1',K_1',I_2',K_2',\mu)$ arises as
$\pr_{13*}\bigl(\pr_{23}^*\cP(I_2,J_2,\mu_2)^R\otimes \pr_{12}^*\cP(I_1,J_1,\mu_1)\bigr)^{1\times \sym_{J_1\cap J_2}\times 1}$; see 
(\ref{PiPjsumorbit})  and (\ref{PiPjsum}).
By the adjunction $(\Res,\Inf)$, it follows that the map $\Coho^{k-j}(\cP(i,j)_k)\to\Coho^{k-j} (\cP(i-1,j)_k)$ is of the form
$\Inf_{\sym_j\times \sym_{k-j}\times \sym_{\ell-k,e}\times  \sym_{k}}^{\sym_\ell\times\sym_\ell}(f)$ for some 
\[f\colon \check\cC_k^{k-i}(\reg_{[k],[k],e}\otimes  p^* \omega_C^{-(k-j)})\to \check\cC_k^{k-i+1}(\reg_{[k],[k],e}\otimes  p^* \omega_C^{-(k-j)})\,.\]  
Thus, by Corollary \ref{dsufficient} it is sufficient to show that the component \[\Coho^{k-j}\bigr(\cQ([j],[k],[i],[k],e)\bigl)\to \Coho^{k-j}\bigl(\cQ([j],[k],[i-1],[k],e)\bigr)\] of 
$\Coho^{k-j}(\cP(i,j)_k)\to\Coho^{k-j}(\cP(i-1,j)_k)$ is non-zero. By (\ref{PiPjsumorbit}) and (\ref{PiPjsum}) we have
\begin{align*}
 \cQ([j],[k],[i],[k],e)\cong \bigl[\cP\bigl([i],[n+i],e\bigr)^R\star\cP\bigl([j],J_1,e\bigr)\bigr]^{\sym_{n+i+j-k}}\,.
\end{align*}
where a possible choice of $J_1$ is $J_1=[n+i+j-k]\cup[n+i+1,n+k]$.
In degree $k-j$ the $\sym_{n+i+j-k}$-action on $\cP\bigl([i],[n+i],e\bigr)^R\star\cP\bigl([j],J_1,e\bigr)$ is trivial since given by the representation $\wedge^0\rho_{n+i+j-k}$; see (\ref{summandT}). Hence,
\begin{align}\label{term1}
 \Coho^{k-j}\Bigl(\cQ([j],[k],[i],[k],e)\Bigr)\cong \Coho^{k-j}\Bigl(\cP\bigl([i],[n+i],e\bigr)^R\star\cP\bigl([j],J_1,e\bigr)\Bigr)\,.
\end{align}
Analogously, we get 
\begin{align}\label{term2}
 \Coho^{k-j}\Bigl(\cQ([j],[k],[i-1],[k],e)\Bigr)\cong \Coho^{k-j}\Bigl(\cP\bigl([i-1],[2,n+i],e\bigr)^R\star\cP\bigl([j],J_1,e\bigr)\Bigr)\,.
\end{align} 
Under the isomorphisms (\ref{term1}) and (\ref{term2}), $\Coho^{k-j}\bigr(\cQ([j],[k],[i],[k],e)\bigl)\to \Coho^{k-j}\bigl(\cQ([j],[k],[i-1],[k],e)\bigr)$ corresponds to the map
\begin{align*}
\Coho^{k-j}\Bigl(\cP\bigl([i],[n+i],e\bigr)^R\star\cP\bigl([j],J_1,e\bigr)\Bigr)\to \Coho^{k-j}\Bigl(\cP\bigl([i-1],[2,n+i],e\bigr)^R\star\cP\bigl([j],J_1,e\bigr)\Bigr)  
\end{align*}
induced by the restriction $\reg_{[i-1],[2,n+i],e}\to \reg_{[i],[n+i],e}$. By Section \ref{inducedmapsummands}, it is an isomorphism.   
\end{proof}
\subsection{The curve case: fully faithfulness}
\begin{prop}\label{curvecoh} For $X=C$ a curve we have $\cP^R\star\cP\cong \Inf_{\sym_{\ell,e}}^{\sym_\ell,\sym_\ell}\reg_{\Delta_{C\times C^\ell}}$.
\end{prop}
\begin{proof}
We consider the spectral sequences $E(j)_1^{p,q}=\Coho^q((\cP^{-p})^R\star\cP^j)\,\Longrightarrow\,\Coho^{p+q}(\cP^R\star \cP^j)$; see Section \ref{specs}. By (\ref{Pijkcurve}) and the previous lemma, the $(k-j)$-th row of $E(j)_1$ for $k=j,\dots,\ell$ is given by  
the complex 
$\Inf_{\sym_j\times \sym_{k-j}\times \sym_{\ell-k,e}\times \sym_k}^{\sym_\ell\times\sym_\ell}\check\cC_k^{\bullet}(\reg_{[k],[k],e}\otimes  p^* \omega_X^{-(k-j)})\otimes \alt_{[j+1,k]}$ shifted into degrees $[-k,0]$. The inflation functor is exact. Thus, all the rows of the spectral sequences are 
exact with one exception: The zero row of $E(0)_1$ is given by the single non-zero object $E(0)_1^{0,0}=\Coho^0(\cP(0,0)_0)$. 
It follows that $\cP^R\star\cP^j=0$ for $j\ge 1$ and $\cP^R\star\cP^0\cong\Coho^0(\cP(0,0)_0)$. Now by the spectral sequence (\ref{specs2}) or, alternatively, by the fact that $\cP^R\star\cP$ is a left convolution of 
$\cP^R\star\cP^0\to\cP^R\star\cP^1\to \dots\to \cP^R\star\cP^\ell$ it follows that
$\cP^R\star\cP\cong\Coho^0(\cP(0,0)_0)\cong \Inf_{\sym_{\ell,e}}^{\sym_\ell,\sym_\ell}\reg_{\Delta_{C\times C^\ell}}$.  
\end{proof}
\begin{proof}[Proof of Proposition A(i)]
The identity functor $\id\colon \D^b_{\sym_\ell}(C\times C^\ell)\to \D^b_{\sym_\ell}(C\times C^\ell)$ equals the equivariant FM transform with kernel $\Inf_{\sym_{\ell\Delta}}^{\sym_\ell\times \sym_\ell} \Delta_{C\times C^\ell}$; see Section \ref{equiFM}. Note that we have $\sym_{\ell\Delta}=\sym_{\ell,e}\subset \sym_\ell\times \sym_\ell$. Thus, Proposition A(i) follows by the previous proposition.
\end{proof}
\subsection{The curve case: orthogonality}\label{cort}
In this section we will indicate the proof of Proposition A (ii).
Lemma \ref{o1} and \ref{l} which state formulae for the convolution products $\cP_{\ell,n}^{iR}\star \cP_{\ell',n'}^j$ for $n+\ell=n'+\ell'$ and the induced maps between them. The proofs, which are analogous to the computations of Sections \ref{directsummand}, \ref{inducedmapsummands}, \ref{ifsub}, and \ref{curveinduced}, are left to the reader. The author decided to carry out the computations of $\cP_{\ell,n}^{iR}\star \cP_{\ell',n'}^j$ in the Sections \ref{directsummand} and \ref{ifsub} only in the case $(\ell,n)=(\ell',n')$ in order to avoid the heavier notation. 
This case is sufficient for the proof of Proposition A (i) and Theorem C.
In particular, the reader mainly interested in Theorem C may skip the rest of the current subsection.

Let now $\ell,n,\ell',n'\in \Z$ be integers such that $n> \max\{1,\ell\}$, $n'> \max\{1,\ell'\}$, and $n+\ell=n'+\ell'$. 
For $I_1\subset K_1\subset[\ell]$, $I_2\subset K_2\subset [\ell']$, and $\mu\colon \bar K_1\to \bar K_2$ a bijection we consider the subvariety $\Gamma_{K_1,K_2,\mu}\subset X\times X^\ell\times X\times X^{\ell'}$ given by
\[\Gamma_{K_1,K_2,\mu}:=\bigl\{(x,x_1,\dots,x_\ell,z,z_1,\dots,z_{\ell'})\mid x=x_a=z_b\,\forall\, a\in K_1,b\in K_2\,,\,x_c=z_{\mu(c)}\,\forall\, c\in \bar K_1 \bigr\}\]
and set $\reg_{K_1,K_2,\mu}=\reg_{\Gamma_{K_1,K_2,\mu}}$ as well as 
\[\cQ(I_1,K_1,I_2,K_2,\mu):=\reg_{K_1,K_2,\mu}\otimes\alt_{K_1\setminus I_1}\otimes \alt_{K_2\setminus I_2}\otimes  p^*\left(\Lambda_{n'+i+j-k}^*(X)\otimes \omega_X^{-(k-j)}\right)[-(k-j) d]\,.\]
Again, $p\colon X\times X^\ell\times X\times X^{\ell'}\to X$ denotes the projection to the third factor.
For $0\le i\le \ell'$, $0\le j\le \ell$, and $\max\{n'-n+i,j\}\le k\le \ell$ we set
\[\cP(i,\ell',j,\ell)_k:=\Inf_{\sym_j\times\sym_{k-j}\times \sym_{\ell-k,e}\times \sym_i\times \sym_{k+n-n'-i}}^{\sym_\ell\times\sym_\ell}\cQ([j],[k],[i],[k+n-n'],e)\,.\]
\begin{lemma}\label{o1}
$\cP_{\ell',n'}^i\star\cP_{\ell,n}^j\cong \bigoplus_{k=\max\{n'-n+i,j\}}^{\ell}\cP(i,\ell',j,\ell)_k$. 
\end{lemma}
\begin{proof}
 This follows from computations analogous to those of Sections \ref{directsummand} and \ref{ifsub}. 
\end{proof}
\begin{lemma}\label{l}
 Let $X=C$ be a curve and $\ell'>\ell$. Then for $0\le j\le \ell$ and $j\le k\le \ell$ we have $\Coho^{k-j}(\cP_{\ell',n'}^{iR}\star \cP_{\ell,n}^j)=0$ for $i>k+n-n'$ and the sequence  
\[0\to \Coho^{k-j}(\cP_{\ell',n'}^{(k+n-n')R}\star \cP_{\ell,n}^j)\to\dots\to \Coho^{k-j}(\cP_{\ell',n'}^{0R}\star \cP_{\ell,n}^j)\to 0\,,\]
whose differentials are induced by the differentials $\cP^{i-1}\to \cP^i$, is isomorphic to 
\[\Inf_{\sym_j\times \sym_{k-j}\times \sym_{\ell-k,e}\times \sym_{k+n-n'}}^{\sym_\ell\times\sym_\ell}\check\cC_{k+n-n'}^{\bullet}(\reg_{[k],[k+n-n'],e}\otimes  p^* \omega_C^{-(k-j)})\otimes \alt_{[j+1,k]}\,.\]  
In particular, it is an exact sequence.
\end{lemma}
\begin{proof}
This follows from computations analogous to those of Sections \ref{inducedmapsummands} and \ref{curveinduced}. 
\end{proof}
\begin{prop}\label{orthker}
Let $X=C$ be a curve and $\ell'>\ell$. Then $\cP_{\ell',n'}^R\star\cP_{\ell,n}=0$. 
\end{prop}
\begin{proof}
 This follows from Lemma \ref{l} together with spectral sequences analogous to those of Section \ref{specs}.
\end{proof}
This proves Proposition A(ii).
\subsection{The surface case: induced maps}
Let $X$ be a smooth surface and still $n>\max\{\ell,1\}$.
We set \[\tilde S_X:=(\_)\otimes p^*\omega_X[2]\in \Aut(\D^b_{\sym_\ell\times\sym_\ell}(X\times X^\ell\times X\times X^\ell))\] and 
$\tilde S_X^{-[a,b]}:=\tilde S_X^{-a}\oplus \tilde S_X^{-(a+1)}\oplus\dots \oplus \tilde S_X^{-b}$ for $a\le b$ two integers. 
By Lemma \ref{Tinva} we have $p^*\Lambda_m^*(X)=\tilde S_X^{-[0,m-1]}(\reg_{X\times X^\ell\times X\times X^\ell})$. Hence, for $k=\max\{i,j\},\dots,\ell$ we get
\begin{align}
 \notag\cP(i,j)_k&\cong\tilde S_X^{-[k-j,n+i-1]}\Inf_{\sym_j\times\sym_{k-j}\times \sym_{\ell-k,e}\times \sym_i\times \sym_{k-i}}^{\sym_\ell\times\sym_\ell}\reg_{[k],[k],e}\otimes\alt_{[j+1,k]}\otimes \alt_{[i+1,k]}\\ \label{Pijksurface}
&\cong \tilde S_X^{-[k-j,n+i-1]} \Inf_{\sym_k\times \sym_{\ell-k,e}\times \sym_k}^{\sym_\ell\times\sym_\ell}\hat\cC_k^{j}\check\cC_k^{k-i}(\reg_{[k],[k],e})\,.
\end{align}
Note that the inner \v Cech complex is taken with respect to $\sym_k$ considered as a subgroup of $\sym_\ell\times \sym_\ell$ by the embedding into the second factor, while the outer \v Cech complex is taken with respect to $\sym_k$ considered as a subgroup of 
$\sym_\ell\times \sym_\ell$ by the embedding into the first factor.  
\begin{lemma}\label{itoi}
Let $k,k'\in[\max\{i,j\},\ell]$. The components $\Coho^{q}(\cP(i,j)_k)\to\Coho^{q} (\cP(i-1,j)_{k'})$ of the morphism $\Coho^{q}(\cP^{iR}\star\cP^j)\to \Coho^{q} (\cP^{(i-1)R}\star\cP^j)$ which is induced by the differential $\cP^{i-1}\to \cP^i$ are zero for 
$k\neq k'$.
For $k-j\le r\le n+i-2$ the component $\Coho^{2r}(\cP(i,j)_k)\to\Coho^{2r} (\cP(i-1,j)_k)$ is given by $\Inf_{\sym_k\times \sym_{\ell-k,e}\times \sym_k}^{\sym_\ell\times\sym_\ell}\hat\cC_k^{j}\check d_k^{k-i}(\reg_{[k],[k],e}\otimes p^*\omega_X^{-r})$. 
\end{lemma}
\begin{proof}
The components $\cP(I'_2,J'_2,\mu'_2)\to \cP(I_2,J_2,\mu_2)$ of $\cP^{i-1}\to \cP^{i}$ are non-vanishing only if $I_2'\subset I_2$. Thus, following the computations of Section \ref{ifsub}, the only components 
$\cP(i,j)_k\to \cP(i-1,j)_{k'}$ of $\cP^{iR}\star\cP^j\to \cP^{(i-1)R}\cP^j$ which are possibly non-zero are those with $k=k'$ or  $k-1=k'$. 
But $\Coho^q(\cP(i,j)_k)\to \Coho^{q}(\cP(i-1,j)_{k-1})$ is zero by Lemma \ref{Homvanish}. 

For the proof of the second assertion it is, as in the curve case, sufficient to show that  
\[\Coho^{2r}\bigr(\cQ([j],[k],[i],[k],e)\bigl)\to \Coho^{2r}\bigl(\cQ([j],[k],[i-1],[k],e)\bigr)\]
 is non-zero; see Corollary \ref{dsufficient}. By (\ref{PiPjsumorbit}) and (\ref{PiPjsum}) we have
\begin{align*}
 \Coho^{2r}\Bigl(\cQ([j],[k],[i],[k],e)\Bigr)\cong \Coho^{2r}\Bigl(\cP\bigl([i],[n+i],e\bigr)^R\star\cP\bigl([j],J_1,e\bigr)\Bigr)^{\sym_{[n+i+j-k]}}\,.
\end{align*}
where a possible choice of $J_1$ is $J_1=[n+i+j-k]\cup[n+i+1,n+k]$.
Also,
\begin{align*}
 &\Coho^{{2r}}\Bigl(\cQ([j],[k],[i-1],[k],e)\Bigr)\\\cong &\Coho^{{2r}}\Bigl(\cP\bigl([i-1],[2,n+i],e\bigr)^R\star\cP\bigl([j],J_1,e\bigr)\Bigr)^{\sym_{[2,n+i+j-k]}}\\
\cong&\Bigl[\bigoplus_{a\in [n+i+j-k]} \Coho^{{2r}}\bigl(\cP\bigl([i-1],[n+i]\setminus\{a\},e\bigr)^R\star\cP\bigl([j],J_1,e\bigr)\bigr)\Bigr]^{\sym_{[n+i+j-k]}}\,.
\end{align*}
The second isomorphism is due to Section \ref{Dansection}. 
The components of the induced map
\begin{align}\label{map1}&\Coho^{2r}\Bigl(\cP\bigl([i],[n+i],e\bigr)^R\star\cP\bigl([j],J_1,e\bigr)\Bigr)\\\notag\longrightarrow& \bigoplus_{a\in [n+i+j-k]}\Coho^{{2r}}\Bigl(\cP\bigl([i-1],[n+i]\setminus\{a\},e\bigr)^R\star\cP\bigl([j],J_1,e\bigr)\Bigr)\end{align}
are given under the isomorphism (\ref{summandT}) by the canonical surjections $\rho_{[n+i+j-k]}\to\rho_{[n+i+j-k]\setminus\{a\}}$; see Section \ref{inducedmapsummands}. 
It follows by \cite[Lemma B.6 (3)]{Sca1} that the map (\ref{map1}) induces an isomorphism on the $\sym_{[n+i+j-k]}$-invariants. 
\end{proof}
\begin{lemma}\label{jtoj}
Let $k,k'\in[\max\{i,j+1\}, \ell]$. The components $\Coho^{q}(\cP(i,j)_k)\to\Coho^{q} (\cP(i,j+1)_{k'})$ of the morphism $\Coho^{q}(\cP^{iR}\star\cP^j)\to \Coho^{q} (\cP^{iR}\star\cP^{j+1})$ which is induced by the differential $\cP^{j}\to \cP^{j+1}$ are zero 
for $k'\notin \{k,k+1\}$.
For $k-j\le r\le n+i-1$ the component $\Coho^{2r}(\cP(i,j)_k)\to\Coho^{2r} (\cP(i,j+1)_k)$ is given by $\Inf_{\sym_k\times \sym_{\ell-k,e}\times \sym_k}^{\sym_\ell\times\sym_\ell}\hat d^{j}_k\bigl(\check \cC_k^{k-i}(\reg_{[k],[k],e}\otimes p^*\omega_X^{-r})\bigr)$. 
\end{lemma}
\begin{proof}
The first assertion follows from the fact that the components $\cP(I_1,J_1,\mu_1)\to \cP(I'_1,J'_1,\mu'_1)$ of $\cP^{j}\to \cP^{j+1}$ are non-zero only if $I_1\subset I'_1$.

The proof of the second assertion is also analogous to the proof of the second assertion of Lemma \ref{itoi}:
It is sufficient to show the non-vanishing of 
\[\Coho^{{2r}}\bigr(\cQ([j],[k],[i],[k],e)\bigl)\to \Coho^{{2r}}\bigl(\cQ([j+1],[k],[i],[k],e)\bigr)\,.\]
Set $J_2:=[n+i+j+1-k]\cup[n+j+2,n+k]$. There are isomorphisms
\begin{align*}
 \Coho^{2r}\Bigl(\cQ([j+1],[k],[i],[k],e)\Bigr)\cong \Coho^{2r}\Bigl(\cP\bigl([i],J_2,e\bigr)^R\star\cP\bigl([j+1],[n+j+1],e\bigr)\Bigr)^{\sym_{[n+i+j+1-k]}}\,.
\end{align*}
and
\begin{align*}
 &\Coho^{{2r}}\Bigl(\cQ([j],[k],[j],[k],e)\Bigr)\\
\cong &\Bigl[\bigoplus_{b\in [n+i+j+1-k]} \Coho^{{2r}}\Bigl(\cP\bigl([i],J_2,e\bigr)^R\star\cP\bigl([j],[n+j+1]\setminus\{b\},e\bigr)\Bigr)\Bigr]^{\sym_{[n+i+j+1-k]}}\,.
\end{align*}
Under the isomorphism (\ref{summandO}), the components of the induced map
\begin{align}\label{map2}&\bigoplus_{b\in [n+i+j+1-k]} \Coho^{{2r}}\Bigl(\cP\bigl([i],J_2,e\bigr)^R\star\cP\bigl([j],[n+j+1]\setminus\{b\},e\bigr)\Bigr)\\\notag&\longrightarrow \Coho^{2r}\Bigl(\cP\bigl([i],J_2,e\bigr)^R\star\cP\bigl([j+1],[n+j+1],e\bigr)\Bigr)\end{align}
are given by the canonical injections $\rho_{[n+i+j+1-k]\setminus\{b\}}\to\rho_{[n+i+j+1-k]}$. 
It follows by \cite[Lemma B.6 (4)]{Sca1} that the map (\ref{map2}) induces an isomorphism on the $\sym_{[n+i+j+1-k]}$-invariants. 
\end{proof}
In fact, the component $\Coho^{{2r}}(\cP(i,j)_k)\to\Coho^{{2r}} (\cP(i,j+1)_{k+1})$ corresponds under (\ref{Pijksurface}) to the map induced by the restriction $\reg_{[k],[k],e}\to \reg_{[k+1],[k+1],e}$. But this will not be relevant for our purposes. 
\subsection{The surface case: cohomology}
Recall that for $0\le m<k$ there is the \textit{stupid truncation} $\sigma^{\le m}\check \cC_k^\bullet$ with
\[\sigma^{\le m}\check \cC_k^\bullet=(0\to \check \cC_k^0\to \dots\to \check \cC_k^m\to 0)\quad,\quad\Coho^\alpha(\sigma^{\le m}\check\cC_k^\bullet)=\begin{cases}
\coker{\check d^{m-1}} \quad&\text{for $\alpha=m$}\\
0\quad&\text{else}                                                                                                                   
                                                                                                                \end{cases}
\]
where for $m=0$ we have $\coker \check d^{-1}=\check \cC_k^0=\C$. 
For $\max\{i,j\}\le k\le \ell$ we set
\[\cR(i,j,k):=\Inf_{\sym_k\times \sym_{\ell-k,e}\times \sym_k}^{\sym_\ell\times \sym_\ell}\hat\cC_k^{j}\bigl(  
\cT(i,k) \bigr)\,,\,\, \cT(i,k):=\Coho^{k-i}\bigl(\sigma^{\le k-i}\check\cC_k^\bullet( \reg_{[k],[k],e}\otimes p^*\omega_X^{-(n+i-1)})\bigr).
\]
\begin{lemma}
For $0\le j\le \ell$ and $1\le i\le \ell$ we have 
$E(j)_2^{-i,2(n+i-1)}\cong\oplus_{k=\max\{i,j\}}^\ell \cR(i,j,k)$. For $j\ge 1$ these are the only non-vanishing terms on the 2-level of $E(j)$. For $j=0$ there are in addition the non-vanishing terms 
$E(0)_2^{0,2r}=\Inf_{\sym_{\ell,e}}^{\sym_\ell\times\sym_\ell}\reg_{\Delta_{X\times X^\ell}}\otimes p^*\omega_X^{-r}$ for $r=0,\dots,n-1$. 
\end{lemma}
\begin{proof}
 The terms $E(j)^{p,q}_1=\Coho^q(\cP^{-pR}\star\cP^j)$ are described by (\ref{PiPjInf}) together with (\ref{Pijksurface}). We see that the only non-vanishing rows on the 1-level of $E(j)$ have $q=2r$ for $r=k-j,\dots, n+\ell-1$. By Lemma \ref{itoi} for $i\ge 1$ the row
 $q=2(n+i-1)$ is the complex 
\begin{align}\label{qrow}&\sigma^{\le -i}\Bigl(\bigoplus_{k=\max\{i,j\}}^\ell\Inf_{\sym_k\times \sym_{\ell-k,e}\times \sym_k}^{\sym_\ell\times \sym_\ell}\hat\cC_k^{j}\bigl( \check\cC_k^\bullet( \reg_{[k],[k],e}\otimes p^*\omega_X^{-(n+i-1)}) \bigr)[k]     \Bigr)\\
\cong&\notag
\bigoplus_{k=\max\{i,j\}}^\ell\sigma^{\le k-i}\Bigl(\Inf_{\sym_k\times \sym_{\ell-k,e}\times \sym_k}^{\sym_\ell\times \sym_\ell}\hat\cC_k^{j}\bigl( \check\cC_k^\bullet( \reg_{[k],[k],e}\otimes p^*\omega_X^{-(n+i-1)}) \bigr)     \Bigr)[k]\,.
 \end{align} 
Since the functor $\Inf_{\sym_k\times \sym_{\ell-k,e}\times \sym_k}^{\sym_\ell\times \sym_\ell}\hat\cC_k^{j}(\_)$ is exact, the cohomology of the row $q=2(n+i-1)$ is concentrated in degree $-i$ and equal to $\oplus_{k=\max\{i,j\}}^\ell\cR(i,j,k)$.
This proofs the first assertion.

For $r=k-j,\dots,n-1$ the row $q=2r$ is given by \[\bigoplus_{k=j}^\ell\Inf_{\sym_k\times \sym_{\ell-k,e}\times \sym_k}^{\sym_\ell\times \sym_\ell}\hat\cC_k^{j}\bigl( \check\cC_k^\bullet( \reg_{[k],[k],e}\otimes p^*\omega_X^{-r}) \bigr)[k]\,.\] Thus, it is an exact complex with one exception: In the case $j=0$ the one-term complex $\Inf_{\sym_{\ell,e}}^{\sym_\ell\times \sym_\ell}\hat\cC_0^{0}
\bigl( \check\cC_0^\bullet( \reg_{\emptyset,\emptyset,e}\otimes p^*\omega_X^{-r}) \bigr)[0]\cong \Inf_{\sym_{\ell,e}}^{\sym_\ell\times\sym_\ell}\reg_{\Delta_{X\times X^\ell}}\otimes p^*\omega_X^{-r}[0]$ occurs as a direct summand; compare Convention \ref{C0}.   
\end{proof}
\begin{cor}\label{HqPP}
For $j=1,\dots \ell$ we have
\[\Coho^q(\cP^R\star\cP^j)=\begin{cases}
                           \oplus_{k=\max\{i,j\}}^\ell\cR(i,j,k)\quad&\text{for $q=2(n-1)+i$, $i=1,\dots,\ell$}\\
0\quad&\text{else.}
                          \end{cases}
\]
Furthermore, 
\[\Coho^q(\cP^R\star\cP^0)=\begin{cases}
                           \oplus_{k=\max\{i,j\}}^{\ell}\cR(i,0,k)\quad&\text{for $q=2(n-1)+i$, $i=1,\dots,\ell$}\\
\Inf_{\sym_{\ell,e}}^{\sym_\ell\times\sym_\ell}\reg_{\Delta_{X\times X^\ell}}\otimes p^*\omega_X^{-r}\quad&\text{for $q=2r$, $r=0,\dots,n-1$}\\
0\quad&\text{else.}
                          \end{cases}
\]
\end{cor}
\begin{proof}
By the positioning of the non-vanishing terms, we see that all the $E(j)$ degenerate at the 2-level. The result follows since $E(j)^q=\Coho^q(\cP^R\star\cP^j)$; see Section \ref{specs}.  
\end{proof}
\begin{lemma}\label{exactsum}
 Let $\cA$ be an abelian category and $(C_\alpha^\bullet,d_\alpha)$ be complexes in $\cA$ for $\alpha=1,\dots,m$. Let $C^\bullet$ be a complex with terms $C^j=C_1^j\oplus\dots\oplus C_m^j$ and differentials of the form 
 \[d^j=\begin{pmatrix}
  d_1^j & 0 &\cdots &0\\
*&d_2^j&\cdots&0\\
\vdots&\ddots&\ddots&\vdots\\
*&\cdots&*&d_m^j 
  \end{pmatrix}
\]
where the stars stand for arbitrary morphisms. Then, if all the $C_\alpha^\bullet$ are exact, $C^\bullet$ is exact too.
\end{lemma}
\begin{proof}
Let $B^\bullet$ be the complex with terms $B^j=C_1^j\oplus \dots\oplus C_{m-1}^j$ and differentials 
\[d_B^j=\begin{pmatrix}
  d_1^j & 0 &\cdots &0\\
*&d_2^j&\cdots&0\\
\vdots&\ddots&\ddots&\vdots\\
*&\cdots&*&d_{m-1}^j\,. 
  \end{pmatrix}
\]
By induction, we can assume that $B^\bullet$ is an exact complex. There is the short exact sequence of complexes $0\to C_m^\bullet\to C^\bullet\to B^\bullet$ where the first map is given by the inclusion of the last direct summand and the second map is the 
projection to the first $m-1$ direct summands. The exactness of $C^\bullet$ follows from the associated long exact cohomology sequence.
\end{proof}
For $r\in \Z$ we set $\bar \cS_X^r:=\tilde S_X^r\bigr(\Inf_{\sym_{\ell,e}}^{\sym_\ell\times\sym_\ell}\reg_{\Delta_{X\times X^\ell}}\bigl)=\Inf_{\sym_{\ell,e}}^{\sym_\ell\times\sym_\ell}\reg_{\Delta_{X\times X^\ell}}\otimes p^*\omega_X^{r}[2r]$.
We have $\bar \cS_X^r=(\bar\cS_X^1)^{\star r}$ and $\FM_{\bar\cS_X^r}=\bar S_X^r=(\_)\otimes (\omega_X^r\boxtimes \reg_{X^\ell})[2r]\colon \D^b_{\sym_\ell}(X\times X^\ell)\to \D^b_{\sym_\ell}(X\times X^\ell)$.
\begin{prop}\label{surfacecoh}
 $\Coho^*(\cP^R\star\cP)=\bar \cS_X^{-[0,n-1]}:=\bar\cS_X^0\oplus \bar\cS_X^{-1}\oplus\dots\oplus \bar\cS_X^{-(n-1)}$.
\end{prop}
\begin{proof}
 We consider the spectral sequence $E_1^{p,q}=\Coho^q(\cP^R\star \cP^p)\,\Longrightarrow\, \Coho^{p+q}(\cP^R\star\cP)$; see (\ref{specs2}). By Corollary \ref{HqPP} the only non-vanishing rows of $E_1$ are 
\[q=0,2,\dots,2(n-1), 2(n-1)+1,\dots,2(n-1)+\ell\,.\]
Note that the terms of the row $q=2(n-1)+i$ for $i=1,\dots,\ell$ equal those of the exact complex
$\bigoplus_{k=i}^\ell \Inf_{\sym_k\times \sym_{\ell-k,e}\times \sym_k}^{\sym_\ell\times \sym_\ell}\hat\cC_k^{\bullet}(\cT(i,k))$.
We set $\cR(i,j,k)=0$ and $\hat d_k^j=0$ for $j>k$. By Lemma \ref{jtoj} the map $d^j\colon E_1^{j,q}=\oplus_{k=i}^\ell\cR(i,j,k)\to E_1^{j+1,q}=\oplus_{k=i}^\ell\cR(i,j+1,k)$ is given by
\[d^j= \begin{pmatrix}
  d_i^j & 0 &\cdots &0\\
*& d_{i+1}^j&\cdots&0\\
0&\ddots&\ddots&\vdots\\
0&0&*&d_\ell^j 
  \end{pmatrix}
\quad,\quad d_k^j=\hat d_k^j(\cT(i,k))\,.\]
It follows by Lemma \ref{exactsum} that the row $q=2(n-1)+i$ is exact for all $i=1,\dots,\ell$. 
For $r=0,\dots,n-1$, the row $q=2r$ has only one non-vanishing term, namely $E_1^{0,2r}=\bar \cS_X^{-r}$.
In summary, the only non-vanishing terms on the 2-level are $E_2^{0,2r}=\bar \cS_X^{-r}$ for $r=0,\dots,n-1$ from which the proposition follows. 
\end{proof}
\subsection{The surface case: splitting and monad structure}\label{surfsplitmon}
\begin{lemma}\label{qiso}
 Let $\cA$ be an abelian category with enough injectives and $A^\bullet,B^\bullet,C^\bullet\in \D^b(\cA)$ together with morphisms $f\colon A^\bullet \to C^\bullet$, $g\colon B^\bullet\to C^\bullet$ in $\D^b(\cA)$. Let there be an $m\in \Z$ such that the cohomology 
of $A^\bullet$ and $B^\bullet$ is concentrated in degrees smaller than $m$ and such that $\Coho^i(f)$ as well as $\Coho^i(g)$ are isomorphisms for all $i<m$. Then $A^\bullet\cong B^\bullet$ in $\D^b(\cA)$.
\end{lemma}
\begin{proof}
We may assume that $A^i=B^i=0$ for all $i\ge m$; see \cite[Exercise 2.31]{Huy}.
Choose an injective complex $I^\bullet$ which is quasi-isomorphic to $C^\bullet$. Then $f$ and $g$ are represented by morphisms of complexes $f^\bullet\colon A^\bullet\to I^\bullet$ and $g^\bullet\colon B^\bullet\to I^\bullet$; see \cite[Lemma 2.39]{Huy}.
These morphisms factor through the smart truncation $\tau^{\le m-1}I^\bullet$ and by the hypothesis these factorisations are quasi-isomorphisms. Thus, they are isomorphisms in $\D^b(\cA)$ which proves the assertion.   
\end{proof}
\begin{prop}\label{PP}
$\cP^R\star\cP\cong \bar S_X^{-[0,n-1]}$.
\end{prop}
\begin{proof}
 This follows by applying the previous lemma to the situation that $f\colon \cP^R\star \cP\to \cP^R\star \cP^0$ is the map induced by the canonical map $\cP\to \cP^0$ and $g$ is the composition \[\bar S_X^{-[0,n-1]}\to \cP^{0R}\star\cP^0\to \cP^{R}\star\cP^0\] 
where the first map is the inclusion of the direct summand $\cP(0,0)_0\cong\bar S_X^{-[0,n-1]}$  under the isomorphism (\ref{PiPjInf}) and the second map is induced by $\cP\to \cP^0$. 
\end{proof}
\begin{proof}[Proof of Theorem C]
By the previous proposition, $H_{\ell,n}$ fulfils condition (i) of a $\P^{n-1}$-functor.

Set $F_{\ell,n}:=\delta_{[n]*}\circ \MM_{\alt_n}\circ \triv\colon \D^b_{\sym_\ell}(X\times X^\ell)\to \D^b_{\sym_n\times \sym_\ell}(X^n\times X^\ell)$; compare (\ref{H0}). 
By (\ref{GRG}) we have $F_{\ell,n}^RF_{\ell,n}=\bar S_X^{-[0,n-1]}$. 
The unit of adjunction $\eta\colon \id\to \Res\circ \Inf$ gives a map of monads $F_{\ell,n}^R\eta F_{\ell,n}\colon F_{\ell,n}^R F_{\ell,n}\to H_{\ell,n}^{0R}H_{\ell,n}^0$.
On the level of the FM kernels it coincides with the inclusion $\bar S_X^{-[0,n-1]}\to \cP^{0R}\star\cP^0$.
Since $F_{\ell,n}=F_{0,n}\boxtimes \id_{\D^b_{\sym_\ell}(X^\ell)}$, the monad multiplication 
$\mu(F_{\ell,n})\colon F_{\ell,n}^RF_{\ell,n}F_{\ell,n}^RF_{\ell,n}\to F_{\ell,n}^RF_{\ell,n}$ equals $\mu(F_{0,n})\boxtimes \id$. By \cite{Kru3}, $F_{0,n}$ is a $\P^{n-1}$-functor. In particular, the monad structure of $F_{0,n}$ has the right shape, i.e.\
the components $S_X^{-1}\circ S_X^{-k}\to S_X^{-(k+1)}$ of $\mu(F_{0,n})$ are isomorphisms for $k=0,\dots,n-2$. Thus, also the components $\bar S_X^{-1}\circ \bar S_X^{-k}\to \bar S_X^{-(k+1)}$ of $\mu(F_{\ell,n})$ are isomorphisms for 
$k=0,\dots,n-2$.  
Equivalently, on the level of the FM kernels the components $\bar \cS_X^{-1}\circ \bar \cS_X^{-k}\to \bar \cS_X^{-(k+1)}$ of the monad multiplication 
\[\cP(\emptyset,[n],e)^R\star \cP(\emptyset,[n],e)\star \cP(\emptyset,[n],e)^R\star \cP(\emptyset,[n],e)\to \cP(\emptyset,[n],e)^R\star \cP(\emptyset,[n],e)\,,\] which we denote again by $\mu(F_{\ell,n})$, are isomorphisms.
Let 
$U:=(X\times X^\ell)\setminus(\cup_{\emptyset\neq I\subset [\ell]} D_{I})$ and $u\colon U\to X\times X^{\ell}$ the open embedding. Then $H_{\ell,n}\circ u_*\cong H^0_{\ell,n}\circ u_*$. It follows that the components 
$\bar \cS_X^{-1}\circ \bar \cS_X^{-k}\to \bar \cS_X^{-(k+1)}$ of $\mu(H_{\ell,n})\colon \cP^R\star \cP\star \cP^R\star \cP\to \cP^R\star \cP$ are isomorphisms over $U\times U$. 
Since the $\cS_X^{-k}[2k]$ are direct sums of line bundles on the graphs of the $\sym_\ell$-action on $X\times X^\ell$ and the codimension of the complement of $U$ in $X\times X^\ell$ is 2, it follows that the components $\bar \cS_X^{-1}\circ \bar \cS_X^{-k}\to \bar \cS_X^{-(k+1)}$ of $\mu(H_{\ell,n})$ are isomorphisms over the whole
 $X\times X^\ell\times X\times X^\ell$ which amounts to condition (ii) of a $\P^{n-1}$-functor.

That the $H_{\ell,n}$ satisfy condition (iii) of a $\P^{n-1}$-functor was already shown in Section \ref{cond3}.       
\end{proof}

\subsection{The case of the generalised Kummer stacks}\label{Kummercase}
Let $X=A$ be an abelian variety of dimension $d$. For $\ell,n\in \N$ with $n>\max\{\ell,1\}$ we set
\begin{align*}M_{\ell,n}A :=\{(a,a_1,\dots,a_\ell)\mid n\cdot a+a_1+\dots+a_\ell=0\}\subset A\times A^\ell,,\\
N_{n+\ell-1}:=\{(b_1,\dots,b_{n+\ell})\mid b_1+\dots+b_{n+\ell}=0\}\subset A^{n+\ell}\,.\end{align*}
Note that these subvarieties are invariant under the $\sym_\ell$-action on $A\times A^\ell$ and the $\sym_{n+\ell}$-action on $A^{n+\ell}$, respectively.
For $(I,J,\mu)\in \Index(i)$ we set $\hat \Gamma_{I,J,\mu}:=\Gamma_{I,J,\mu}\cap(M_{\ell,n}A\times N_{n+\ell-1}A)$ where 
$\Gamma_{I,J,\mu}\subset A\times A^\ell\times A^{n+\ell}$ is the subvariety described in Section \ref{cP}.
Furthermore, let $\hat \cP(I,J,\mu):=\reg_{\hat \Gamma_{I,J,\mu}}\otimes \alt_J$ and
\begin{align*}\hat \cP^i:=\Inf_{\sym_{i}\times \sym_{\ell-i,e}\times \sym_{n+i}}^{\sym_{\ell}\times \sym_{n+\ell}} \hat \cP([i],[n+i],e)=\bigoplus_{\Index(i)}\hat \cP(I,J,\mu)\in \Coh_{\sym_\ell\times \sym_{n+\ell}}(M_{\ell,n}A\times N_{n+\ell-1}A)\,.\end{align*}
We set $\hat H_{\ell,n}:=\FM_{\hat \cP}\colon\D^b_{\sym_\ell}(M_{\ell,n}A)\to \D^b_{\sym_{n+\ell}}(N_{n+\ell-1}A)$ where $\hat \cP$ is the complex 
\begin{align}\label{hatPcomplex}\hat \cP:=\hat\cP_{\ell,n}=\bigl(0\to \hat \cP^0\to \dots\to \hat \cP^\ell\to 0\bigr)\,.\end{align}
For $I\subset [n+\ell]$ with $|I|=n$, the morphism $\delta_I\colon A\times A^\ell\to A^{n+\ell}$ restricts to a morphism 
$\hat \delta_I\colon M_{\ell,n}A\to N_{n+\ell-1} A$. The functor $\hat H_{\ell,n}^0$ is given by the composition
\begin{align}\label{hat0}
 \D^b_{\sym_\ell}(M_{\ell,n}A)\xrightarrow{\MM_{\alt_n}\circ \triv} \D^b_{\sym_n\times \sym_\ell}(M_{\ell,n}A)\xrightarrow{\hat \delta_{[n]*}}\D^b_{\sym_n\times \sym_\ell}(N_{n+\ell-1}A)\xrightarrow\Inf \D^b_{\sym_{n+\ell}}(N_{n+\ell-1}A)\,.
\end{align}
\begin{remark}
 Let $\iota \colon M_{\ell,n}\to A\times A^\ell$ and $\iota'\colon N_{n+\ell-1}A\to A^{n+\ell}$ denote the close embeddings. We have $\hat \cP(I,J,\mu)\not\cong (\iota\times \iota')^*\cP(I,J,\mu)$ where $(\iota\times \iota')^*$ denotes the derived pull-back (note that the equality holds if we consider the non-derived pull-back instead). The reason is that $\Gamma_{I,J,\mu}$ and 
$M_{\ell,n}A\times N_{n+\ell-1}A$ do not intersect transversally inside of $A\times A^\ell\times A^{n+\ell}$. Indeed, 
\[\codim\bigl(M_{\ell,n}A\times N_{n+\ell-1}A\hookrightarrow A\times A^\ell\times A^{n+\ell}\bigr)=4d\quad,\quad \codim\bigl(\hat \Gamma_{I,J,\mu}\hookrightarrow \Gamma_{I,J,\mu}  \bigr)=2d\,.\]
\end{remark}
\begin{lemma}\label{cross} 
 We have $(\iota \times \id)_*\hat\cP^i\cong (\id\times\iota')^*\cP^i$ and $(\id \times \iota')_*\hat\cP^i\cong (\iota\times\id)^*\cP^i$ for every $i=0,\dots,\ell$. Also, $(\iota \times \id)_*\hat\cP\cong (\id\times\iota')^*\cP$ and 
$(\id \times \iota')_*\hat\cP\cong (\iota\times\id)^*\cP$. 
\end{lemma}
\begin{proof}
For every $0\le i\le \ell$ and $(I,J,\mu)\in \Index(i)$ the two diagrams of closed embeddings 
\[
\begin{CD}
\hat\Gamma_{I,J,\mu}
@>{}>>
\Gamma_{I,J,\mu} \\
@V{}VV
@V{}VV \\
A\times A^\ell\times N_{n+\ell-1}A
@>{}>>
A\times A^\ell\times A^{n+\ell}
\end{CD} 
\quad,\quad
\begin{CD}
\hat\Gamma_{I,J,\mu}
@>{}>>
\Gamma_{I,J,\mu} \\
@V{}VV
@V{}VV \\
M_{\ell,n}A\times A^{n+\ell}
@>>>
A\times A^\ell\times A^{n+\ell}
\end{CD}
\]
are transversal intersections. It follows from the base change \cite[Corollary 2.27]{Kuz} that
\[(\iota \times \id)_*\hat\cP(I,J,\mu)\cong (\id\times\iota')^*\cP(I,J,\mu)\quad,\quad (\id \times \iota')_*\hat\cP(I,J,\mu)\cong (\iota\times\id)^*\cP(I,J,\mu)\]
for all $(I,J,\mu)\in \Index(i)$.
The result follows from the definition of $\cP^i$ as a direct sum of the $\cP(I,J,\mu)$ and (\ref{hatPcomplex}).
\end{proof}
\begin{cor}
 The functor $\hat H_{\ell,n}$ is a restriction of $H_{\ell,n}$ in the sense that $\hat H_{\ell,n}\circ \iota^*\cong \iota'^*\circ H_{\ell,n}$ and $\iota'_*\circ\hat H_{\ell,n}\cong H_{\ell,n}\circ\iota_*$. 
\end{cor}
\begin{prop}\label{Kummercoh}
\[
 \Coho^*(\hat \cP^R\star \hat\cP)\cong\begin{cases}
\Inf_{\sym_{\ell,e}}^{\sym_\ell\times \sym_\ell}\reg_\Delta[0]\quad&\text{for $A=E$ an elliptic curve,}\\
\Inf_{\sym_{\ell,e}}^{\sym_\ell\times \sym_\ell}\reg_\Delta([0]\oplus[-2]\oplus\dots\oplus[-2(n-1)])\quad&\text{for $A$ an abelian surface.}
                                                          \end{cases}
\]
Furthermore, for $\ell',n'$ with $n'>\ell'>\ell$ and $n'+\ell'=n+\ell$ we have $\Coho^*(\hat \cP_{\ell',n'}^R\star\hat \cP_{\ell,n})=0$ in the case of an elliptic curve.
\end{prop}
\begin{proof}
 It follows from the previous corollary taking right adjoints that 
\begin{align*}\iota_*\circ \hat H_{\ell,n}^R\circ \hat H_{\ell,n}\cong  H_{\ell,n}^R\circ  H_{\ell,n}\circ \iota_*\,.\end{align*}
The FM kernel of the left hand side is $(\id\times\iota)_*(\hat \cP^R\star\hat \cP)\in \D^b_{\sym_\ell\times \sym_\ell}(M_{\ell,n}A\times A\times A^\ell)$ and the FM kernel of the right hand side is 
\[(\iota\times \id)^*(\cP^R\star \cP)\cong\begin{cases}
                                          \Inf_{\sym_{\ell,e}}^{\sym_\ell\times \sym_\ell}\reg_{\Gamma_\iota}[0]\quad&\text{for $\dim A=1$,}\\
\Inf_{\sym_{\ell,e}}^{\sym_\ell\times \sym_\ell}\reg_{\Gamma_\iota}([0]\oplus[-2]\oplus\dots\oplus[-2(n-1)])\quad&\text{for $\dim A=2$.} 
                                          \end{cases}
\]
This follows from the Propositions \ref{curvecoh} and \ref{PP} together with the fact that 
\[
\begin{CD}
\Gamma_{\iota}
@>{}>>
\Delta_{A\times A^\ell} \\
@V{}VV
@V{}VV \\
M_{\ell,n}A\times A\times A^\ell
@>{}>>
A\times A^\ell\times A\times A^\ell
\end{CD} 
\]
is a transversal intersection. Since $(\id\times \iota)_*$ is an exact functor on the level of coherent sheaves,  $\Coho^*\bigl((\id\times \iota)_*(\hat \cP^R\star\hat \cP)\bigr)\cong (\id\times \iota)_*\Coho^*\bigl(\hat \cP^R\star\Hat \cP\bigr)$.
The formulae for $\Coho^*(\hat \cP^R\star \hat\cP)$ follow from the uniqueness of the cohomology of FM kernels; see \cite[Theorem 1.2]{CSuniqueness}. 
Analogously, the vanishing of $\Coho^*(\hat \cP_{\ell',n'}^R\star \hat \cP_{\ell,n})$ in the case of a curve follows from Proposition \ref{orthker}. 
\end{proof}
The case of an elliptic curve proves Proposition A'.
\begin{proof}[Proof of Theorem C']
Let $A$ be an abelian surface. 
We need to compute one component of $\hat \cP^{0R}\star\hat \cP^0$, namely $\hat \cP(\emptyset,[n],e)^R\star\hat \cP(\emptyset,[n],e)$.  
 There is the diagram 
\begin{align}\xymatrix{
            & M_{\ell,n}A\times \hat \Gamma_{\emptyset,[n],e} \ar^{\hat\iota_2}[dr]\ar^{\hat r}[d]& & \\
   \hat T \ar_{\hat\pi}[dd]\ar^{\hat u}[ur]\ar_{\hat v}[dr]   & M_{\ell,n}A\times \hat \Delta_{{[n]}}\times M_{\ell,n}A\ar^{\hat t}[r] &  M_{\ell,n}A\times N_{n+\ell-1}A\times M_{\ell,n}A\ar^{\hat \pr_{13}}[dd] &   \\
        &   \hat\Gamma_{\emptyset,[n],e}\times M_{\ell,n}A\ar^{\hat s}[u]\ar_{\hat \iota_1}[ur]    & &\\
\Delta_{M_{\ell,n}A}\ar[rr]&  & M_{\ell,n}A\times M_{\ell,n}A & 
} \end{align} 
where $\hat T=(M_{\ell,n}A\times \hat \Gamma_{\emptyset,[n],e})\cap ( \hat \Gamma_{\emptyset,[n],e}\times M_{\ell,n}A)$ and $\hat \Delta_{[n]}:=\Delta_{[n]}\cap N_{n+\ell-1}A$. The upper part is a diagram of closed embeddings with the same properties as 
diagram (\ref{selfintdiag}). Furthermore the restriction $\hat \pi$ of the projection $\hat \pr_{13}$ is an isomorphism onto the diagonal and $\codim\hat r=\codim \hat v=2\ell =\dim M_{\ell,n}A$. Thus, analogously to Sections \ref{directsummand} and \ref{ifsub}, one computes
\[\bigl[\hat \pr_{13}\bigl(\hat\pr_{23}^* \hat \cP(\emptyset,[n],e)^R\otimes \hat \pr_{12}^*\hat \cP(\emptyset,[n],e)\bigr)\bigr]^{\sym_{[n]}}\cong \reg_{\Delta_{M_{\ell,n}A}}([0]\oplus[-2]\oplus \dots\oplus [-2(n-1)])\,.\]
Thus, $\hat\cP(0,0)_0:=\Inf_{\sym_{\ell,e}}^{\sym_\ell,\sym_\ell}\reg_\Delta([0]\oplus[-2]\oplus\dots\oplus[-2(n-1)])$ occurs as a direct summand of $\hat \cP^{0R}\star \hat \cP^0$. By the same arguments as in the proof of Proposition \ref{Kummercoh} one can deduce 
from Lemma \ref{cross} that $\Coho^*(\hat \cP^{iR}\star\hat\cP^j)\cong (\iota\times\iota)^*\Coho^*(\cP^{iR}\star\cP^j)$ and $\Coho^*(\hat \cP^{R}\star\hat\cP^j)\cong (\iota\times\iota)^*\Coho^*(\cP^{R}\star\cP^j)$ for all $i,j\in[ \ell]$. 
Here for once $(\iota\times\iota)^*$ denotes the non-derived pull-back. It follows that the composition $\hat \cP(0,0)_0\hookrightarrow \hat\cP^{0R}\star\hat\cP^0\to \hat\cP^{R}\star\hat\cP^0$ induces an isomorphism on the cohomology in the degrees $\le 2(n-1)$. 
The same holds for the morphism $\hat \cP^R\star\hat\cP\to \hat \cP^R\star\hat\cP^0$.
Thus, we can apply Lemma \ref{qiso} to conclude that $\hat\cP^R\star\hat \cP\cong \hat\cP(0,0)_0$. It follows that 
\[\hat H_{\ell,n}^R\circ\hat H_{\ell,n}=\id\oplus[-2]\oplus  \dots\oplus[-2(n-1)]\]
which proves condition (i) of a $\P^{n-1}$-functor.    

We set $\hat F_{\ell,n}:=\hat \delta_{[n]*}\circ \MM_{\alt_{[n]}}\circ \triv$; compare (\ref{hat0}). Since the diagram 
\[
\begin{CD}
M_{\ell,n}A
@>{\hat\delta_{[n]}}>>
N_{n+\ell-1}A \\
@V{\iota}VV
@V{\iota'}VV \\
A\times A^\ell
@>{\delta_{[n]}}>>
A^{n+\ell}
\end{CD} 
\]
is a transversal intersection, we have $N_{\hat\delta_{[n]}}\cong N_{\delta_{[n]}| M_{\ell,n}A}$ as $\sym_n$-bundles. Thus, one can show following \cite[Section 3]{Kru3} that $\hat F_{\ell,n}$ is a $\P^{n-1}$-functor with cotwist $[-2]$. 
Now condition (ii) of a $\P^{n-1}$-functor for $\hat H_{\ell,n}$ can be deduced the same way as it was done for $H_{\ell,n}$ in Section \ref{surfsplitmon}.

Since the $\P$-cotwist as well as the Serre functors are just shifts, condition (iii) can be verified by a simple dimension count.
\end{proof}
\section{Interpretation of the results}
\subsection{Spherical and $\P$-twists}\label{SPtwists}
The \textit{$\P$-twist} associated to a $\P^n$-functor $F\colon \D^b_G(M)\to\D^b_H(N)$ with $\P$-cotwist $D$ is defined as the double cone 
\begin{align}\label{dcone}P_F:=\cone\left(\cone(FDF^R\to FF^R)\to \id  \right)\,.\end{align}
The map defining the inner cone is given by the composition 
\[FDF^R\xrightarrow{FjF^R}FF^RFF^R\xrightarrow{\eps FF^R-FF^R\eps}FF^R\]
where $j$ is the inclusion of the direct summand $D$. The map defining the outer cone is induced by the counit
$\eps\colon FF^R\to \id$; for details see \cite[Section 3.3]{Add}. The functor $P_F\colon \D^b_H(N)\to \D^b_H(N)$ is always an autoequivalence; see \cite[Theorem 3]{Add}. On the spanning class $\im F\cup \ker F^R=\im F\cup (\im F)^\perp$ the twist $P_F$ is given by
\begin{equation}\label{Ptwist}P_F\circ F\cong F\circ D^{n+1}[2]\quad, \quad P_F(B)=B\quad\text{for $B\in \ker F^R$.}\end{equation} 
For $S\colon \D^b_G(M)\to\D^b_H(N)$ a spherical functor the associated \textit{spherical twist} is defined as the cone $T_S:=\cone(S\circ S^R\xrightarrow{\eps} \id)$ of the counit. We have
\begin{equation}\label{Stwist}T_S\circ F\cong F\circ C[1]\quad, \quad T_S(B)=B\quad\text{for $B\in \ker S^R$.}\end{equation} 
In the case that $S$ is split spherical, i.e.\ a $\P^1$-functor, $T_S^2\cong P_S$. Let $\Psi\in \Aut(\D^b_G(M))$  and $\Phi\in \Aut(\D^b_H(N))$. Then 
\begin{align}\label{twistrel}T_{S \circ\psi}\cong T_S\quad,\quad P_{F\circ \Psi}\cong P_F\quad,\quad T_{\Phi\circ S}\cong \Phi\circ T_S\circ \Phi^{-1}\quad,\quad P_{\Phi\circ F}=\Phi\circ P_F\circ\Phi^{-1}\,;\end{align}
see \cite[Proposition 13]{AA} and \cite[Lemma 2.3]{Kru3}.
\subsection{The case $n=1$ and comparison to \cite{CL}}\label{CL}
Regarding (\ref{H0}) it is a natural extension to the case $n=1$ to set 
\[H_{\ell,1}^0:=\Inf_{\sym_\ell}^{\sym_{\ell+1}}\colon \D^b_{\sym_\ell}(X\times X^\ell)\to \D^b_{\sym_{\ell+1}}(X^{\ell+1})\,.\]
While the functors $H_{\ell,n}$ for $n\ge 2$ are $\P^{n-1}$-functors (in the surface case), the functor $H_{\ell,1}^0$ is a $\P^{\ell}$-functor (for $\dim X$ arbitrary) as follows from the following general observation.

Let $G$ be a finite group and $H\le G$ a subgroup such that there is an element $g\in G$ of order $n=[G:H]$ such that $1,g,\dots,g^{n-1}$ forms a system of representatives of the right cosets.
Let $G$ act on a variety $M$.
Recall that, in this case, the inflation functor is given  by
\begin{align}\label{infdef}\Inf:=\Inf_H^G\colon \D^b_H(M)\to \D^b_G(M)\quad,\quad \Inf(A)=\bigoplus_{k=0}^{n-1}g^{k*} A 
\end{align}
with the linearisation given by permutation of the summands. 
\begin{lemma}\label{infspherical}
The inflation functor $\Inf$ is a $\P^{n-1}$-functor with $\P$-cotwist $g^*$. 
\end{lemma}
\begin{proof}
The left and right adjoint of $\Inf$ is the restriction functor $\Res$. By (\ref{infdef}) we indeed have $\Res\circ \I^nf=\id\oplus g^*\dots\oplus g^{(n-1)*}$.
Condition (iii) of a $\P^{n-1}$-functor amounts to the fact that $\Res\cong g^{(n-1)*}\Res$.
For $(B,\lambda)\in \D^b_G(M)$ the counit map
$\eps\colon \Inf\circ \Res(B)=\bigoplus_{k=0}^{n-1}g^{k*} B\to B$ is given by the components $\lambda_{g^k}^{-1}\colon g^{k*}B\to B$; compare \cite[Section 3]{ElaCatlin}. Using this, one can compute that the monad structure has the desired form.  
\end{proof}
However, the induced twists are not very interesting.
\begin{lemma}\label{Inftwist}
 For $n=[G:H]=2$ the spherical twist $T_F$ associated to $F=\Inf$ equals the autoequivalence $M_{\alt}[1]:=(\_)\otimes \alt[1]$. For arbitrary $n\ge 2$ the $\P$-twist $P_F$ equals $[2]$. 
\end{lemma}
\begin{proof}
 Let $n=2$ and $(B,\lambda)\in \D^b_G(M)$. Let $\phi\colon B\otimes \alt\to \Inf\Res(B)$ be the $G$-equivariant morphism with components $\id\colon B\to B$ and $-\lambda_g\colon B\to g^*B$. This gives the exact triangle
\[B\otimes \alt\xrightarrow{\phi} \Inf\Res(B)\xrightarrow\eps B\,.\]
Since $T_F$ is defined by fitting into the exact triangle $\Inf\Res(B)\xrightarrow\eps B\to T_F$, it follows that 
$T_F\cong M_{\alt}[1]$. Still for $n=2$ the assertion for the $\P$-twist follows by the fact that $P_F=T_F^2$.
For $n\ge 3$ one has to do calculations involving the double cone construction (\ref{dcone}) of $P_F$ in order to show the assertion. We omit them.
\end{proof}
Analogously to  Section \ref{simNaka}, for $E\in \D^b(X)$ we consider 
\[
 H_{\ell,1}^0(E):=H_{\ell,1}^0\circ I_E\colon \D^b_{\sym_\ell}(X^\ell)\to \D^b_{\sym_{\ell+1}}(X^{\ell+1})\quad,\quad A\mapsto \Inf_{\sym_\ell}^{\sym_{\ell+1}}(E\boxtimes A)\,.
\]
Let now $X$ be a minimal resolution of the Kleinian singularity $\C^2/\Gamma$ for $\Gamma\subset \SL(2,\C)$ a finite subgroup so that we are in the situation of the classical McKay correspondence. 
Hence there is an equivalence $\D^b(X)\cong \D^b_\Gamma(\C^2)$. For $V_i$ an irreducible representation of $\Gamma$, consider $E_i=\C(0)\otimes V_i\in \D^b_\Gamma(\C^2)$ which corresponds to the line bundle $\reg(-1)$ on a component of the exceptional divisor of $X$. 
Then the $H_{\ell,1}^0(E_i)$ are exactly the functors $P_i(\ell)$ of \cite{CL} which give rise to a categorical action of the Heisenberg algebra on the derived categories of the Hilbert schemes of points on $X$. For higher $n$ the construction in \cite{CL} does not 
give explicit lifts of the Nakajima operators $q_{n}$. Instead, functors $P_i^{(n)}$ are constructed which 
correspond to other generators of the Heisenberg algebra than the Nakajima operators do; see \cite[Section 8.2]{CL}. The functors $P_i^{(n)}(\ell)$ are given in terms of our notation by
\[
 P_i^{(n)}(\ell)\colon \D^b_{\sym_\ell}(X^\ell)\to \D^b_{\sym_{n+\ell}}(X^{n+\ell})\quad,\quad A\mapsto \Inf_{\sym_\ell\times \sym_n}^{\sym_{n+\ell}}(E_i^{\boxtimes \ell}\boxtimes A)
\]
which is a direct summand of $H_{n+\ell-1,1}^0(E_i)\circ \dots\circ H_{1+\ell,1}^0(E_i) \circ H_{\ell,1}^0(E_i)$. In contrast,
$H_{\ell,n}^0(E)$ is given by $A\mapsto \Inf_{\sym_\ell\times \sym_n}^{\sym_{n+\ell}}(E_{\Delta}\boxtimes A)$ where $E_\Delta$ denotes the push-forward of $E$ along the small diagonal of $X^\ell$. 
%
%
%
%
\subsection{Induced autoequivalences on the Hilbert schemes}\label{Hilbauto}
Let $X$ be a smooth projective surface and $m\ge 2$. 
We will mostly omit the Bridgeland--King--Reid--Haiman equivalence $\Phi_m\colon \D^b(X^{[m]})\xrightarrow\cong \D^b_{\sym_m}(X^m)$ 
in the notation and interpret every functor between the equivariant derived categories of the cartesian product as one between the derived categories of the Hilbert schemes and vice versa. 
For $m$ even we set $r=\frac m2-1$ and for $m$ odd we set $r=\frac{m-1}2$.
By Theorem C there are the $\P^{m-\ell-1}$-functors $H_{\ell,m-\ell}\colon \D^b(X\times X^{[\ell]})\to \D^b(X^{[m]})$ for $\ell=0,\dots,r$. We denote the associated $\P$-twists by $P_{\ell,m-\ell}:=P_{H_{\ell,m-\ell}}\in \Aut(\D^b(X^{[m]}))$. 
Recall that the group of \textit{standard autoequivalences}
\[\Aut(\D^b(X^{[m]}))\supset \Aut^{st}(\D^b(X^{[m]}))\cong \Z\times \left(\Aut(X^{[m]})\ltimes \Pic(X^{[m]})\right)\] 
is the subgroup spanned by shifts, push-forwards along automorphisms and tensor products by line bundles.
We consider the group 
\[\Aut^{st+H}(\D^b(X^{[m]})):=\langle \Aut^{st}(\D^b(X^{[m]})),P_{0,m},\dots,P_{r,m-r} \rangle \subset \Aut(\D^b(X^{[m]}))\,.\]
The functors $H_{0,2}$ and $H_{1,2}$ are $\P^1$-functors, hence spherical. Accordingly, in the cases $m=2,3$ we replace $P_{0,2}$ and $P_{1,2}$ by their square roots $T_{0,2}:=T_{H_{0,2}}$ and $T_{1,2}:=T_{H_{1,2}}$ in the definition of $\Aut^{st+H}(\D^b(X^{[m]}))$. 

Let $\pi\colon X^m\to S^mX:=X^m/\sym_m$ be the quotient map. 
We write points of the symmetric product as formal sums of points of $X$.
For $\nu=(\nu_1,\dots,\nu_{s})$ a partition of $m$ there is the stratum
\[X^m_\nu:=\bigl\{x\in X^m\mid \pi(x)=\nu_1\cdot y_1+\dots+\nu_{s}\cdot y_{s}\text{ with pairwise distinct $y_i\in X$}\bigr\}\subset X^m\,.\]      
We denote the complement of its closure by $\bar X^m_{\nu}$. 
 For $1\le k\le m$ we set $\nu(k):=(k,1,\dots,1)$ as a partition of $m$. 
Note that, for $r< k\le m$ we have 
\begin{align}\label{stratum}
X^m_{\nu(k)}=\bigl(\bigcup_{I\subset [m]\,,\, |I|=k}\Delta_I\bigr)\setminus \bigl(\bigcup_{I\subset [m]\,,\, |I|=k+1}\Delta_I\bigr)\quad,\quad \bar X^m_{\nu(k)}=X^m\setminus \bigl(\bigcup_{I\subset [m]\,,\, |I|=k}\Delta_I\bigr)\,.
\end{align}
For $x\in X^m$ we denote by $\orb(x)\subset X^m$ the orbit of $x$ under the $\sym_m$-action on $X^m$. We set $\bar\C(x):=\reg_{\orb(x)}\otimes \alt_m\in \D^b_{\sym_m}(X^m)$.
\begin{lemma}\label{skyscrapervalue}
For $r\le \ell\le m$ we have 
\[
P_{\ell,m-\ell}(\bar\C(x))\cong\begin{cases}
                                \bar \C(x)[-2(m-\ell-1)]\quad &\text{for $x\in X^m_{\nu(m-\ell)}$,}\\
\bar \C(x)\quad&\text{for $x\in \bar X^m_{\nu(m-\ell)}$.}
                               \end{cases}
\]
Also,
\[
 T_{0,2}(\bar \C(x))\cong\begin{cases}
                         \bar \C(x)[-1]\,\,&\text{for $x\in \Delta$,}\\ 
                         \bar \C(x)\,\,& \text{for $x\in X^2\setminus \Delta$,}
                         \end{cases}
\quad, \quad
T_{1,2}(\bar \C(x))\cong\begin{cases}
                         \bar \C(x)[-1]\,\,&\text{for $x\in X^3_{\nu(2)}$,} \\
                         \bar \C(x)\,\,& \text{for $x\in \bar X^3_{\nu(2)}$.}
                         \end{cases}
\]
\end{lemma}
\begin{proof}
Every $x\in X^m_{\nu(m-\ell)}$ has a point of the from $y=(y_1,\dots,y_\ell,y,\dots,y)$ in its $\sym_m$-orbit. Then $H_{\ell,m-\ell}(\C(y,y_1,\dots,y_\ell))\cong\bar \C(y)\cong \bar\C(x)$. By (\ref{stratum}) it also follows that 
$H_{\ell,m-\ell}^R(\bar \C(x))=0$ for $x\in \bar X^m_{\nu(m-\ell)}$. The assertion for the $\P$-twist follows by (\ref{Ptwist}) and that for the spherical twists by (\ref{Stwist}).
\end{proof}
Note that $X^m_{\nu(m-\ell)}\cup \bar X^m_{\nu(m-\ell)}\subsetneq X^m$ so that the above does not describe the value of the $P_{\ell,m-\ell}$ on all skyscraper sheaves $\bar \C(x)$ of orbits. In fact, for $x\in X^m\setminus (X^m_{\nu(m-\ell)}\cup \bar X^m_{\nu(m-\ell)})$ the object 
$P_{\ell,m-\ell}(\bar\C(x))$ is again supported on $\orb(x)$ but not simply a shift of $\bar \C(x)$.  

\begin{prop}\label{abel}
The abelianisation of the group $\Aut^{st+H}(\D^b(X^{[m]}))$ is given by
\[
 \Aut^{st+H}(\D^b(X^{[m]}))_{ab}\cong \Z\times \bigl(\Aut(X^{[m]})\ltimes \Pic(X^{[m]})\bigr)_{ab}\times \Z^{r+1}\,.
\]
\end{prop}
\begin{proof}
Let $\Psi=\MM_\cL\circ \phi_*\circ P_{0,m}^{a_0}\circ\dots\circ P_{r,m-r}^{a_r}[b]$ for $\cL\in \Pic(X^{[m]})$, $\phi\in \Aut(X^{[m]})$, and $a_0,\dots,a_r,b\in \Z$. We have to show that $\Psi=\id$ implies $\cL\cong \reg_X$, $\phi=\id$, and $a_0=\dots=a_r=b=0$. 
Let $\xi=\{x_1,\dots,x_m\}\subset X$ be a reduced subscheme of length $m$. Under the BKRH equivalence $\C([\xi])$ corresponds to $\bar \C(x)$ with $x=(x_1,\dots,x_m)$. Since $x\in \bar X^m_{\nu(k)}$ for all $k\ge 2$ we get by the previous lemma 
$\Psi(\C([\xi]))=\C([\phi(\xi)])[b]$ which implies $b=0$ and $\phi=\id$. Let now $x\in X^m_{m-r}$ and $A=\Phi_m^{-1}(\bar\C(x))$. Again by the previous lemma, $\Psi(A)=\cL\otimes A[-2(m-r-1)a_r]$ which shows $a_r=0$. Testing inductively the values of 
$\Psi(\C(x))$ for $x\in X^m_{\nu(m-\ell)}$ shows that also $a_{r-1}=\dots=a_0=0$ hence $\Psi=\MM_\cL$. Finally $\Psi(\reg_{X^{[m]}})=\cL$ shows $\cL\cong \reg_{X^{[m]}}$. The proof goes through in the same way in the cases $m=2,3$ where $P_{0,2}$ and $P_{1,2}$ have to be replaced by 
$T_{0,2}$ and $T_{1,2}$.   
\end{proof}
\begin{remark}
Even before taking the abelianisation, the $P_{\ell,m-\ell}$ commute with a large class of autoequivalences namely with those induced by automorphisms and line bundles on the surface $X$. For the proof in the case $\ell=0$ see \cite[Lemma 5.4]{Kru3}. 
The proof for arbitrary $\ell$ is similar. 
\end{remark}
\begin{lemma}
Every autoequivalence in $\Aut^{st+H}(\D^b(X^{[m]}))$ is rank-preserving up to sign and sends objects that are supported on a non-trivial subset of $X^{[m]}$ to objects with the same property. 
\end{lemma}
\begin{proof}
For $\ell=0,\dots,r$ and every object $A\in \D^b_{\sym_\ell}(X\times X^\ell)$ the object $H_{\ell,m-\ell}(A)$ is supported on the closure of $X^m_{\nu(\ell)}$ hence has rank zero. 
By the double cone construction (\ref{dcone}) of the $\P$-twist it follows that $\rank P_{\ell,m-\ell}(B)=\rank B$ for all $B\in \D^b(X^{[m]})$. Also every standard autoequivalence is rank preserving with the exception of odd shifts which multiply the rank by $-1$.
The argument for the second assertion is similar.
\end{proof}
There are plenty of examples, especially in the case that $X$ is a K3 or Enriques surface, of autoequivalences in $\D^b(X^{[m]})$ which are not rank preserving or send skyscraper sheaves to objects whose support is the whole $X^{[m]}$. By the previous lemma we know that these autoequivalences are not contained in $\Aut^{st+H}(\D^b(X^{[m]}))$.
Here is a list of those examples the author is aware of:
\begin{enumerate}
 \item For every non rank-preserving autoequivalence $\Psi\in \Aut(\D^b(X))$ on the surface,
 Ploog's construction \cite[Section 3.1]{Plo} gives a non rank-preserving autoequivalence $\Psi^{[n]}\in \Aut(\D^b(X^{[m]}))$.
 \item For $X$ a K3 surface, the universal ideal sheaf functor $F=\FM_{\I_\Xi}\colon \D^b(X)\to \D^b(X^{[m]})$ is a $\P^{m-1}$-functor and the induced twist $P_F$ sends skyscraper sheaves to objects whose support is $X^{[m]}$; see \cite{Add} and \cite[Lemma 5.7]{Kru3}.
 \item There is a variant $G\colon \D^b(X)\to \D^b(X^{[m]})$ of $F$ called the truncated universal ideal sheaf functor; see \cite{KSos} or Section \ref{trunca}. For $X$ a K3 surface, it is again a $\P^{m-1}$-functor that sends skyscraper sheaves to objects whose support is $X^{[m]}$.
  \item For $X$ an Enriques surface, the universal ideal sheaf functor $F=\FM_{\I_\Xi}\colon \D^b(X)\to \D^b(X^{[m]})$ is neither spherical nor a $\P$-functor but fully faithful. Nevertheless, there is an induced autoequivalence which is 
not rank-preserving; see \cite{KSos}.
  \item For $X$ an Enriques surface also the truncated universal ideal functor is fully faithful and induces another non rank-preserving autoequivalence.
 \item For $X$ a K3 surface, the structural sheaf $\reg_{X^{[m]}}$ is a $\P^m$-object and the induced twist $P_\reg$ sends skyscraper sheaves to objects whose support is $X^{[m]}$. The same holds for twists along general $\P^n$-objects which are supported on the whole $X^{[m]}$. Spherical objects in $\D^b(X)$ induce 
$\P^m$-objects in $\D^b(X^{[m]})$; see \cite{PS}. 
 \item For $X$ an Enriques surface, the structural sheaf $\reg_{X^{[m]}}$ is an exceptional object and induces an autoequivalence which is not rank-preserving. The same holds for any exceptional object of non-zero rank. Exceptional objects in $\D^b(X)$ induce 
exceptional objects in $\D^b(X^{[m]})$; see \cite{KSos}. 
\item For $X=A$ an abelian surface the pull-back $\Sigma^*\colon \D^b(A)\to\D^b_{\sym_m}(A^m)\cong \D^b(A^{[m]})$ along the summation morphism is a $\P^{m-1}$-functor with $\P$-cotwist $[-2]$; see \cite{Mea}. Let $\Delta\subset A^m$ be the small diagonal. Then $\supp(P_{\Sigma^*}(\reg_\Delta))=A^m$. Indeed, the composition $\Sigma\circ \delta\colon A\to A$ with the embedding of the small diagonal is multiplication by $m$.
 Hence, $\Sigma_*\reg_{\Delta}$ is a locally free sheaf of rank $m$.
Thus, also $\Sigma^*\Sigma_*(\reg_\Delta)$ is a locally free sheaf of rank $m$ and the claim follows by
by the double cone construction (\ref{dcone}) of the $\P$-twist.
\end{enumerate}
\subsection{Induced autoequivalences on the Kummer varieties} 
Recall that for $A$ an abelian surface there is the BKRH equivalence $\hat \Phi\colon \D^b(K_{m-1}A)\xrightarrow \cong \D^b_{\sym_m}(N_{m-1}A)$. Thus the $\P$-functors of Theorem C' can be considered as $\P$-functors to the derived category of the generalised Kummer variety.
For the induced twists $\hat P_{\ell,m-\ell}:=P_{\hat H_{\ell,m-\ell}}$ the picture is very similar to the Hilbert scheme case with one exception:
The $\P^{m-1}$-functor $\hat H_{0,m}\colon \D^b(A_m)\to \D^b_{\sym_m}(N_{m-1}A)$ splits into the $\P^{m-1}$-objects $\bar \C(a,\dots,a)=\C(a,\dots,a)\otimes \alt_m\in \D^b_{\sym_m}(N_{m-1}A)$ for $a\in A_m$; see \cite[Section 6]{Kru3}. Thus, it gives 
rise to $m^4$ different autoequivalences $\hat P_a:=P_{\bar \C(a,\dots,a)}$ such that $\prod_{a\in A_m}\hat P_a\cong\hat P_{0,m}$. 
We define $\Aut^{st+H}(\D^b(K_{m-1}A))\subset \Aut(\D^b(K_{m-1}A))$ by
\[\Aut^{st+H}(\D^b(K_{m-1}A)):=\langle \Aut^{st}(\D^b(K_{m-1}A)),\hat P_{a}\,:\, a\in A_m,\hat P_{1,m-1},\dots,\hat P_{r,m-r} \rangle \,.\]
Analogous to Proposition \ref{abel} we get
\begin{align}\label{Kummerabel}
 \Aut^{st+H}(\D^b(K_{m-1}A))_{ab}\cong \Z\times \bigl(\Aut(K_{m-1}A)\ltimes \Pic(K_{m-1}A)\bigr)_{ab}\times\Z^{m^4}\times \Z^{r}\,.
\end{align}
Again, all the autoequivalences in the group $\Aut^{st+H}(\D^b(K_{m-1}A))$ are rank-preserving up to sign and preserve objects with non-trivial support. Thus, neither the $\P$-twist induced by the universal ideal functor (see \cite{Mea}) nor twists induced by $\P^{m-1}$-objects whose support is the whole $K_{m-1}A$, for example $\reg_{K_{m-1}A}$, are contained in $\Aut^{st+H}(\D^b(\K_{m-1}A))$.
\subsection{Relation to the (truncated) universal ideal functors}\label{trunca}
Let $X$ be a smooth quasi-projective surface and $F=F_m=\FM_{\I_\Xi}\colon \D^b(X)\to \D^b(X^{[m]})$ the universal ideal functor.
It follows from results of \cite{Sca1} that $\Phi\circ F=\FM_{\cK^{}}$ where 
$\cK^{}\in \D^b_{\sym_m}(X\times X^m)$ is the complex concentrated in degrees $[0,m]$ given by
\[0\to \reg_{X\times X^m}\to \bigoplus\limits_{i=0}^m\reg_{D_i}\to \bigoplus_{|I|=2}\reg_{D_I}\otimes \alt_I\to \bigoplus_{|I|=3}\reg_{D_I}\otimes \alt_I \dots \to \reg_{D_{[m]}}\otimes \alt_{[m]}\to 0\,;\]
see \cite[Section 6]{Mea}. We define the \textit{truncated universal ideal functor} as \[G:=G_m:=\FM_{\sigma^{\le1}\cK^{}}\colon \D^b(X)\to \D^b_{\sym_m}(X^m)\quad,\quad \sigma^{\le 1}\cK^{}=(0\to \reg_{X\times X^m}\to \oplus_{i=0}^m\reg_{D_i}\to 0)\,.\] 
We have $G^RG\cong F^RF$; see \cite[Section 5]{KSos}. It follows that $G$ is again a $\P^{m-1}$-functor if $X$ is a K3 surface and fully faithful if $X$ is an Enriques surface or more generally a surface with $p_g=0=q$. 
In the following we perform the computations at the level of the functors since we find it a bit easier and more intuitive.
However, in order to show that the induced maps between the functors take the expected form one has to do the calculation at the level of the FM kernels what we omit.  
We have
\begin{align}&\FM_{\cK^0}\cong\Ho^*(\_)\otimes \reg_{X^m}\quad,\quad \FM_{\cK^1}\cong\Inf_{\sym_{m-1}}^{\sym_m}\circ\pr_1\circ\triv\,,\\
\label{ge2} &\FM_{\cK^q}\cong \Inf_{\sym_q\times \sym_{m-q}}^{\sym_m}\circ \delta_{[q]*}\circ p_{[q]}^*\circ \MM_{\alt_q}\circ\triv\quad\text{for $q=2,\dots,m$}\end{align}
where for $I\subset [m]$ we denote by $p_I\colon \Delta_I\cong X\times X^{\bar I}\to X$ the projection to the first factor.
Consider $\pr_X^*\circ \triv\colon \D^b(X)\to \D^b_{\sym_{m-2}}(X\times X^{m-2})$.
\begin{lemma}
$H_{m-2,2}\circ \pr_X^*\circ \triv\cong \FM_{\sigma^{\ge 2}\cK^{}[2]}$. 
\end{lemma}
\begin{proof}
Under the isomorphism $\Delta_{[2]}\cong X\times X^{m-2}$ we have $\pr_X=p_{[2]}$. Moreover, for $i\ge 1$ under the identification
$\Delta_{[i+2]}\cong X\times X^{m-i-2}$ we have $\pr_X\circ \iota_{[i]}=p_{[i+2]}$. Now it follows from comparing (\ref{Hi}) and  (\ref{ge2}) that $H_{m-2,2}^i\circ \pr_X^*\circ \triv \cong \FM_{\cK^{i+2}}$.     
\end{proof}
\begin{lemma}\label{HLG}
 $H_{m-2,2}^L\circ G\cong \pr_X^*\circ \triv[-1]$. 
\end{lemma}
\begin{proof}
We set $\ell=m-2$. 
Note that the left adjoints $H_{\ell,2}^{iL}$ are given by (\ref{HiR}) with $\iota_{[i]*}$ replaced by $\iota_{[i]!}$ and $\delta_{[2+i]}^!$ replaced by $\delta_{[2+i]}^*$. We have
 \[(\_)^{\sym_{2+i}} \MM_{\alt_{2+i}} \delta_{[2+i]}^*\Res(\reg_{X^m})\cong (\reg_{X\times X^{\ell-i}}\otimes \alt_{2+i})^{\sym_{2+i}}=0\] which shows that $H_{2,m-2}^{iL}\circ \FM_{\cK^0}=0$ for all $0\le i\le m-2$. Let $q_a\colon X\times X^{\ell-i}\to X$ 
denote the projection to the $a$-th factor of $X^{\ell-i}$. For $E\in \D^b(X)$ we have 
\begin{align*}&(\_)^{\sym_{2+i}} \MM_{\alt_{n+i}} \delta_{[2+i]}^*\Res(\oplus_{i=1}^m \pr_i^*E)\\
\cong &\oplus_{a=1}^{\ell-i}(\alt_{2+i}\otimes q_i^*E)^{\sym_{2+i}} \oplus (E \otimes \alt_{2+i-1}\boxtimes\reg_{X^{\ell-i}} )^{\sym_{2+i-1}}
\cong \begin{cases} 0 \quad&\text{for $i> 0$,}\\
\pr_X^*E\quad&\text{for $i=0$.}       
      \end{cases}
\end{align*}
This shows that $H_{m-2,2}^L\circ G\cong H_{m-2,2}^L\circ \FM_{\cK^1}[-1]\cong H_{m-2,2}^{0L}\circ \FM_{\cK^1}[-1]=\pr_X\circ \triv[-1]$. 
\end{proof}
For $m\ge 4$ the functors $H_{m-2,2}$ are not spherical. Nevertheless, one can consider the twist $T_{m-2,2}:=\cone(H_{m-2,2}H_{m-2,2}^R\to \id)$ and its left adjoint 
$T_{m-2,2}^L=\cone(\id \to H_{m-2,2}H_{m-2,2}^L)$ though they are not autoequivalences for $m\ge 4$.
\begin{prop}\label{truncatedcomparison}
 $T_{m-2,2}^L\circ G\cong F$ for $m\ge2$.
\end{prop}
\begin{proof}
The previous two lemmas combined give $H_{m-2,2}\circ H_{m-2,2}^L\circ G\cong \FM_{\sigma^{\ge 2}\cK^{}[1]}$. It follows that
$T_{m-2,2}^L\circ G\cong \cone(\FM_{\sigma^{\le 1}\cK^{}}\to \FM_{\sigma^{\ge 2}\cK^{}[1]})\cong \FM_{\cK^{}}\cong F$.
\end{proof}
In the cases $m=2,3$ the twists $T_{m-2,2}$ and $T_{m-2,2}^L=T_{m-2,2}^{-1}$ are autoequivalences. Thus, the above proposition
gives another proof that $G^RG\cong F^RF$ in these cases. Furthermore, by (\ref{twistrel}) we get  
\begin{cor}
For $X$ a K3 surface, we have the relations $T_{0,2}T_{F_2}=T_{G_2}T_{0,2}$ and $T_{1,2}P_{F_3}=P_{G_3}T_{1,2}$ in $\Aut(\D^b(X^{[2]}))$ and $\Aut(\D^b(X^{[3]}))$, respectively. 
\end{cor}
We get similar relations for $X$ an Enriques surface by \cite[Remark 3.12]{KSos}.
Let now $A$ be an abelian surface and $\hat F_m=\hat F=\FM_{\I_{\hat \Xi}}\colon \D^b(A)\to \D^b(K_{m-1}A)$ the universal ideal functor. 
It is a $\P^{m-2}$-functor for $m\ge 3$ and a $\P^1$-functor for $m=2$; see \cite{Mea} and \cite{KMea}.
Then $\hat \Phi\circ \hat F=\FM_{\hat \cK}$ where $\hat \cK$ is the complex
\[\reg_{A\times N_{m-1}A}\to \bigoplus\limits_{i=0}^m\reg_{\hat D_i}\to \bigoplus_{|I|=2}\reg_{\hat D_I}\otimes \alt_I\to \bigoplus_{|I|=3}\reg_{\hat D_I}\otimes \alt_I \dots \to \reg_{\hat D_{[m]}}\otimes \alt_{[m]}\to 0\,;\]
where $\hat D_I=D_I\cap(A\times N_{m-1}A)$. Following the computations in \cite[Section 6]{Mea} one sees that the truncated version $\hat G_m:=\hat G:=\FM_{\sigma^{\le 1}\cK}$ again satisfies $\hat G\circ \hat G^R\cong \hat F^R\circ\hat F$. 
Hence $\hat G$ is again a $\P^{m-2}$ functor for $m\ge 3$ and a $\P^1$-functor for $m=2$. Now, we can prove analogously to above that $\hat T_{m-2,2}^L\circ \hat F_m=\hat G_m$. The key step is again to compute that 
\begin{align}\label{hatHLG}
 \hat H_{m-2,2}^L\circ \hat G_m\cong \hat \pr_X\circ \triv[-1]\quad\text{where}\quad \hat \pr_X=(M_{m-2,2}A\hookrightarrow A\times A^{m-2}\xrightarrow{\pr_X}A)\,.
\end{align}
\subsection{Braids on hyperk\"ahler fourfolds}\label{braidauto}
We say that two elements $a,b$ of a group satisfy the \textit{braid relation} if $aba=bab$. 
Two twists $T_E,T_F$ along spherical objects satisfy the braid relation if $\Hom^*(E,F)=\C[n]$ for some $n\in \Z$; see \cite{ST}.
There is the following straight forward generalisation which gives a criterion for twists along spherical functors to satisfy the braid relation; compare \cite[Theorem 1.2]{ALdg}. 
\begin{prop}\label{braidcrit}
 Let $G=\FM_\cG, H=\FM_\cH\colon \D^b(M)\to\D^b(N)$ be two spherical functors such that $G^R H\cong \id$ and $\Hom_{\D^b(M\times N)}(\cG,\cH)=\C$. Let $\cF=\cone(\cG\xrightarrow\psi\cH)$ for $0\neq\psi\in \Hom(\cG,\cH)$ and set $F=\FM_\cF$. 
Then $F$ is also a spherical functor and every pair of $T_F,T_G,T_H$ spans $\langle T_F,T_G,T_H\rangle$ and satisfies the braid relation.
\end{prop}
\begin{proof}
Composing the triangle $G G^R\to \id \to T_G$ with $H$ and using that $G^R H\cong \id$ we get the triangle $G\to H\to T_G H$. The map $G\to H$ of this triangle is non-zero. Indeed, otherwise we had $T_G H=H\oplus G[1]$ contradicting the 
fact that all three functors are spherical. Because of $\Hom(\cG,\cH)=\C$ it follows that $F\cong T_G H$.
This shows that $F$ is spherical and by (\ref{twistrel}) there is the relation
\begin{align}\label{eq1}
 T_F=T_G\circ T_H\circ T_G^{-1}
\end{align}
in $\Aut(\D^b(N))$. Taking left adjoints of $G^R H\cong \id$ we get $H^L G\cong \id$. This gives the triangle $T_H^{-1}G\to G\to H$ which shows that $T_H^{-1}G=F[-1]$ and 
\begin{align}\label{eq2}
 T_F=T_H^{-1}\circ T_G\circ T_H\,.
\end{align}
The assertion follows from the equations (\ref{eq1}) and (\ref{eq2}).
\end{proof}
\begin{prop}
\begin{enumerate}
\item Let $X$ be a K3 surface.
Then every pair of $T_{F_2},T_{G_2},T_{0,2}$ spans the subgroup $\langle T_{F_2},T_{G_2},T_{0,2}\rangle \subset \Aut(\D^b(X^{[2]}))$ and satisfies the braid relation.
\item Let $A$ be an abelian surface.
Then every pair of $T_{\hat F_3},T_{\hat G_3},\hat T_{1,2}$ spans the subgroup $\langle T_{\hat F_3},T_{\hat G_3},\hat T_{1,2}\rangle \subset \Aut(\D^b(K_2A))$ and satisfies the braid relation.
\end{enumerate}
\end{prop}
\begin{proof}
 We apply the above proposition to $G=G_2$ and $H=H_{0,2}[-1]$. By Lemma \ref{HLG} we have $H^LG=\id$ and hence also $G^RH=\id$. We have 
 \[\cG=\sigma^{\le 1}\cK=(0\to \reg_{X\times X^2}\to\reg_{D_1}\oplus\reg_{D_2}\to 0)\quad,\quad \cH=\reg_{D_{[2]}}\otimes \alt_2[-1]\,.\]
 Now, $\Hom(\reg_{X\times X^2},\reg_{D_{[2]}}\otimes \alt_2)^{\sym_2}=0$ and 
 \[\Hom(\reg_{D_1}\oplus\reg_{D_2},\reg_{D_{[2]}}\otimes \alt_2)^{\sym_2}\cong \Hom(\reg_{D_1},\reg_{D_{[2]}})=\C\]
 by (\ref{Danlem}). Hence $\Hom_{\D^b_{\sym_2}(X\times X^2)}(\cG,\cH)=\C$. Thus, the assumptions of Proposition \ref{braidcrit} are fulfilled (one either has to apply the BKRH equivalence or make the straight-forward generalisation of Proposition \ref{braidcrit} to equivariant spherical FM transforms).

In the second case we set $G=\hat G_3$ and $H=\hat H_{1,2}[-1]$. Note that $M_{1,2}A\cong A$ so that (\ref{hatHLG}) gives $H^LG=\id$.
The verification that $\Hom(\cG,\cH)=\C$ is similar to the first case.
\end{proof}
\subsection{Semi-orthogonal decompositions in the curve cases}
Let $C$ be a smooth curve.
As stated in Corollary B, it follows from Proposition A that there is the semi-orthogonal decomposition
\begin{align}\label{sod}\D^b_{\sym_m}(C^m)=\langle \cA_{0,m},\cA_{1,m-1},\dots,\cA_{r,m-r},\cB\rangle\quad,\quad \cA_{\ell,m-\ell}=H_{\ell,m-\ell}(\D^b_{\sym_\ell}(C\times C^\ell))\,.\end{align}
Since the symmetric product $S^mC$ is smooth, the pull-back along $\pi\colon C^m\to S^mC$ maps to the bounded derived category, i.e.\ $\pi^* \triv\colon \D^b(S^mC)\to \D^b_{\sym_m}(C^m)$. 
Note that this holds exclusively for curves. Since $(\pi_*\reg_{C^m})^{\sym_m}=0$, it follows by projection formula that $(\_)^{\sym_m} \pi_* \pi^* \triv\cong \id$ which means that $\pi^* \triv$ is fully faithful. For $I\subset [m]$ with $|I|\ge 2$ we have 
$(\_)^{\sym_I}\MM_{\alt_I}\delta_I^*\Res \pi^* \triv=0$ .
Hence, $H_{\ell,m-\ell}^L\pi^*\triv=0$ for all $\ell\ge 2$ which shows that $\pi^*\triv(\D^b(S^mC))\subset \cB$.

Also note that the semi-orthogonal decomposition (\ref{sod}) can be further refined by considering for $\ell\le r$ the semi-orthogonal decompositions of $\D^b_{\sym_\ell}(C\times C^\ell)$  induced by those of $\D^b_{\sym_\ell}(C^\ell)$ which are themselves of the form (\ref{sod}).

\subsection{Induced autoequivalences in the curve cases}\label{curveauto}
\begin{lemma}\label{altcomp}
 Let $X$ be a smooth variety of odd dimension. Then \[\bar S_X^{-(n-1)}H_{\ell,n}^L\MM_{\alt_{n+\ell}}\cong M_{\alt_\ell}H_{\ell,n}^R\quad\text{for $n\ge 2$}\,.\] 
\end{lemma}
\begin{proof}
There is the following crucial difference between the odd and the even dimensional case. If $X$ is even dimensional, the $\sym_m$-equivariant canonical bundle of $X^m$ equals $\omega_{X^m}\cong \omega_X^{\boxtimes m}$ where the linearisation is the one 
acting by permuting the factors, while in the odd dimensional case the linearisation is given by $\omega_{X^m}\cong \omega_X^{\boxtimes m}\otimes\alt_m$. Indeed, on the fibre $\omega_{X^m}(x)$ the stabiliser of $x\in X^m$ acts by permuting blocks of 
length $d=\dim X$ in the wedge product.

With this in mind the proof is the same as the proof that $\bar S_X^{-(n-1)}H_{\ell,n}^L\cong H_{\ell,n}^R$ for $X$ even dimensional; see Section \ref{cond3}.
\end{proof}
Let now $X=C$ be a smooth curve. Then $H_{\ell,n}\colon \D^b_{\sym_\ell}(C\times C^\ell)\to \D^b_{\sym_{n+\ell}}(X^{n+\ell})$ is fully faithful for $n>\min\{\ell,1\}$; see Proposition A. 
Let $m=n+\ell$ and $\mathfrak A_{m}\subset \sym_{m}$ be the alternating group. The functor $\Res\colon \D^b_{\sym_m}(X^m)\to\D^b_{\mathfrak A_m}(X^m)$ is spherical with cotwist $M_\alt$ and twist $\tau^*[1]$ where $\tau$ is any element of $\sym_m\setminus \mathfrak A_m$.
This follows by the Lemmas \ref{infspherical} and \ref{Inftwist} together with the fact that a functor $F$ is spherical if and only if $F^R$ is spherical with the roles of the twist and cotwists exchanged; see \cite[Theorem 1.1]{ALdg}.
It follows by  \cite[Theorem 4.13]{HLS} 
that the composition $\Res_{\sym_{n+\ell}}^{\Alt_{n+\ell}}\circ H_{\ell,n}$ is again a spherical functor. 
For $\tau\in\sym_{m}\setminus\Alt_{m}$ we have $\tau^*\circ \Res\circ H_{\ell,n}\cong \Res\circ H_{\ell,n}$. Thus, the spherical twist $\tilde T_{\ell,n}:=T_{\Res\circ H_{\ell,n}}\in \Aut(\D^b_{\Alt_{m}}(X^{m}))$ is 
equivariant, i.e.\ $\tau^*\circ T_{\ell,n}\cong T_{\ell,n}\circ \tau^*$, by (\ref{twistrel}). Hence, there is an induced autoequivalence $T_{\ell,n}\in \Aut(\D^b_{\sym_m}(X^{m}))$; see \cite[Theorem 6]{Plo}. 
By the same arguments, for $E$ an elliptic curve  the fully faithful functors $\hat H_{\ell,n}\colon\D^b_{\sym_\ell}(M_{\ell,n}E)\to \D^b_{\sym_{n+\ell}}(N_{n+\ell-1}E)$ induce autoequivalences of $\D^b_{\Alt_{n+\ell}}(N_{n+\ell-1}E)$  
and $\D^b_{\sym_{n+\ell}}(N_{n+\ell-1}E)$. 
  
One can also construct $T_{\ell,n}\in \Aut(\D^b_{\sym_{m}}(X^{m}))$ directly as the cone \[T_{\ell,n}=\cone(H_{\ell,n}H_{\ell,n}^R\oplus \MM_\alt H_{\ell,n}H_{\ell,n}^R \MM_\alt\xrightarrow{c}\id)\] where $c$ is the direct sum of the counit maps. 
It follows by Lemma \ref{altcomp} that 
\[
 T_{\ell,n}H_{\ell,n}\cong M_{\alt_{n+\ell}}H_{\ell,n}M_{\alt_\ell}\bar S_C^{-(n-1)}[1]\quad ,\quad T_{\ell,n}M_{\alt_{n+\ell}} H_{\ell,n}\cong H_{\ell,n}M_{\alt_\ell}\bar S_C^{-(n-1)}[1]\,,
\]
and $T_{\ell,n}(B)=B$ for $B\in \ker H_{\ell,n}^R\cap \ker (H_{\ell,n}^R\MM_{\alt})$. Thus 
\begin{align}\label{curvetwistvalues}
 T_{\ell,n}(\bar C(x))=\bar\C(x)[-(n-2)]\,\,\text{ for $x\in X^{n+\ell}_{\nu(n)}$}\quad, \quad 
T_{\ell,n}(\bar C(x))=\bar\C(x)\,\,\text{ for $x\in \bar X^{n+\ell}_{\nu(n)}$}
\end{align}
similarly to Lemma \ref{skyscrapervalue}.  
\subsection{Some conjectures}\label{conjectures}
As we see from the list at the end of Section \ref{Hilbauto}, there are surfaces $X$ for which the subgroup 
\[\Aut^{st+H}(\D^b(X^{[m]}))=\langle \Aut^{st}(\D^b(X^{[m]})),P_{0,m},\dots,P_{r,m-r} \rangle\subset\Aut(\D^b(X^{[m]}))\]
generated by standard autoequivalences and twists along the Nakajima $\P$-functors is far smaller than the full group $\Aut(\D^b(X^{[m]}))$. However, in the case that the canonical bundle of the surface is ample or anti-ample there are no non-standard 
autoequivalences coming from the surface (see \cite{BO}) and we expect the following to hold.
\begin{conj}
 Let $X$ be a smooth projective surface with $\omega_X$ ample or anti-ample. Then for $m=2,3$ we have $\Aut^{st+H}(\D^b(X^{[m]}))=\Aut(\D^b(X^{[m]}))$, i.e.\
 \[\Aut(\D^b(X^{[2]}))=\langle \Aut^{st}(\D^b(X^{[2]})),T_{0,2}\rangle \quad,\quad \Aut(\D^b(X^{[3]}))=\langle \Aut^{st}(\D^b(X^{[3]})),P_{0,3},T_{1,2}\rangle\,.\]
\end{conj}
For $m\ge 4$, we expect there to be further autoequivalences $p_{r+1,m-r-1},\dots, p_{m-3,3},t_{m-2,2}$ with the same behaviour as the one of the twists along the Nakajima $\P$-functors that is described in Lemma \ref{skyscrapervalue}. 
That means that we should have
\begin{align*}
 p_{\ell,m-\ell}(\bar \C(x))&\cong\begin{cases}
                              \bar \C(x)[-2(m-\ell-1)]\,\,&\text{for $x\in X^m_{\nu(\m-\ell)}$,}\\
                             \bar \C(x)\,\,&\text{for $x\in \bar X^m_{\nu(\m-\ell)}$,}\\
                             \end{cases}
\\ 
 t_{m-2,2}(\bar \C(x))&\cong\begin{cases}
                              \bar \C(x)[-1]\,\,&\text{for $x\in X^m_{\nu(2)}$,}\\
                             \bar \C(x)\,\,&\text{for $x\in \bar X^m_{\nu(2)}$.}\\
                             \end{cases}
\end{align*}
The main evidence is that there is in fact always an autoequivalence one can consider as $t_{m-2,2}$, namely the composition $\Phi^{-1}\MM_\alt \Phi \MM_{\reg(D/2)}$ where $D$ denotes the boundary divisor of the Hilbert scheme. 
Indeed, for $x\in X^m_{\nu(2)}$ we have $\Phi^{-1}(\bar \C(x))\cong  \reg_{\mu^{-1}(\pi(x))}(-1)$ and $\Phi^{-1}(\bar \C(x)\otimes \alt)\cong  \reg_{\mu^{-1}(\pi(x))}(-2)[1]$ where $\pi\colon X^m\to S^mX$ is the quotient morphism and 
$\mu\colon X^{[m]}\to S^mX$ is the Hilbert--Chow morphism. This is shown in the case $m=2$ in \cite[Proposition 4.4 \& Remark 4.6]{Kru3} and the proof for general $m$ is the same. Furthermore, $\reg(D/2)_{|\mu^{-1}(\pi(x))}\cong \reg_{\mu^{-1}(\pi(x))}(-1)$. 
Thus, $\Phi^{-1}\MM_\alt \Phi \MM_{\reg(D/2)}(\Phi^{-1}(\bar \C(x))\cong \Phi^{-1}(\bar \C(x))[-1]$.  

As explained in Section \ref{curveauto} there are also autoequivalences of $\D^b_{\sym_m}(C^m)$ that act as a shift on certain strata $C^m_{\nu(k)}$ and as the identity on $\bar C^m_{\nu(k)}$. Thus, one might also guess that this kind of 
autoequivalences exist in $\D^b_{\sym_m}(X^m)$ for $X$ a smooth variety of arbitrary dimension, even though for $\dim X\ge 3$ the functors $H_{\ell,n}$ are far from being fully faithful or $\P$-functors; compare Remark \ref{higherdiminva}.

\bibliographystyle{alpha}
\bibliography{references}

\end{document}